\documentclass[12pt,final]{amsart}
\usepackage[utf8]{inputenc}

\usepackage[style=alphabetic,url=false, isbn=false, doi=false,maxbibnames=99]{biblatex}
\AtEveryBibitem{
    \clearfield{urlyear}
    \clearfield{urlmonth}
}
\usepackage{xspace,amssymb,epsfig,calligra, enumerate, enumitem,euscript,eufrak,enumerate,amsmath,nicefrac,mathdots}
\DeclareMathAlphabet{\mathbbm}{U}{bbm}{m}{n}
\usepackage{tikz,etex,xspace,afterpage}
 \usepackage{epstopdf,ifpdf,todonotes}
\usepackage{accents}
\usepackage{stackengine}

\stackMath
\usepackage{nameref,zref-xr} \zxrsetup{toltxlabel} \zexternaldocument*[SD-]{./after-08-24} 
\usepackage{showkeys}
\usepackage[margin=3cm,footskip=25pt,headheight=20pt]{geometry}

\usepackage{hyperref}

\usepackage{tabularx}
\usepackage{anindex}
\anindexsetup{
    heading,
    lines,
}
\usepackage{fancyhdr,mathrsfs}

\usepackage[capitalize]{cleveref}
  \usepackage{pdfsync}

\begin{document}
\newtheoremstyle{all}{11pt}{11pt}{\slshape}{}{\bfseries}{}{.5em}{}

\theoremstyle{all}
\newtheorem{itheorem}{Theorem}
\newtheorem{theorem}{Theorem}[section]
\newtheorem{proposition}[theorem]{Proposition}
\newtheorem{corollary}[theorem]{Corollary}
\newtheorem{lemma}[theorem]{Lemma}
\newtheorem{assumption}[theorem]{Assumption}
\newtheorem{definition}[theorem]{Definition}
\newtheorem{ques}[theorem]{Question}
\newtheorem{conjecture}[theorem]{Conjecture}
\newtheorem{physics}[theorem]{Physics Motivation}

\theoremstyle{remark}
\newtheorem{remark}[theorem]{Remark}
\newtheorem{examplex}{Example}
\newenvironment{example}
  {\pushQED{\qed}\renewcommand{\qedsymbol}{$\clubsuit$}\examplex}
  {\popQED\endexamplex }
\renewcommand{\theexamplex}{{\arabic{section}.\roman{examplex}}}
\newcommand{\nc}{\newcommand}
\newcommand{\renc}{\renewcommand}
\numberwithin{equation}{section}
\renc{\theequation}{\arabic{section}.\arabic{equation}}

\newcounter{subeqn}
\renewcommand{\thesubeqn}{\theequation\alph{subeqn}}
\newcommand{\subeqn}{\refstepcounter{subeqn}\tag{\thesubeqn}}\makeatletter
\@addtoreset{subeqn}{equation}
\newcommand{\newseq}{\refstepcounter{equation}}
  \nc{\kac}{\kappa^C}
\nc{\alg}{T}
\nc{\salg}{W}
\nc{\zero}{o}
\nc{\weights}{\iota}
\nc{\psalg}{\mathscr{W}}

\nc{\Isalg}{\mathscr{H}}
\nc{\Lco}{L_{\la}}
\nc{\qD}{q^{\nicefrac 1D}}
\nc{\ocL}{M_{\la}}
\nc{\excise}[1]{}
\nc{\Dbe}{D^{\uparrow}}
\nc{\Dfg}{D^{\mathsf{fg}}}
\nc{\frob}{\mathsf{f}}

\nc{\op}{\operatorname{op}}
\nc{\Sym}{\operatorname{Sym}}
\nc{\Symt}{S}
\nc{\tr}{\operatorname{tr}}
\newcommand{\Mirkovic}{Mirkovi\'c\xspace}
\nc{\tla}{\mathsf{t}_\la}
\nc{\llrr}{\langle\la,\rho\rangle}
\nc{\lllr}{\langle\la,\la\rangle}
\nc{\K}{\mathbbm{k}}
\nc{\Stosic}{Sto{\v{s}}i{\'c}\xspace}
\nc{\cd}{\mathcal{D}}
\nc{\cT}{\mathcal{T}}
\nc{\vd}{\mathbb{D}}
\nc{\Fp}{{\mathbb{F}_p}}
\nc{\lift}{\gamma}
\nc{\cox}{h}
\nc{\Aut}{\operatorname{Aut}}
\nc{\R}{\mathbb{R}}
\renc{\wr}{\operatorname{wr}}
  \nc{\Lam}[3]{\La^{#1}_{#2,#3}}
  \nc{\Lab}[2]{\La^{#1}_{#2}}
  \nc{\Lamvwy}{\Lam\Bv\Bw\By}
  \nc{\Labwv}{\Lab\Bw\Bv}
  \nc{\nak}[3]{\mathcal{N}(#1,#2,#3)}
  \nc{\hw}{highest weight\xspace}
  \nc{\al}{\alpha}
  \nc{\gK}{K}
  \nc{\gk}{\mathfrak{k}}
  
\newcommand{\LLoc}{\mathbb{L}\!\operatorname{Loc}}
\newcommand{\Rsecs}{\mathbb{R}\Gamma_\bS}

\newlength{\dhatheight}
\newcommand{\doublehat}[1]{\settoheight{\dhatheight}{\ensuremath{\hat{#1}}}\addtolength{\dhatheight}{-0.35ex}\hat{\vphantom{\rule{1pt}{\dhatheight}}\smash{\hat{#1}}}}

\newcommand{\dgmod}{\operatorname{-dg-mod}}
  \nc{\be}{\beta}
  \nc{\bM}{\mathbf{m}}
  \nc{\Bu}{\mathbf{u}}

  \nc{\bkh}{\backslash}
  \nc{\Bi}{\mathbf{i}}
  \nc{\Bm}{\mathbf{m}}
  \nc{\Bj}{\mathbf{j}}
 \nc{\Bk}{\mathbf{k}}
  \nc{\Bs}{\mathbf{s}}
\newcommand{\bS}{\mathbb{S}}
\newcommand{\bT}{\mathbb{T}}
\newcommand{\bt}{\mathbbm{t}}

\nc{\hatD}{\widehat{\Delta}}
\nc{\bd}{\mathbf{d}}
\nc{\D}{\mathcal{D}}
\nc{\mmod}{\operatorname{-mod}}  
\nc{\AS}{\operatorname{AS}}
\newcommand{\red}{\mathfrak{r}}

\nc{\RAA}{R^\A_A}
  \nc{\Bv}{\mathbf{v}}
  \nc{\Bw}{\mathbf{w}}
\nc{\Id}{\operatorname{Id}}
\nc{\Cth}{S_h}
\nc{\Cft}{S_1}
\def\MHM{{\operatorname{MHM}}}

\newcommand{\cM}{\mathcal{M}}
\newcommand{\cD}{\mathcal{D}}
\newcommand{\LCP}{\operatorname{LCP}}
  \nc{\By}{\mathbf{y}}
\nc{\eE}{\EuScript{E}}
  \nc{\Bz}{\mathbf{z}}
  \nc{\coker}{\mathrm{coker}\,}
  \nc{\C}{\mathbb{C}}
\nc{\ab}{{\operatorname{ab}}}
\nc{\wall}{\mathbbm{w}}
  \nc{\ch}{\mathrm{ch}}
  \nc{\de}{\delta}
  \nc{\ep}{\epsilon}
  \nc{\Rep}[2]{\mathsf{Rep}_{#1}^{#2}}
  \nc{\Ev}[2]{E_{#1}^{#2}}
  \nc{\fr}[1]{\mathfrak{#1}}
  \nc{\fp}{\fr p}
  \nc{\fq}{\fr q}
  \nc{\fl}{\fr l}
  \nc{\fgl}{\fr{gl}}
  \nc{\Fr}{\operatorname{Fr}}
\nc{\rad}{\operatorname{rad}}
\nc{\ind}{\operatorname{ind}}
  \nc{\GL}{\mathrm{GL}}
\newcommand{\arxiv}[1]{\href{http://arxiv.org/abs/#1}{\tt arXiv:\nolinkurl{#1}}}
  \nc{\Hom}{\mathrm{Hom}}
  \nc{\im}{\mathrm{im}\,}
  \nc{\La}{\Lambda}
  \nc{\la}{\lambda}
  \nc{\mult}{b^{\mu}_{\la_0}\!}
  \nc{\mc}[1]{\mathcal{#1}}
  \nc{\om}{\omega}
\nc{\gl}{\mathfrak{gl}}
  \nc{\cF}{\mathcal{F}}
\nc{\cC}{\mathcal{C}}
  \nc{\Mor}{\mathsf{Mor}}
  \nc{\HOM}{\operatorname{HOM}}

  \nc{\sHom}{\mathscr{H}\text{\kern -3pt {\calligra\large om}}\,}
  \nc{\Ob}{\mathsf{Ob}}
  \nc{\Vect}{\operatorname{-Vect}}
\nc{\gVect}{\mathsf{gVect}}
  \nc{\modu}{\mathsf{-mod}}
\nc{\pmodu}{\mathsf{-pmod}}
  \nc{\qvw}[1]{\La(#1 \Bv,\Bw)}
  \nc{\van}[1]{\nu_{#1}}
  \nc{\Rperp}{R^\vee(X_0)^{\perp}}
  \nc{\si}{\sigma}
\nc{\sgns}{{\boldsymbol{\sigma}}}
  \nc{\croot}[1]{\al^\vee_{#1}}
\nc{\di}{\mathbf{d}}
  \nc{\SL}[1]{\mathrm{SL}_{#1}}
  
  \nc{\slhat}[1]{\mathfrak{\widehat{sl}}_{#1}}
  \nc{\sllhat}{\slhat{\ell}}
  \nc{\slnhat}{\slhat{n}}
    \nc{\slehat}{\slhat{e}}
   
 \nc{\sle}{\mathfrak{sl}_e}
    \nc{\Th}{\theta}
  \nc{\vp}{\varphi}
  \nc{\wt}{\mathrm{wt}}
\nc{\te}{\tilde{e}}
\nc{\tf}{\tilde{f}}
\nc{\hwo}{\mathbb{V}}
\nc{\soc}{\operatorname{soc}}
\nc{\cosoc}{\operatorname{cosoc}}
 \nc{\Q}{\mathbb{Q}}
\nc{\LPC}{\mathsf{LPC}}
  \nc{\Z}{\mathbb{Z}}
  \nc{\Znn}{\Z_{\geq 0}}
  \nc{\ver}{\EuScript{V}}
  \nc{\Res}[2]{\operatorname{Res}^{#1}_{#2}}
  \nc{\edge}{\EuScript{E}}
  \nc{\Spec}{\operatorname{Spec}}
  \nc{\tie}{\EuScript{T}}
  \nc{\ml}[1]{\mathbb{D}^{#1}}
  \nc{\fQ}{\mathfrak{Q}}
        \nc{\fg}{\mathfrak{g}}
        \nc{\ft}{\mathfrak{t}}
        \nc{\fm}{\mathfrak{m}}
  \nc{\Uq}{U_q(\fg)}
        \nc{\bom}{\boldsymbol{\omega}}
\nc{\bla}{{\underline{\boldsymbol{\la}}}}
\nc{\bmu}{{\underline{\boldsymbol{\mu}}}}
\nc{\bal}{{\boldsymbol{\al}}}
\nc{\bet}{{\boldsymbol{\eta}}}
\nc{\rola}{X}
\nc{\wela}{Y}
\nc{\fM}{\mathfrak{M}}
\nc{\tfM}{\mathfrak{\tilde M}}
\nc{\fX}{\mathfrak{X}}
\nc{\fH}{\mathfrak{H}}
\nc{\fE}{\mathfrak{E}}
\nc{\fF}{\mathfrak{F}}
\nc{\fI}{\mathfrak{I}}
\nc{\qui}[2]{\fM_{#1}^{#2}}
\nc{\cL}{\mathcal{L}}
\nc{\ca}[2]{\fQ_{#1}^{#2}}
\nc{\cat}{\mathcal{V}}
\nc{\cata}{\mathfrak{V}}
\nc{\catf}{\mathscr{V}}
\nc{\hl}{\mathcal{X}}
\nc{\hld}{\EuScript{X}}
\nc{\hldbK}{\EuScript{X}^{\bla}_{\bar{\mathbb{K}}}}
\nc{\Iwahori}{\mathrm{Iwa}}
\nc{\hE}{\mathfrak{E}^{(1)}}
\nc{\Eh}{\mathfrak{E}^{(2)}}
\nc{\hF}{\mathfrak{F}^{(1)}}
\nc{\Fh}{\mathfrak{F}^{(2)}}

\nc{\pil}{{\boldsymbol{\pi}}^L}
\nc{\pir}{{\boldsymbol{\pi}}^R}
\nc{\cO}{\mathcal{O}}
\nc{\Ko}{\text{\Denarius}}
\nc{\Ei}{\fE_i}
\nc{\Fi}{\fF_i}
\nc{\fil}{\mathcal{H}}
\nc{\brr}[2]{\beta^R_{#1,#2}}
\nc{\brl}[2]{\beta^L_{#1,#2}}
\nc{\so}[2]{\EuScript{Q}^{#1}_{#2}}
\nc{\EW}{\mathbf{W}}
\nc{\rma}[2]{\mathbf{R}_{#1,#2}}
\nc{\Dif}{\EuScript{D}}\nc{\MDif}{\EuScript{E}}
\renc{\mod}{\mathsf{mod}}
\nc{\modg}{\mathsf{mod}^g}
\nc{\fmod}{\mathsf{mod}^{fd}}
\nc{\id}{\operatorname{id}}
\nc{\compat}{\EuScript{K}}
\nc{\DR}{\mathbf{DR}}
\nc{\End}{\operatorname{End}}
\nc{\Fun}{\operatorname{Fun}}
\nc{\Ext}{\operatorname{Ext}}
\nc{\Coh}{\operatorname{Coh}}
\nc{\tw}{\tau}
\nc{\second}{\tau}
\nc{\A}{\EuScript{A}}
\nc{\Loc}{\mathsf{Loc}}
\nc{\eF}{\EuScript{F}}
\nc{\LAA}{\Loc^{\A}_{A}}
\nc{\perv}{\mathsf{Perv}}
\nc{\gfq}[2]{B_{#1}^{#2}}
\nc{\qgf}[1]{A_{#1}}
\nc{\qgr}{\qgf\rho}
\nc{\tqgf}{\tilde A}
\nc{\Tr}{\operatorname{Tr}}
\nc{\Tor}{\operatorname{Tor}}
\nc{\cQ}{\mathcal{Q}}
\nc{\st}[1]{\Delta(#1)}
\nc{\cst}[1]{\nabla(#1)}
\nc{\ei}{\mathbf{e}_i}
\nc{\Be}{\mathbf{e}}
\nc{\Hck}{\mathfrak{H}}
\renc{\P}{\mathbb{P}}
\nc{\bbB}{\mathbb{B}}
\nc{\ssy}{\mathsf{y}}
\nc{\cI}{\mathcal{I}}
\nc{\cG}{\mathcal{G}}
\nc{\cH}{\mathcal{H}}
\nc{\coe}{\mathfrak{K}}
\nc{\pr}{\operatorname{pr}}
\nc{\bra}{\mathfrak{B}}
\nc{\rcl}{\rho^\vee(\la)}
\nc{\tU}{\mathcal{U}}
\nc{\dU}{{\stackon[8pt]{\tU}{\cdot}}}
\nc{\dT}{{\stackon[8pt]{\cT}{\cdot}}}
\nc{\BFN}{\EuScript{R}}

\nc{\RHom}{\mathrm{RHom}}
\nc{\tcO}{\tilde{\cO}}
\nc{\Yon}{\mathscr{Y}}
\nc{\sI}{{\mathsf{I}}}
\nc{\sptc}{X_*(T)_1}
\nc{\spt}{\ft_1}
\nc{\Bpsi}{u}
\nc{\acham}{\eta}
\nc{\hyper}{\mathsf{H}}
\nc{\AF}{\EuScript{Fl}}
\nc{\VB}{\EuScript{X}}
\nc{\OHiggs}{\cO_{\operatorname{Higgs}}}
\nc{\OCoulomb}{\cO_{\operatorname{Coulomb}}}
\nc{\tOHiggs}{\tilde\cO_{\operatorname{Higgs}}}
\nc{\tOCoulomb}{\tilde\cO_{\operatorname{Coulomb}}}
\nc{\indx}{\mathcal{I}}
\nc{\redu}{K}
\nc{\Ba}{\mathbf{a}}
\nc{\Bb}{\mathbf{b}}
\nc{\Bc}{\mathbf{c}}
\nc{\Lotimes}{\overset{L}{\otimes}}
\nc{\AC}{C}
\nc{\rAC}{rC}\nc{\defr}{\operatorname{def}}

\nc{\rACp}{\mathsf{C}}
\nc{\ideal}{\mathscr{I}}
\nc{\ACs}{\mathscr{C}}
\nc{\Stein}{\mathscr{X}}
\nc{\pStein}{p\mathscr{X}}
\nc{\pSteinK}{\overline{\mathscr{X}}}
\nc{\No}{H}
\nc{\To}{Q}
\nc{\tNo}{\tilde{H}}
\nc{\tTo}{\tilde{Q}}

\nc{\gaugeG}{G}
\nc{\weylW}{W}
\nc{\matterV}{V}
\nc{\quiver}{\Gamma}
\nc{\Coulomb}{\fM}
\nc{\Rring}{\mathring{R}}
\nc{\hRring}{\mathring{R}^h}
\nc{\ringS}{S}
\nc{\efA}{\EuScript{A}}
\nc{\groupK}{K}
\newcommand{\MQ}{\fM_{\To}}
\nc{\yw}{y}
\nc{\Twist}{T}
\nc{\scrBhat}{\mathscr{\widehat{B}}}
\nc{\sfB}{\ensuremath{\mathsf{B}}}

\nc{\sfBhat}{\ensuremath{\widehat{\mathsf{B}}}}

\nc{\sfA}{\ensuremath{\mathsf{A}}}

\nc{\sfAhat}{\ensuremath{\widehat{\mathsf{A}}}}

\nc{\scrB}{\mathscr{B}}
\nc{\scrT}{\mathscr{T}}
\nc{\Asph}{\ensuremath{\EuScript{A}^{\operatorname{sph}}}}
\nc{\What}{\widehat{\weylW}}
\nc{\tM}{\tilde{\fM}}
\nc{\flav}{\phi}
\nc{\tF}{\tilde{F}}
\newcommand{\cOg}{\mathcal{O}_{\!\operatorname{g}}}
\newcommand{\tcOg}{\mathcal{\tilde O}_{\!\operatorname{g}}}
\newcommand{\dOg}{D_{\cOg}}
\newcommand{\preO}{p\cOg}
\newcommand{\dpreO}{D_{p\cOg}}
\nc{\vertex}{\EuScript{V}(\Gamma)}
\nc{\Wei}{\EuScript{W}}
\setcounter{tocdepth}{2}
\newcommand{\thetitle}{Coherent sheaves and quantum Coulomb branches I:\\tilting bundles from integrable systems}
\newcommand{\theshorttitle}{Coherent sheaves and quantum Coulomb branches I}
\renc{\theitheorem}{\Alph{itheorem}}

\excise{
\newenvironment{block}
\newenvironment{frame}
\newenvironment{tikzpicture}
\newenvironment{equation*}
}

\baselineskip=1.1\baselineskip

 \usetikzlibrary{decorations.pathreplacing,backgrounds,decorations.markings,shapes.geometric,decorations.pathmorphing}
\tikzset{wei/.style={draw=red,double=red!40!white,double distance=1.5pt,thin}}
\tikzset{awei/.style={draw=blue,double=blue!40!white,double distance=1.5pt,thin}}
\tikzset{bdot/.style={fill,circle,color=blue,inner sep=3pt,outer sep=0}}
    \tikzset{ weyl/.style={decorate, decoration={snake}, draw=black!50!green, very thick}}
\tikzset{fringe/.style={gray,postaction={decoration=border,decorate,draw,gray, segment length=4pt,thick}}}
\tikzset{old/.style={gray,thin}}
\tikzset{dir/.style={postaction={decorate,decoration={markings, mark=at position .8 with {\arrow[scale=1.3]{>}}}}}}
\tikzset{rdir/.style={postaction={decorate,decoration={markings, mark=at position .8 with {\arrow[scale=1.3]{<}}}}}}
\tikzset{edir/.style={postaction={decorate,decoration={markings, mark=at position .2 with {\arrow[scale=1.3]{<}}}}}}\begin{center}
\noindent {\large  \bf \thetitle}
\medskip

\noindent {\sc Ben Webster}\footnote{Supported by the NSF under Grant DMS-1151473 and the Alfred P. Sloan Foundation. This research was supported in part by Perimeter Institute for Theoretical Physics. Research at Perimeter Institute is supported in part by the Government of Canada through the Department of Innovation, Science and Economic Development Canada and by the Province of Ontario through the Ministry of Colleges and Universities.}\\  
Department of Pure Mathematics, University of Waterloo \& \\
 Perimeter Institute for Theoretical Physics\\
Waterloo, ON\\
Email: {\tt ben.webster@uwaterloo.ca}
\end{center}
\bigskip
{\small
\begin{quote}
\noindent {\em Abstract.}
In this paper, we consider how the approach of Bezrukavnikov and
Kaledin to understanding the categories of coherent sheaves on
symplectic resolutions can be applied to the Coulomb branches
introduced by Braverman, Finkelberg and Nakajima.  In particular, we
construct tilting generators on resolved Coulomb branches and give
explicit quiver presentations of categories of coherent sheaves on
these varieties, with the wall-crossing functors described by natural
bimodules.
\end{quote}
}

\section{Introduction}
\label{sec:introduction}

Let $V$ be a complex vector space, and let $G$ be a connected reductive
algebraic group with a fixed faithful linear action on $V$. 
Attached to these data, we have a symplectic variety $\Coulomb$ called the {\bf
  Coulomb branch}, defined by Braverman, Finkelberg and Nakajima \cite{BFN},
based on proposals in the physics literature. Many interesting
varieties appear in this way, including quiver varieties in finite and
affine type A, hypertoric varieties, and slices between Schubert cells in affine Grassmannians.  This construction generalizes to give a construction of a number of partial resolutions $\tM$ of $\fM$ \cite{BFNline}; we will
call one of these a {\bf BFN resolution} if it is a resolution of singularities.

Bezrukavnikov and Kaledin have developed a general theory of quantizations
of algebraic varieties in arbitrary characteristic
\cite{BK04a,BKpos} and Kaledin showed that this theory can
be applied to construct tilting generators on symplectic resolutions
of singularities \cite{KalDEQ}.  Kaledin's theory is very powerful
but not very concrete from the perspective of a representation theorist.
In particular, this work shows that the category of coherent sheaves on a conic
symplectic resolution is derived equivalent to the category of modules
over an algebra $A$ (actually to many different algebras, one for each choice of a quantization parameter), defined in \eqref{eq:A-def}. In any particular case, this algebra is quite
challenging to calculate.  Our goal in this paper is to develop
Kaledin's theory as explicitly as possible in the case of Coulomb
branches and, in particular, to describe this algebra $A$. 
 We will show:
\begin{itheorem}\label{th:main}
Any BFN Coulomb branch with a BFN resolution has an explicit
combinatorially presented noncommutative resolution of singularities
$A$.  The category $D^b(A\mmod)$ is equivalent to the derived category of coherent sheaves on any BFN resolution through an explicit tilting generator.
\end{itheorem}
For readers who prefer to live in characteristic 0 as opposed to characteristic
$p$, we should emphasize that the construction of this noncommutative
resolution $A$ and its tilting generator have a construction which is
characteristic free (that is, over $\Z$);  however, we use reduction
to characteristic $p$ and comparison to the Bezrukavnikov-Kaledin
method to confirm Theorem \ref{th:main}.  In fact, our resolution is obtained as the Borel-Moore homology of a variation on the BFN space considered in \cite{BFN}.

Of course, this theorem is only of interest to a reader who knows some
examples of Coulomb branches with BFN resolutions.  The most interesting
come from a quiver gauge theory, that is,
the case which leads to Nakajima quiver varieties as Higgs branches.
Following Nakajima's notation, for a quiver $\Gamma$ with vertex set $\vertex$,
consider dimension vectors $\Bv,\Bw\colon \vertex\to \Z_{\geq 0}$, and
the group and representation
\begin{equation}
G=\prod GL(\C^{v_i})\qquad V=\Big(\bigoplus_{i\to j}
\Hom(\C^{v_i},\C^{v_j})\Big)\bigoplus \Big(\bigoplus_{i\in {\vertex}}
  \Hom(\C^{w_i},\C^{v_i}) \Big),\label{eq:quiver-gauge}
\end{equation}
  with the obvious induced action.  In this
case, the algebra $A$ is a version of a KLR algebra drawn on a
cylinder, as we will show in the second part of this paper \cite{WebcohII}.  Examples include:
\begin{enumerate}
\item When the underlying quiver is of type A, then the resulting
  Coulomb branch is the Slodowy slice to one nilpotent orbit inside
  another in a type A nilcone.  The BFN resolutions in this case are
exactly those that arise from taking the preimage under a resolution of the larger orbit closure by $T^*(SL_n/P)$ for a parabolic $P$. 
\item When the underlying quiver is of type D or E, the Coulomb branch is isomorphic to an affine Grassmannian slice, as shown in \cite[App. B]{BFNplus}.
\item When the underlying quiver is a loop, the resulting Coulomb branch is the $v$-fold symmetric product of the singular surface $S=\C^2/(\Z/w\Z)$.  In particular, one of the BFN resolutions we obtain is the Hilbert scheme of $v$ points on the crepant resolution $\tilde S$.  
\item When the underlying graph is an $n$-cycle, we obtain a Nakajima quiver variety (or more generally a bow variety) for a cycle of size $w=w_i$ whose dimension vectors are related to $\Bv,\Bw$ by a version of rank-level duality \cite{nakajimaCherkisBow2017}.  This includes the results (1) and (3) above as special cases.
\end{enumerate}

The algebra $A$,
which appears as endomorphisms of this tilting generator, can be
interpreted in three very interesting ways:
\begin{enumerate}
\item It can be described algebraically  as a finitely generated algebra constructed directly from the combinatorics of the group $G$ and representation $V$.
\item It can also be described as a convolution algebra in the
 extended BFN category of \cite{websterKoszulDuality2019}, with adjusted flavor and $h=0$ (we
  call these ``$p$-th root conventions.'').
\item It is the endomorphism algebra of a finite sum of line
 defects in the corresponding $\mathcal{N}=4$ supersymmetric $3d$
  gauge theory.
\end{enumerate}
The equivalence of these descriptions is discussed in \cite{websterKoszulDuality2019}:
the equivalence of (1) and (2) is \cite[Th. \ref{SD-thm:BFN-pres}]{websterKoszulDuality2019} and of (2)
and (3) is \cite[Rem. \ref{SD-rmk:QFT}]{websterKoszulDuality2019};  the latter is a motivational
statement rather than a theorem since we are not working with a
precise definition of the category of line defects.

\subsection{Motivation from line defects}
\label{sec:motivation-from-line}

Before getting bogged down in details, let us try to give a general
sketch of our approach.  In this section, we play a little fast and
loose with the existence of certain geometric categories (not to
mention aspects of quantum field theory); we promise to the reader
that no such chicanery will appear in the rest of the paper.

Having fixed the vector space $V$ and gauge
group $G$, we consider the quotient space of $\C((t))$-points
$\mathcal{L}=V((t))/G((t))$.  In the world of derived algebraic geometry, we think
of this as the loop space of the stack quotient $V/G$.  We will
consider the category of $D$-modules on the loop space $\mathcal{L}$.

For the utility of both readers who  wish to read not-entirely-rigorous physics motivation and for those who wish to avoid it, the author will put such motivating paragraphs in ``Physics Motivation'' environments going forward.
\begin{physics}  
As discussed in \cite[\S 1.1]{DGGH}, this category is natural to consider in this context because it should give the category of line defects in the $A$-twist of a 3-dimension $\mathcal{N}=4$ supersymmetric field theory defined by the $\sigma$-model into the $2$-shifted stacky cotangent bundle $T^*[2](V/G)$, turned into a 3-dimensional field theory using the AKSZ formalism.  In particular, the local operators on a point should appear in this category as operators from the trivial line defect to itself.  
\end{physics}

Obviously, there
are many technical issues involved in doing this, and, with apologies
to the reader, we will make no attempt to resolve them.
We simply
ask the reader to accept the existence of this category as a black box
with one basic property:
\begin{enumerate}
\item Given a reasonable map $p\colon \mathcal{M}\to \mathcal{L}$, we
  have the pushforward of the function D-module
  $\mathscr{O}_p=p_*\mathcal{O}_{\mathcal{M}}$.  If we are given a second
  such map $p'\colon \mathcal{M}'\to \mathcal{L}$ then 
  \[\Ext^\bullet(\mathscr{O}_p,
   \mathscr{O}_{p'})=H^{BM}_*(\mathcal{M}\times_{\mathcal{L}}\mathcal{M'})\]
 with composition induced by convolution.  
\end{enumerate}

The definition of Braverman-Finkelberg-Nakajima is that the functions
on the Coulomb branch  of the gauge theory of $(G,V)$ arise when we
take $\mathcal{M}=\mathcal{M}'=V[[t]]/G[[t]]$, the arc space of the
quotient $V/G$; the natural quantization of this ring
appears when we consider this same construction $\C^*$-equivariantly for the
loop action (where $t$ has weight 1).
\begin{physics}
  As mentioned above, the arc space should give the trivial line, so
  the BFN construction gives the local operators in the $A$-twist of
  the $\sigma$-model discussed, with $\C^*$-equivariance giving the
  $\Omega$-deformed version of these operators (see \cite{NaCoulomb}
  and \cite[\S
  1.3]{BBBDN} for more details).
\end{physics}

However, the utility of this perspective does not stop when we have
constructed the Coulomb branch; given any $\C^*$-action $\nu\colon \C^*
\to \Aut_G(V)$ on $V$
commuting with $G$, we can consider the image $\nu(t^k)\cdot
V[[t]]/G[[t]]$ and the map \[p^{(k)}\colon \nu(t^k)\cdot
V[[t]]/G[[t]]\to V((t))/G((t)).\]  The
BFN resolution $\tM$ constructed by
Braverman-Finkelberg-Nakajima in \cite{BFNline} can be defined  the
property that \[\Gamma(\tilde{\fM}; \mathcal{O}(k))\cong
  \Ext^\bullet(\mathscr{O}_{p^{(m)}}, \mathscr{O}_{p^{(m+k)}})\] with
multiplication in the projective coordinate ring given by Yoneda
product.  
 Loop equivariantly, these spaces are not isomorphic, but instead give a $\Z$-algebra
in the sense of Gordon and Stafford \cite{GS}; as discussed in
\cite[\S 5.2]{BLPWquant}, this is the quantum homogeneous coordinate
ring of a quantization of $\tilde{M}$.

Thus, given a map $q\colon \mathcal{M}\to \mathcal{L}$, we can define a coherent
sheaf $\mathcal{Q}_q$ on the BFN resolution defined by the property $\Gamma(\tilde{\fM};\mathcal{Q}_q\otimes \mathcal{O}(k))\cong \Ext^{\bullet}(\mathscr{O}_{p^{(-k)}},\mathscr{O}_q)$.  This is the construction that we require for our tilting generators and noncommutative resolution.
\begin{itheorem}\label{thm:B}
  There is a space  $\mathcal{M}$ and map $q\colon \mathcal{M}\to \mathcal{L}$ such
  that $ \mathcal{Q}_q$ is the tilting generator of Theorem
  \ref{th:main} and  \[A=\Ext^\bullet(\mathscr{O}_{q},
    \mathscr{O}_{q}).\]
\end{itheorem}
We should note that there is not just one such space, but in fact
there are several of them, which give rise to different noncommutative
resolutions and different tilting generators.  These different choices
are related by wall-crossing
functors (as defined, for example, in \cite[\S
2.5.1]{losevModularCategories2021}).  These functors also have a geometric
realization:
\begin{itheorem}\label{ithm:Schobers}
The derived equivalence of coherent sheaves to $A$-modules
intertwines wall-crossing functors with tensor product with the
bimodules $\Ext^\bullet(\mathscr{O}_{q},
    \mathscr{O}_{q'})$ for different spaces $q,q'$ both giving
    resolutions.  These
actions define a Schober in the sense of \cite{kapranovPerverseSchobers2015}, that is,
a perverse sheaf of categories, on a particular subtorus arrangement
in a complex torus.  
\end{itheorem}  
\begin{physics}
  All of these other D-modules can be interpreted naturally in physics in terms of natural modifications of the fields of the theory near the line defect.  We leave a more detailed discussion of this point to future work and the interested reader.

  Unfortunately, the author knows no good physics explanation of which combinations of line operators give noncommutative resolutions, and which do not.  Obviously, this would be an interesting question from a quantum field theory perspective.   
\end{physics}

\subsection{Summary of approach}
\label{sec:summary}

Proving these results depends on comparison with the characteristic
$p$ approach of Bezrukavnikov and Kaledin \cite{BKpos,KalDEQ}.  That
is, we consider quantizations in
characteristic $p$ and apply the approach of \cite{websterKoszulDuality2019} (based
in turn on \cite{FOD,MVdB}) in
positive characteristic.  The focus is on the
action of a large polynomial subalgebra of the quantum Coulomb branch,
and analyzing the representations of this algebra in terms of their
weights for this subalgebra.  In the perspective of Stadnik \cite{Stadnik} to resolving this problem
for hypertoric varieties,  this polynomial subalgebra played a key
role in constructing the requisite \'etale cover where the Azumaya
algebra constructed from a quantization splits.  

This approach extends to the case of a general BFN Coulomb branch.
Whenever we have a BFN resolution $\tM$, we obtain an explicit
tilting generator for $\K$ either of large positive characteristic or
characteristic 0, as described in Theorem  \ref{thm:B}.
This is an extension of work of Gammage, McBreen and the author \cite{mcbreenHomologicalMirror2024,gammageHomologicalMirror2023}, which
shows the same result in the abelian case.

As mentioned above, we'll cover the case of quiver gauge theories in
considerably greater detail in a companion paper \cite{WebcohII}.
Beyond this, there are several interesting possibilities for extension of this
work.  The work of McBreen and author in the abelian case \cite{mcbreenHomologicalMirror2024} can
be used to show one version of homological mirror symmetry for
multiplicative hypertoric varieties, and it would be very interesting
to relate the presentations of $A$ appearing here with the Fukaya
category of multiplicative Coulomb branches (i.e. the algebraic
varieties obtained by the K-theoretic BFN construction).  Interesting progress in this direction has recently been obtained by Aganagi\'c, Danilenko, Li, Shende and Zhou \cite{aganagicQuiverHecke2024}.  

The tilting bundles that appear also have a natural interpretation in
terms of line operators in the corresponding 3-dimensional gauge
theory, and one could hope that other perspectives on these line
operators, such as the vertex algebra perspective suggested in
Costello, Creuzig and Gaiotto \cite{CCG}, will also see the same combinatorial
constructions appear, hopefully eventually leading to a theory of
S-duality where coherent sheaves on Coulomb branches can be described
as a natural object on the Higgs side as well.

\subsection*{Acknowledgements}
\label{sec:acknowledgements}

Many thanks to Roman Bezrukavnikov, Alexander Braverman, Kevin Costello, Tudor
Dimofte, Michael Finkelberg, Justin Hilburn, Gus Lonergan, Ivan Losev,
Ted Stadnik, Alex Weekes and Philsang Yoo for useful discussions on these topics.

\section{Quantum Coulomb branches}

\subsection{Background}
\label{sec:background}

   \notation{${G}$}
   {The gauge group.}
Let us recall the construction of quantum Coulomb branches from
\cite{websterKoszulDuality2019}.  As before, let $G$ be a connected reductive algebraic
group over $\C$, with $G((t)), G[[t]]$ its points over
$\C((t)), \C[[t]]$. For a fixed Borel $B\subset G$, we let $\Iwahori$
be the associated Iwahori subgroup
\[\Iwahori=\{g(t)\in G[[t]]\mid g(0)\in B\}\subset G[[t]].\]  The {\bf affine flag variety} $\AF=G((t))/\Iwahori$ is
just the quotient by this Iwahori.  
\notation{$\Iwahori$}{The Iwahori $\Iwahori=\{g(t)\in G[[t]]\mid g(0)\in B\}\subset G[[t]]$ for a fixed Borel $B$.}

\notation{$V$}{The matter representation.}

\nc{\wtG}{\widetilde{G((t))}}
\notation{$\No$}{$\No=N_{GL(V)}(G)$}
 
   \notation{${F}$}
    {The
flavor group $\No/G$.}
   
     \notation{${T_*}$}
         {A maximal torus of the group $*$.}
Let $V$ be a fixed faithful $G$-representation and let $\No=N_{GL(V)}(G)$ be the normalizer of $G$ in $GL(V)$, and let
$F=\No/G$ be the flavor quotient and $T_\No,T_{F}$ be compatible maximal tori of
these groups.  
   
     \notation{${\ft_{*,\dagger}}$} 
         {The Lie algebra of the torus $T_{*}$ over $\dagger=\Z,\Q,\R,\C$.}
We use $\ft_{\No}$, etc. for the Lie algebra of this
torus, $\ft_{\No,S}$ for the subset where integral weights have
values in a subring $S\subset \C$, with the most important cases being
$S=\R$ and $S=\Z$.    

\notation{$\To$}{The subgroup of $\No$ generated by $\gaugeG$ and a maximal torus of $\No$.}
It's also useful to consider $\To$, the preimage of
$T_F$ in $\No$ and $\tilde{\To}=Q\times \C^*$.  This latter group acts on $V((t))$
such that $vt^a$ has weight $a$ under the second factor and the
obvious action of $\To$.

     \notation{${\flav}$}
   {The flavor: a fixed cocharacter $\flav\colon \C^*\to T_{F}$.}
Fix a flavor $\flav\colon \C^*\to T_F$, and let 
\[\tilde{G} =\{(g,s)\in \To\times \C^* \mid \flav(s)=g\pmod G\}\qquad
  \nu(g,s)=s,\]
with its induced action on $V((t))$.  
That is, $\tilde{G}$ is the pullback of the diagram $\C^*\to T_F \leftarrow
\To$. Let $\tilde{T}$ be the induced torus of this group and
$\tilde{\mathfrak{t}}$ its Lie algebra. 

Fix a subspace $U\subset V((t))$ invariant under $\Iwahori$.  
Let $\VB_U:=(G((t)) \times U)/\Iwahori$.  Note that we have a
natural $G((t))$-equivariant projection map $\VB_U\to V((t))$. 
Let
$\wtG$ be the subgroup of $\No((t))\times
\mathbb{C}^*$ generated by $G((t))$ and the image of
$\tilde{G}\hookrightarrow \tilde{G}\rtimes \C^*$ included via the
identity times $\nu$.

       \notation{${\EuScript{A}}$}
       {The Iwahori Coulomb branch algebra $\EuScript{A}=H_*^{BM, \wtG}(\VB_{V[[t]]}\times_{V((t))}\VB_{V[[t]]})$.}
\begin{definition} The BFN Steinberg algebra $\efA$ is the equivariant Borel-Moore homology group
   \[\efA=H_*^{BM, \wtG}(\VB_{V[[t]]}\times_{V((t))}\VB_{V[[t]]};\K)\] endowed with the convolution multiplication.  
 \end{definition}
 As discussed in Section \ref{sec:motivation-from-line}, this algebra
is intended to match an Ext algebra in the category of $D$-modules on
$\mathcal{L}$.  We can avoid any technicalities about the nature of
this category by considering this
homology space instead, and interpreting the equivariant homology $H_*^{BM,
  \wtG}(\VB_{V[[t]]}\times_{V((t))}\VB_{V[[t]]};\K)$ using the techniques in \cite[\S 2(ii)]{BFN}. 
  As usual, we let $h$ be
the equivariant parameter corresponding to the character $\nu$, and $\ringS_h=H^*(B\tilde{T};\K)=\K[\tilde{\ft}]$, which is naturally a subalgebra of $\EuScript{A}$ under the identification $H^{BM,\wtG}_*(\VB_{V[[t]]})\cong S_h$.  When we specialize $h=0,1$, we will write $S_0,S_1$, etc.
  \notation{$S_*$}
            {The symmetric algebra on $\tilde{\ft}^*$, that is, the ring of functions on the affine variety $\tilde{\ft}$, with the parameter $h$ specialized at $h=*$.}

       \notation{${\EuScript{A}^{\operatorname{sph}}}$}
       {The quantum Coulomb branch algebra  $\EuScript{A}^{\operatorname{sph}}=H_*^{BM, \wtG}(\EuScript{Y}_{V[[t]]}\times_{V((t))}\EuScript{Y}_{V[[t]]})$.}
 The original BFN algebra $\Asph$ is
defined in essentially the same way, using
$\EuScript{Y}_{V[[t]]}:=(G((t)) \times V[[t]])/G[[t]]$.  The algebras $\EuScript{A}^{\operatorname{sph}}$ and $\EuScript{A}$
  are Morita equivalent by \cite[Th. \ref{SD-th:Morita}]{websterKoszulDuality2019}, with
  $e_{\operatorname{sph}}
  \EuScript{A}e_{\operatorname{sph}}=\EuScript{A}^{\operatorname{sph}}$
  for an idempotent $e_{\operatorname{sph}}\in \EuScript{A}$.
  \begin{physics}
    In the parlance of quantum field theory, $\Asph$ is the algebra of local operators on a trivial line defect, and $\efA$ the
    algebra of local operators on the defect that comes from coupling
    to super quantum mechanics on the flag variety $G/B$; this
    precisely the ``abelianizing'' line denoted $\mathbb{V}_{\mathcal{I}}$ in \cite[\S 7]{DGGH}.  The
    equivariant parameter $h$ corresponds to the $\Omega$-background for the circle rotating
    around this line in $\R^3$, as discussed in \cite[\S 6]{BBBDN}.
  \end{physics}

\begin{definition}\label{def:Coulomb-branch}
The {\bf Coulomb branch} $\Coulomb$ for $(V,G)$ is the spectrum of the algebra \Asph after specialization at $h=0$ (at which point it becomes commutative).  The {\bf quantum Coulomb branch} is the specialization of this algebra at $h=1$.
\end{definition}
\notation{${\Coulomb}$}{The Coulomb branch of the gauge theory with gauge group $\gaugeG$ and matter representation $\matterV$ (Definition \ref{def:Coulomb-branch}).}

Of course, $\To$ still acts on $V$, and thus has an associated Coulomb
branch $\MQ$. 
     \notation{$\MQ$}
    {The Coulomb branch $\Coulomb$ attached to the group
  $\To$ acting on $V$ with its usual action.}

  \notation{${K}$}
    {The Langlands dual $T_F^\vee$, or equivalently, the
     Pontryagin dual $\Hom(\ft^*_{\Z}, \C^*)$.}  
 As discussed in \cite[\S 3]{BFN} and \cite[\S
\ref{SD-sec:pres-extend-categ}]{websterKoszulDuality2019}, this Coulomb branch has a Hamiltonian action of
$\groupK=T_F^\vee$, the Langlands dual of the dual of the torus of the flavor
group $F$ with moment map given by
$\mathfrak{t}_F^*\to H^*_{\To}(pt)$, and $\fM$ is the categorical
quotient of the zero-level of the moment map on $\fM_{\To}$, the
Coulomb branch for $\To$.  For a given cocharacter of $T_F$
(considered as a character of $K$), we can instead take the
associated GIT quotient of $\fM_{\To}$, which gives a variety
$\tM$ which maps projectively to $\fM$.  As mentioned in the
introduction, if $\tilde{\fM}\to \fM$ is a resolution of singularities
(or equivalently, if $\tilde{\fM}$ is smooth) then we call it a {\bf
  BFN resolution}.
 \notation{${\tM}$}{A BFN resolution of the Coulomb branch $\Coulomb$.}

\subsection{The extended category}
\label{sec:extended}
The quantization of the Coulomb branch attached to $(G,V)$ appears as an endomorphism algebra in a larger category, building on the geometric definition of this algebra by Braverman, Finkelberg and Nakajima
\cite{NaCoulomb,BFN}. This category is not unique; there are actually
many variations on it one could choose, and it will be convenient for
us to incorporate a parameter $\delta\in (0,1)\subset \R$ into its
definition; in \cite{websterKoszulDuality2019}, we assumed that $\delta=1/2$, but this
played no important role in the results of that paper (in fact, some
results become simpler if we choose $\delta$ generic instead).
   
   \notation{${\delta}$}
   {A parameter between the open interval $(0,1)\subset \R$ used in the definition of $\scrB$.}

\notation{$\ft_{1,*,\R}$}{The preimage
of 1 under the projection  $\tilde{\ft}_{*,\R}\to \C=\operatorname{Lie}(\C^*)$}
Let $\ft_{1,\To,\R}\subset \tilde{\ft}_{\To,\R}$ be the preimage
of 1 under projection to $\C=\operatorname{Lie}(\C^*)$ and let $\ft_{1,\R}=\ft_{1,
  \To,\R}\cap\tilde{\ft}$, be the space of real
lifts of the cocharacter $\flav$.
\newcommand{\varphimid}{\varphi^{\operatorname{mid}}}
As in \cite{websterKoszulDuality2019}, we let
$\{\varphi_i\}$ be the multiset of weights of $V$ (considered as
functions on $\tilde{\ft}_{\To}$)  and we let \begin{equation}\label{eq:varphi}
	\varphi_i^+=\varphi_i\qquad
\varphimid_i= (1-\delta) \varphi_i^+-\delta\varphi_i^-=\varphi_i+\delta\nu\qquad \varphi_i^-=-\varphi_i-\nu.  
\end{equation}
     
       \notation{${\varphi_i}$}
         {The weights of  $ V$ over $G$ or $\No$.}

       \notation{${\varphi_i^{\operatorname{mid}}}$}
         {The average of $\varphi_i^+$ and $-\varphi_i^-$.}

Given any $\acham\in \ft_{1,
  \To,\R}$, we can consider the induced action on the
vector space $V((t))$.  
\begin{itemize}
\item Let $\Iwahori_\acham$ be the subgroup whose Lie algebra is the sum of positive weight
spaces for the adjoint action of $\acham$. This only depends on the
alcove in which $\acham$ lies, i.e. which chamber of the arrangment
\[\{\alpha(\acham)=n\mid \al\in \Delta, n\in \Z\}\] contains
$\acham$; the subgroup $\Iwahori_\acham$ is an Iwahori if $\acham$ does not
lie on any of these hyperplanes. 
\item Let  $U_\acham\subset V((t))$ be the subspace of elements of 
weight $\geq -\delta$ under $\acham$.  This subspace is closed under the action of
$\Iwahori_\acham$.  This only depends on the vector $\Ba$ such that
\begin{equation}
\acham\in \AC_{\Ba}=\{\xi \in \ft_{1,
  \To,\R}\mid a_i<\varphimid_i(\xi)<a_i+1\text{
  for all $i$}\}.\label{eq:aff-cham}
\end{equation}
\end{itemize}

We call $\acham$ {\bf unexceptional} if does not lie on the unrolled matter hyperplanes
\[varphi_i^{\operatorname{mid}}(\acham)=n\mid n\in
\Z\}\] and {\bf generic} if it is unexceptional and does not lie on any
of the unrolled root hyperplanes $\{\alpha(\acham)=n\mid
n\in\Z\}$. We'll call the hyperplanes generic points avoid the {\bf
  unrolled hyperplane arrangment}.  

For any $\acham\in  \ft_{1,
  \To,\R}$, we can
consider
$\VB_{\acham}:=\VB_{U_\acham}:=G((t))\times_{\Iwahori_{\acham}}U_{\acham}$,
the associated vector bundle.  
The space $ \ft_{1,
  \To,\R}$ has a natural adjoint action of
$\What=N_{\wtG}(T)/T$, and of course,
$U_{w\cdot \acham}=w\cdot U_{\acham}$.  
      
     \notation{${\widehat{W}}$}
    {The affine Weyl group of $G$.  The semi-direct
     product of $ \weylW$ and the coweight lattice of $T$.}

   \notation{${{}_{\acham}\VB_{\acham'}}$}
             {The subspace $\left\{(g,v(t))\in G((t))\times
        U_{\acham}\mid g\cdot v(t)\in
        U_{\acham'}\right\}/\Iwahori_{\acham}\subset \VB_{\acham}$.}
\newcommand{\dVB}{\VB}
We let
\begin{equation}
{}_{\acham}\dVB_{\acham'}=\left\{(g,v(t))\in G((t))\times
        U_{\acham}\mid g\cdot v(t)\in
        U_{\acham'}\right\}/\Iwahori_{\acham}.\label{eq:aXa}
    \end{equation}
        
 \notation{${\mathscr{B}}$}
 {The extended BFN category, defined in Definition \ref{def:extended-BFN}.}
\begin{definition}\label{def:extended-BFN}
  Let the {\bf extended BFN category} $\scrB^+$ be the category whose
  objects are unexceptional cocharacters $\acham\in  \ft_{1,
  \No,\R}$, with morphisms given by:
  \begin{equation*}
    \Hom(\acham,\acham')=H_*^{BM, \wtG}(\VB_{\acham}\times_{V((t))}\VB_{\acham'};\K)\\
    \cong H_*^{BM, \tilde T}\left({}_{\acham}\VB_{\acham'};
    \K\right).
\end{equation*}
Let $\scrB$ be the subcategory whose objects are given by $\ft_{1,\R}$.
\end{definition}
As before, this homology is defined using the techniques in \cite[\S
2(ii)]{BFN}.
\begin{physics}\label{physics:line}
  As discussed in Section \ref{sec:motivation-from-line}, the objects
  in this category can be interpreted as D-modules on the loop space
  of $V/G$ (for those inclined toward stacks) or as line defects in a
  $\sigma$-model (for those inclined toward quantum field theory).
  In the notation of  \cite[\S 4.3]{DGGH}, these correspond to the
  Lagrangian $\mathcal{L}_0$ given by the conormal to $U_{\acham}$ and
  the subgroup $\mathcal{G}_0=\Iwahori_{\acham}$.   
  The author has no especially good explanation from either of these
  perspectives why this is the ``right'' subcategory of line operators
  to consider when there are many others available, but it does get the job done.
\end{physics}

   \notation{${\second}$}
   {The cocharacter  that acts on $V((t))$ with weight $a$ on $vt^a$.}
Note that by assumption, the
cocharacter $\second$ defined by acting on $vt^a$ by weight $a$ is unexceptional, but not generic and  $U_{\second}=V[[t]]$.
Given any unexceptional point $\acham$, it has a neighborhood in the
classical topology, which necessarily contains a generic point, on
which $U_{\acham'}=U_{\acham}$. Thus we can find a generic element $\zero$ of the fundamental alcove such that $U_\zero=V[[t]]$. In this case, we
have that $\Iwahori_\second=G[[t]]$ and $\Iwahori_\zero$ is the standard Iwahori so \begin{equation}\label{eq:A-B}
    \Asph=\Hom_{\mathscr{B}}(\second,\second)\qquad \EuScript{A}=\Hom_{{\mathscr{B}}}(\zero,\zero)
\end{equation}  Thus, this extended category encodes the structure of $\EuScript{A}$.
\begin{definition}
  Let $\Phi(\acham,\acham')$ be the product of the terms
  $\varphi^+_i-nh$ over pairs $(i,n) \in [1,d]\times \Z$ such that the inequalities below both hold:
  \[\varphi_i^{\operatorname{mid}}(\acham)>n \qquad  \varphi_i^{\operatorname{mid}}(\acham')<n \] hold.   
  $\Phi(\acham,\acham',\acham'')$ be the product of the terms
  $\varphi^+_i-nh$ over pairs $(i,n)\in [1,d]\times \Z$ such that we
   the inequalities of (\ref{pmp}) or the inequalities of (\ref{mpm}) below hold:\newseq
  \[\subeqn\label{pmp}\varphi_i^{\operatorname{mid}}(\acham'')>n \qquad
    \varphi_i^{\operatorname{mid}}(\acham')<n \qquad \varphi_i^{\operatorname{mid}}(\acham)>n\] 
  \[\subeqn\label{mpm}\varphi_i^{\operatorname{mid}}(\acham'')<n \qquad
    \varphi_i^{\operatorname{mid}}(\acham')>n\qquad \varphi_i^{\operatorname{mid}}(\acham)<n. \] These terms correspond to the hyperplanes that a path
  $\acham\to\acham'\to \acham''$ must cross twice. 
\end{definition}

\notation{${y_w}$}
             {The homology classes that give the action of
   $\widehat{W}$ in $\mathscr{B}$.}
 
Recall from \cite[Thm. \ref{SD-thm:BFN-pres}]{websterKoszulDuality2019} that we have:
\begin{theorem}\label{thm:BFN-pres}
  The morphisms in the extended BFN category are generated by
  \begin{enumerate}
  \item $\yw_w$ for $w\in \What$, the graph of a lift of $w$:
\begin{equation*}
\yw_w=[\big\{(w,v(t))\mid v(t)\in
U_\acham\big\}]/\Iwahori_\acham;\label{eq:y-def}
\end{equation*}
\notation{${r(\eta,\eta')}$}
             {The homology classes corresponding to fibers over
   torus fixed points.}
  \item $ {r(\acham,\acham')}$ for $\acham,\acham'\in \ft_{1,\To,\R}$
    generic
\begin{equation*}
r(\acham,\acham')=\left[\{(e,v(t)) \in T((t))\times U_{\acham'}\mid v(t)\in
U_{\acham}\}/T[[t]]\right]\in \Hom_{\ab}(\acham',\acham) ;\label{eq:r-def}
\end{equation*}

\notation{${\Bpsi}_{\al}$}
            {The homology classes corresponding crossing a root hyperplane.}
  \item $ {u_{\al'}}(\acham)= {u_{\al-n\delta}}(\acham)$ for
    $\acham_\pm$ affine chambers adjacent across $\al'(\acham)=0$ for
    $\al'\in \hatD$ an affine root (i.e. $\al'=\al-n\delta$ for some
    finite root $\al'$)
\begin{equation*}
 {u_{\al'}}(\acham)=[{\left\{(gv(t),g \cdot\Iwahori_{\pm} ,g\cdot \Iwahori_{\mp})\in \VB_{\acham^\pm}\times_{V((t))}\VB_{\acham^{\mp}}
  \mid g\in
  G((t)), v(t)\in U_\acham\right\}}];\label{eq:psi-def}
\end{equation*}
  \item the polynomials in $\ringS_h$.
  \end{enumerate}
  This category has a polynomial representation where each object $\acham$ is assigned to $H_*^{BM,\wtG}(\VB_{\acham})\cong \Cth\cdot [\VB_{\acham}]$, and the generators above act by:
  \newseq \begin{align*}
\subeqn\label{eq:ract} 
 {r(\acham,\acham')} \cdot f [\VB_{\acham'}]&=\Phi(\acham,\acham') f\cdot [\VB_{\acham}]\\
  \subeqn\label{eq:psiact}
\Bpsi_{\al}\cdot f[\VB_{\acham_{\pm}}]&=\partial_{\al}(f)\cdot [\VB_{\acham_{\mp}}]\\
   \subeqn\label{eq:wact} \yw_w\cdot f[\VB_{\acham}]&=(w\cdot f)[\VB_{w\cdot \acham}]\\
     \subeqn\label{eq:muact}
\mu \cdot f [\VB_{\acham}]&=\mu f\cdot [\VB_{\acham}]
  \end{align*} \newseq 
The relations between these operators are given by:
 \begin{align*}
\subeqn\label{eq:dot-commute}
\mu \cdot  r(\acham,\acham') &= r(\acham,\acham')\cdot \mu   \\
\subeqn\label{eq:weyl1}
  y_{\zeta}\cdot\mu\cdot y_{-\zeta}&=\mu+h \langle \zeta,\mu\rangle \\
\subeqn\label{eq:wall-cross1}
r(\acham,\acham') r(\acham'',\acham''')&=
\delta_{\acham',\acham''}\Phi(\acham,\acham',\acham''')
                                         r(\acham,\acham''')\\
\subeqn\label{eq:coweight2}
y_w\cdot y_{w'}&=y_{ww'}\\
\subeqn\label{eq:conjugate2}
  y_w r(\acham',\acham) y_w^{-1}&=r(w\cdot \acham',w\cdot\acham) \\
\subeqn\label{eq:weyl2}
  y_w \mu y_w^{-1}&=w\cdot \mu\\
   \subeqn\label{eq:psi2}
\Bpsi_{\al}^2&=0\\
   \subeqn\label{eq:psi}
\underbrace{\Bpsi_{\al}\Bpsi_{s_\al \beta}\Bpsi_{s_\al s_{\beta}\al}\cdots}_{m_{\al\be}}
             &=\underbrace{\Bpsi_{\beta}\Bpsi_{s_{\beta}\al}\Bpsi_{s_{\beta}s_\al\beta}\cdots}_{m_{\al\be}}\\
   \subeqn\label{eq:psiconjugate}
y_w\Bpsi_{\al}y_{w^{-1}}&=\Bpsi_{w\cdot \al}\\
  \subeqn\label{eq:psipoly}
\Bpsi_{\al} \mu-(s_{\al}\cdot\mu)\Bpsi_{\al}
             &=r(\eta_{\mp},\eta_{\pm})\partial_{\al}(\mu) \end{align*} whenever these
           morphisms are well-defined
and finally, if $\acham'_\pm$ and $\acham''_\pm$ are two pairs of chambers
opposite across $\al(\acham)=0$ on opposite sides of an intersection
of affine root and flavor hyperplanes as shown in \Cref{fig:weight-root} and
$\acham,\acham'''$ differ by a $180^\circ$ rotation around the
corresponding codimension 2 subspace:
 \addtocounter{subeqn}{1}
\begin{multline*}
\subeqn\label{eq:triple} 
r(\acham''',\acham'_-)\Bpsi_{\al} r(\acham'_+,\acham)
-r(\acham''',\acham''_-)\Bpsi_{\al}
r(\acham''_+,\acham)\\=\partial_\al\left(\Phi(\acham'_+,\acham) \cdot s_{\al}\Phi(\acham,\acham'_-)\right)
r(\acham''',s_\al\acham) s_\al.
\end{multline*}
\end{theorem}
\begin{figure}
    \centering
   \begin{equation*}
      \begin{tikzpicture}[very thick]
        \draw[dotted] (-2,0) -- node [at
        start,left]{$\al$}(2,0); \draw (2,-1)--  (-2,1); \draw (-2,-1)-- (2,1); 
\draw (1,-2)-- (-1,2); \draw (-1,-2)--  (1,2); 
\node at (1.7,-.5)
        {${\acham'_+}$}; 
\node at (1.7,.5)
        {${\acham'_-}$}; 
        \node at (-1.7,-.5)
        {${\acham''_+}$};
        \node at (-1.7,.5)
        {${\acham''_+}$};
\node at (1.3,-1.3)
        {${\acham}$}; \node at (-1.3,1.3)
        {${\acham'''}$};
\node at (0,-1.3){$\cdots$}; 
\node at (0,1.3){$\cdots$};
      \end{tikzpicture}
    \end{equation*}
    
    \caption{A point of intersection for weight and root hyperplanes}
    \label{fig:weight-root}
\end{figure}

One important change in the characteristic $p$ case is that the
representation defined by (\ref{eq:ract}--\ref{eq:muact}) is no longer
faithful, since the same is true of the corresponding representation
of $\widehat{W}$: translations by cocharacters divisible by $p$ act
trivially.

It is possible to fix this, though it is somewhat less pleasant to
think about. Fix $h=g\in \C$ (we will of course be primarily interested in the cases $g=0,1$).  Let $K$ be the fraction field of $\ringS_g$, and consider
the induced action by convolution on $K\otimes_{S_g}
H_*^{BM,T}(\VB_{\acham}^T)\cong \oplus_{\la\in X_*(T)}K\cdot
[t^\la]$.  We will not explicitly check that the action we define
below arises from convolution due to the complications of
localization in equivariant cohomology for loop groups, but it is
worth pointing to as our source of inspiration.
  \begin{lemma}\label{lem:frac-rep}
There is a faithful action of $\scrB^+$  that sends every object
to  $K_X:=  \oplus_{\la\in X_*(T)}K\cdot [t^\la]$
given by  the formulas \newseq\begin{align*}
\subeqn\label{eq:ract-frac}
r  (\acham,\acham') \cdot f [t^\la]&=\Phi(\acham,\acham') f\cdot
                                           [t^\la]\\
\subeqn\label{eq:psiact-frac}
\Bpsi_{\al}\cdot f [t^\la]&=\frac{s_{\al}f}{\al}[t^{s_{\al}\la}]-\frac{f}{\al}[t^{\la}]\\
   \subeqn\label{eq:wact-frac}
\yw_w\cdot f[t^\la]&=(w\cdot f)[t^{w\la}]\\
     \subeqn\label{eq:muact-frac}
\mu \cdot f [t^\la]&=\mu f\cdot [t^\la].
  \end{align*} 
  \end{lemma}

Now, consider the case where $\hbar=0$.  In this case,
$\End_{\mathscr{B}}(\second,\second)\cong \K[\fM]$ is the space of functions on the
  Coulomb branch.  The different non-isomorphic objects of
  $\scrB_{\al}$ define interesting modules over $\K[\fM]$, considering
  $Z_{\eta}=\Hom_{\mathscr{B}}(\zero,\eta)$ as a right module under composition.   
\begin{lemma}\label{lem:Q-rank}
  The module $Z_{\eta}$, considered as a coherent sheaf on $\K[\fM]$, is generically free of rank $\# \weylW$.
\end{lemma} 
\begin{proof}
  As a module over $\End_{\mathscr{B}}(\zero,\zero)$, the module $\End_{\mathscr{B}}(\zero,\eta)$ is generically free of rank 1, since $\zero$ and $\eta$ become isomorphic after inverting all weights and roots.  Since at $h=0$, the algebra $\End_{\mathscr{B}}(\zero,\zero)$ is Azumaya over $\K[\fM]$ with degree $\#W$, this implies that $Z_{\eta}$ generically has the correct rank.  
\end{proof}

     \notation{${ {}_{\phi+\nu}\mathscr{T}{}_{\phi}}$}
          {The $\mathscr{B}_{\phi+\nu}\operatorname{-}\mathscr{B}_{\phi}$
  bimodule formed by the appropriate quotient of
  $\mathscr{T}(\nu)$, the morphisms of weight $\nu$ in $\scrB^{\To}$.}
    
     \notation{${ {}_{\phi+\nu}{T}{}_{\phi}}$}
          {The twisting bimodule $ {}_{\phi+\nu}\scrT{}_{\phi}(\zero,\zero)$.}
One other construction we'll need to connect to wall-crossing functors
is the twisting bimodules 
  ${}_{\phi+\nu}\scrT{}_{\phi}$ 
and
${}_{\phi+\nu}\Twist_{\phi}={}_{\phi+\nu}\mathscr{T}{}_{\phi}(\zero,\zero)$ relating two flavors $\flav$ and
$\flav+\nu$ that differ by $\nu\in \ft_{\Z,F}$
defined in \cite[Def. \ref{SD-def:Bdefr}]{websterKoszulDuality2019}.
These are constructed much as the Hom spaces in Definition
\ref{def:extended-BFN}: let ${}_{\acham}\VB^{(\nu)}_{\acham'}$ be the
component of the space ${}_{\acham}\dVB^{\To}_{\acham'}$  as defined in
\eqref{eq:aXa} lying above $t^\nu$ in the affine Grassmannian of
$F$, and we have:
\begin{equation}
  \label{eq:aXnua}
  { {}_{\phi+\nu}\mathscr{T}{}_{\phi}}(\acham,\acham')=  H_*^{BM, \tilde T}\left({}_{\acham}\VB^{(\nu)}_{\acham'}; 
    \K\right).
\end{equation}
This is a subspace of $H_*^{BM, \tilde T}\left({}_{\acham}\dVB^{\To}_{\acham'}\right)$. 
We can identify this space with the corresponding morphism space $Hom_{\scrB^{\To}}(\acham,\acham')$ in the extended BFN category $\scrB^{\To}$ for the group $\To$, modulo the maximal ideal $\mathsf{m}_{\phi+\nu}\subset \Sym(\mathfrak{t}_{F}^*)$ corresponding to $h(\phi+\nu)$.  The homology classes lying over $t^{\nu}$ are exactly those whose weight for the action of $T_{F}^{\vee}$ is $\nu$ (by the definition of this action); this ensures that the right action of $\Sym(\mathfrak{t}_{F}^*)$ kills the maximal ideal $\mathsf{m}_{\phi}$.

\subsection{Representations}
\label{sec:reps}
Using this presentation, we can analyze the structure of this category of representations in
characteristic $p$, just as we did in characteristic 0 in \cite[\S
\ref{SD-sec:pres-extend-categ}]{websterKoszulDuality2019}.   It is worth noting that the group $G$, representation
$V$ and its associated objects are unchanged; we simply consider their
homology over $\K$, a field of characteristic $p$.  To save ourselves
heartburn, we assume that $p$ is not a torsion prime for the group
$G$.  This is not a problematic restriction, since we will typically
assume that $p\gg 0$.    Throughout this subsection, we specialize $h=1$.  We'll let $\ringS_1$ denote the specialization of $\ringS_h$ at $h=1$.\notation{$\ringS_1$}{The specialization of $\ringS_h$ at $h=1$.}

Let $M$ be a finite-dimensional representation
of the category $\scrB$ (which we will also call $\mathscr{B}$-modules), that is, a functor from $\mathscr{B}$ to the
category $\K\Vect$ of finite dimensional $\K$-vector spaces.  

These are closely related to the theory of $\EuScript{A}$-modules since the finite-dimensional vector space $N:=M(\zero)$ has an
induced $\EuScript{A}$-module structure.  Furthermore, since $\Hom(\acham,\zero)$
and $\Hom(\zero,\acham)$ are finitely generated as
$\EuScript{A}$-modules, this is a quotient functor, with left adjoint given by
\[N\mapsto
\mathscr{B}\otimes_{\EuScript{A}}N(\acham):=\Hom(\acham,\zero)\otimes_{\EuScript{A}}N.\]

Now, let us return to the theory of $\mathscr{B}$-modules. Of course, if we restrict the action on $M(\eta)$ to the subalgebra
$\ringS_1$, then this vector space breaks up as a sum of {\bf weight spaces}:
\begin{equation}
  \label{eq:W-def}
  \Wei_{\upsilon,\acham}(M)=\{m\in M(\acham)\mid \mathfrak{m}_{\upsilon}^Nm=0 \text{ for } N\gg 0\},
\end{equation} 
where $\mathfrak{m}_{\upsilon}\subset S_1$ is the maximal ideal defined by polynomials vanishing at $\upsilon$.
\notation{$\mathfrak{m}_{\upsilon}$}{The maximal ideal defined by polynomials vanishing at $\upsilon\in \ft_{1}$.}
We can think of this as an exact functor
$\Wei_{\upsilon,\acham}\colon \mathscr{B}\operatorname{-fdmod}\to \K\Vect$.
In \cite{websterKoszulDuality2019}, we employed these weight functors to probe the
category of representations of  $\mathscr{B}$.   Versions of this
construction have appeared a number of places in the literature,
including work of Musson and van der Bergh \cite{MVdB} and Drozd,
Futorny and Ovsienko \cite{FOD}.
\begin{definition}
  Let $\scrBhat$ be the category whose objects are the
  set $\EuScript{J}$ of pairs of generic $\acham\in \ft_\R+\tau$ and any
  $\upsilon\in \ft_{1,\K}$, such that
  \[\Hom_{\widehat{\mathscr{B}}}((\acham',\upsilon'),(\acham,\upsilon))=\varprojlim
  \Hom_{{\mathscr{B}}}(\acham',\acham)/(\mathfrak{m}_{\upsilon}^N
  \Hom_{{\mathscr{B}}}(\acham',\acham)+\Hom_{{\mathscr{B}}}(\acham',\acham)\mathfrak{m}_{\upsilon'}^N).\] 
\end{definition}
\notation{$\scrBhat$}{The category whose objects are pairs of objects in $\scrB$ and maximal ideals in  $\ringS_{1}$, with morphism spaces given by appropriate completion.  }  
We can apply \cite[Theorem B]{WebGT} here to get a sense of the size
of this algebra---the endomorphism algebra
of any object in this category is again a Galois order in a skew group
algebra, but the group is now the stabilizer of $\acham'$ in the
affine Weyl group $\What$. Since we are now in characteristic $p$, this
contains all translations that are $p$-divisible, and so this
stabilizer is the semi-direct product of a parabolic subgroup in the
finite Weyl group with this $p$-scaled group of translations.

 \notation{${\mathscr{\widehat{B}}_{\upsilon}}$}
    {The subcategory of $\scrBhat$ where we only allow objects  $(\eta,\upsilon)$ with $\eta\in \ft_{\second}, \upsilon\in \upsilon'+\ft_\Z$. }
 
 \notation{${\mathscr{\widehat{A}}_{\upsilon}}$}
     {The subcategory of $\scrBhat$ where we only allow objects of the form
$(\zero,\upsilon)$ for $\upsilon \in \upsilon'+\ft_\Z$.  }
Note that since $\K$ is of characteristic $p$, the set $\ft_{1,\K}$
has an action of $\ft_{\Z/p\Z}$ by addition.
We let ${\mathscr{\widehat{B}}_{\upsilon'}}$ be the subcategory where we only
allow objects  with $\upsilon\in \upsilon'+\ft_{\Z/p\Z}$ and let
$\mathscr{\widehat{A}}_{\upsilon'}$ be the subcategory with the objects of the form
$(\zero,\upsilon)$ for $\upsilon \in \upsilon'+\ft_{\Z/p\Z}$.  

This category is useful in that it lets us organize how the weight spaces
of different values of $\acham$ relate.  For any $\mathscr{B}$-module
$M$, the functor $(\upsilon, \acham)\mapsto W_{\upsilon,\acham}(M)$
defines a representation of $\widehat{\mathscr{B}}$.  We have an
analogue in this situation of \cite[Lem. \ref{SD-lem:weight-complete}]{websterKoszulDuality2019}, which is a
special case of a more general result of Drozd-Futorny-Ovsienko
\cite[Th. 17]{FOD}:
\begin{lemma}\label{lem:FOD}
  The functor above defines an equivalence of the category $\scrB\mmod_{\upsilon'}$ of finite dimensional $\scrB$-modules with weights in $\What\cdot \upsilon'$ to
  the category of representations of
  $\widehat{\mathscr{B}}_{\upsilon'}$ in $\K\Vect$.

The analogous functor defines an equivalence of the category $\EuScript{A}\mmod_{\upsilon'}$ of finite dimensional
$\EuScript{A}$-modules  with weights in
$\widehat{W}\cdot \upsilon'$ to the category of
representations of $\mathscr{\widehat{A}}_{\upsilon'}$ in $\K\Vect$.
\end{lemma} 
There is an important difference
between the characteristic $p$ and characteristic $0$ cases:
A module in $\EuScript{A}\mmod_{\upsilon'}$ with $\K$ of characteristic $p$ with finite dimensional
weight spaces is necessarily finite dimensional (since the affine Weyl
group orbit of any weight is finite) while it is typically not finite-dimensional if $\K$
has 
characteristic $0$. This is why here we only study finite dimensional
modules, while in \cite{websterKoszulDuality2019}, we study the category of all weight modules.

\subsection{The homogeneous presentation}
\label{sec:homo}
We wish to give a homogeneous presentation of the categories
$\mathscr{\widehat{A}}_{\upsilon'}$ and
$ {\mathscr{\widehat{B}}_{\upsilon'}}$, as in
\cite[\S \ref{SD-sec:higgs-coulomb}]{websterKoszulDuality2019}.  For simplicity, we assume that $\K=\Fp$ for
$p$ not dividing $\#  \weylW$ and we are still specializing $h=1$.

Recall that we have fixed $\flav\in \ft_{F,\Z}$; we can without
loss of generality choose a lift $\tilde{\phi}$ which is fixed by the action of $W$  to $\ft_{1,\To,
  \Z[\frac{1}{\# W}]}$ over the ring $\Z[\frac{1}{\# W}]$ of integers with the order of the
group $W$ inverted 
(note that this might not be
possible over $\Z$).   This has a unique reduction to
$\ft_{1,\Fp}$, which is again $W$-invariant.  Since
$\ft_{1,\Fp}$ is a torsor for $\ft_{\Fp}$, we can assume
$\upsilon'=\tilde{\phi}\pmod {p}$.

The
stabilizer $\widehat{W}_{\upsilon'}$ of $\upsilon'$ in the affine Weyl
group is generated by $s_{\al_i}$ for all $i$, and 
translations by $p \ft_{\Z}$.  Note that the map \[{(\cdot )}_p\colon \widehat{W}\to \widehat{W}_{\upsilon'} \qquad w_p(x)= p\cdot w\Big(\frac{1}{p}x\Big)\]
is an isomorphism between these groups.  Furthermore, for reflection
in an affine root $\al$, we have that $(s_\al)_p=s_{\al^{(p)}}$ for
some possibly different root $\al^{(p)}$.  

We can also understand this homomorphism in terms of the Frobenius map $\frob_V\colon V((t))\to V((t))$ given by $\frob_V(v(t))=v(t^p)$: the element $w_p$ is the unique one satisfying $w_p\circ \frob_V=\frob_V \circ w$.  Similarly, we will want to understand the interaction between $U_{\eta}$ and this map.  Consider the $\tilde{G}\times \C^*$ action on $V((t))$ as usual (that is, with $\C^*$ acting by loop rotation).  The map $\frob_V$ intertwines the action of 
$G((t))$ with the action twisted by the endomorphism $\frob_G(g(t))=g(t^p)$.

Note, we cannot extend this automorphism to the semi-direct product incorporating the loop scaling, since we would need to act on the loop $\C^*$ by $s\mapsto s^{1/p}$.  The corresponding automorphism on the level of Lie algebras is well-defined however, and we denote it by $\frob_{\tilde{\fg}}$.   Note that this does not preserve the subalgebra $\{(X,d\nu(X))\mid X\in \tilde{\fg}\}$, and thus does not preserve $\widetilde{\fg((t))}$. 

This shows that we need to have a different flavor in order to write $\frob_V^{-1}(U_\eta)$ as the same sort of subspace.  
\begin{definition}\label{def:pth-root}
  The {\bf ``$p$th root''} conventions for the extended category are taking:
  \begin{itemize}
    \item the gauge group $G$ and representation $V$;
      \item the (rational) flavor $\phi_{1/p}(s)=(\phi_0(s^{1/p}),s)$ in place of $\flav$;
      \item the constant $\delta/p$ in place of $\delta$;
      \item specialize the equivariant parameter $h=0$.
  \end{itemize}
  Throughout, we will use sans-serif letters to denote objects defined
  in the $p$th root conventions. In particular, we write
  $\mathsf{t}_{1,\R}$ for  $\ft_{1,\R}$,  $\mathsf{U}_\eta$ for
  $U_{\eta}$.  We let $\sfB$ denote the category $\scrB$ when we use this
  modified flavor and constant. 
\end{definition}

\begin{definition}
  Given $\eta\in \mathsf{t}_{1,\To,\R} $, let $\eta_p\in\ft_{1,\To,\R}$ be the
  element  $\eta_p=p\cdot \frob_{\mathfrak{\tilde{g}}}(\eta)$.  Given
  $\xi\in \ft_{1,\To,\R}$, let  $\xi_{1/p}\in \mathsf{t}_{1,\To,\R}$ be the unique solution to $\xi=(\xi_{1/p})_p$, that is, the inverse map.
\end{definition}
Like the description for $\phi$ above, if we write $\eta(t)=(\eta_0(s),s)$, then $\eta_{1/p}(s)=(\eta_0(t^{1/p}),s)$.  \notation{$\eta_{1/p}$}{The rescaled $\R$-cocharacter defined by $\eta_{1/p}(s)=(\eta_0(t^{1/p}),s)$.}

Note that for $w\in \widehat{W}$, we have $(w\eta)_p=w_p\eta_p$, and that
\begin{equation}
\varphimid_i(\eta_p)=p\mathsf{\varphi}^{\operatorname{mid}}_i(\eta)\label{eq:p-mid}
\end{equation}
(where the second is calculated using $p$th root conventions).
\begin{lemma}
  $\frob_V^{-1}(U_{\eta_p})=\mathsf{U}_{\eta}$.
\end{lemma}
\begin{remark}
  The map $\xi\mapsto \xi_{1/p}$ has the effect of shrinking the space
  $\ft_{1,\To,\R}$ by a factor of $\frac{1}{p}$, and the unrolled matter
  hyperplanes, which are defined by  $\varphimid_i(\xi)$ taking an
  integral value, become the hyperplanes where this same function
  (with the $p$th root conventions)  takes on a value in
  $\frac{1}{p}\Z$.  Thus we only keep every $p$th one of these
  hyperplanes as a matter hyperplane.

  Each unrolled matter hyperplane separates the locus of $\eta$ such
  that a given weight vector in $V((t))$ lies in $U_\eta$ from the
  locus where is does not; the hyperplanes we keep in the $p$th root
  conventions are those that correspond to vectors in $V((t^p))$.
\end{remark}

\begin{example}
  Let us consider the running example from \cite{websterKoszulDuality2019}: the gauge
  group $G=GL(2)$ acting on $V=\C^2\oplus \C^2$.  The flavor group $F$
  is isomorphic to $PGL(2)$, so choosing a flavor is choosing a
  cocharacter into this group, fixing the difference between the
  weights of this cocharacter on the two copies of $\C^2$.  Let's
  consider $p=5$, and choose the flavor so that the difference is
  $\varphi_1^{\operatorname{mid}}-\varphi_3^{\operatorname{mid}}=3$.
  In Figure \ref{fig:pthroot}, we draw the images under $\acham\mapsto
  \acham_{1/p}$ of all unrolled matter hyperplanes, but draw those
  which do not remain as matter hyperplanes for the $p$th root data in
  gray and with thinner weight.
  \begin{figure}\label{fig:pthroot}
    \centering
    \[\tikz[very thick,scale=1.4]{ \draw[old]
    (1.16,2.5)-- (1.16,-3.5) node[scale=.5, at
    end,below]{$\varphi_3^{\operatorname{mid}}=-1/5$}  node[scale=.5, at
    start,above]{$\varphi_1^{\operatorname{mid}}=2/5$}; \draw[old] (2.5,1.16)-- (-3.5,1.16)
    node[scale=.5, at
    end,left]{$\varphi_4^{\operatorname{mid}}=-1/5$}
    node[scale=.5, at
    start,right]{$\varphi_2^{\operatorname{mid}}=2/5$};
\draw (2.16,2.5)-- (2.16,-3.5)
    node[scale=.5, at
    end,below]{$\varphi_3^{\operatorname{mid}}=0$}
    node[scale=.5, at
    start,above]{$\varphi_1^{\operatorname{mid}}=3/5$}; \draw (2.5,2.16)-- (-3.5,2.16)
    node[scale=.5, at
    end,left]{$\varphi_4^{\operatorname{mid}}=0$}
    node[scale=.5, at
    start,right]{$\varphi_2^{\operatorname{mid}}=3/5$};
\draw[old] (0.16,2.5)-- (0.16,-3.5)
    node[scale=.5, at
    end,below]{$\varphi_3^{\operatorname{mid}}=-2/5$} 
    node[scale=.5, at
    start,above]{$\varphi_1^{\operatorname{mid}}=1/5$}; \draw[old] (2.5,.16)-- (-3.5,.16) 
    node[scale=.5, at
    end,left]{$\varphi_4^{\operatorname{mid}}=-2/5$}
    node[scale=.5, at
    start,right]{$\varphi_2^{\operatorname{mid}}=1/5$}; \draw (-.84,2.5)-- (-.84,-3.5)
    node[gray,scale=.5, at
    end,below]{$\varphi_3^{\operatorname{mid}}=-3/5$}
    node[scale=.5, at
    start,above]{$\varphi_1^{\operatorname{mid}}=0$}; \draw (2.5,-.84)-- (-3.5,-.84) 
    node[gray,scale=.5, at
    end,left]{$\varphi_4^{\operatorname{mid}}=-3/5$}
    node[scale=.5, at
    start,right]{$\varphi_2^{\operatorname{mid}}=0$}; \draw[old] (-1.84,2.5)-- (-1.84,-3.5) node[scale=.5, at
    end,below]{$\varphi_3^{\operatorname{mid}}=-4/5$} 
    node[scale=.5, at
    start,above]{$\varphi_1^{\operatorname{mid}}=-1/5$}; \draw[old] (2.5,-1.84)-- (-3.5,-1.84) 
    node[scale=.5, at
    end,left]{$\varphi_4^{\operatorname{mid}}=-4/5$}
    node[scale=.5, at
    start,right]{$\varphi_2^{\operatorname{mid}}=-1/5$}; 

    \draw (-2.84,2.5)-- (-2.84,-3.5) node[scale=.5, at
    end,below]{$\varphi_3^{\operatorname{mid}}=-1$} 
    node[gray, scale=.5, at
    start,above]{$\varphi_1^{\operatorname{mid}}=-2/5$}; \draw (2.5,-2.84)-- (-3.5,-2.84) 
    node[scale=.5, at
    end,left]{$\varphi_4^{\operatorname{mid}}=-1$}
    node[gray, scale=.5, at
    start,right]{$\varphi_2^{\operatorname{mid}}=-2/5$}; \draw[dotted] (-3.5,-3.5) -- node[scale=.5, right,at
    end]{$\alpha=0$}(2.5,2.5); \draw[dotted,gray, thin] (-3.5,-2.5) --
    node[scale=.5, below left,at
    start]{$\alpha=1/5$}(1.5,2.5); \draw[dotted,gray, thin]
    (-2.5,-3.5) -- node[scale=.5, above right,at
    end]{$\alpha=-1/5$}(2.5,1.5); \draw[dotted,gray, thin] (-3.5,-1.5)
    -- node[scale=.5, below left,at
    start]{$\alpha=2/5$}(0.5,2.5); \draw[dotted,gray, thin]
    (-1.5,-3.5) -- node[scale=.5, above right,at
    end]{$\alpha=-2/5$}(2.5,0.5); \draw[dotted,gray, thin] (-3.5,-.5)
    -- node[scale=.5, below left,at
    start]{$\alpha=3/5$}(-.5,2.5); \draw[dotted,gray, thin] (-.5,-3.5)
    -- node[scale=.5, above right,at
    end]{$\alpha=-3/5$}(2.5,-.5); \draw[dotted,gray, thin] (-3.5,.5)
    -- node[scale=.5, below left,at
    start]{$\alpha=4/5$}(-1.5,2.5); \draw[dotted,gray, thin] (.5,-3.5)
    -- node[scale=.5, above right,at
    end]{$\alpha=-5$}(2.5,-1.5); \draw[dotted] (-3.5,1.5) --
    node[scale=.5, below left,at
    start]{$\alpha=1$}(-2.5,2.5); \draw[dotted] (1.5,-3.5) --
    node[scale=.5, above right,at end]{$\alpha=-1$}(2.5,-2.5);}\]
\caption{The effect of $p$th root conventions on matter hyperplanes}
\end{figure}
\end{example}

   \notation{${\rACp_{\Ba}}$}
        {The chamber $\rACp_{\Ba}=\{\xi \in \mathsf{t}_{1,\R}\mid a_i<\varphi_i^{\operatorname{mid}}(\xi)<a_i+1\text{
  for  $i=1,\dots, d$}\}$.}
We'll want to consider the affine chambers for $\Ba\in \Z^{d}$ as in Section \ref{eq:aff-cham}: 
\[\rACp_{\Ba}=\{\xi \in \mathsf{t}_{1,\R}\mid a_i<\varphimid_i(\xi)<a_i+1\}\label{eq:r-aff-cham}\]  
\[\ACs=\{\Ba\in \Z^{d}\mid \rACp_{\Ba}\neq 0\}.\] 
  The reader should note that we use the $p$th root conventions here.
Consider $\widehat{W}$, the extended affine Weyl group, acting on
$\mathsf{t}_{1,\To,\R}$ via the usual level 1 action. Note that if $w\in
\widehat{W}$ and $\rACp_{\Ba}\neq 0$ then $w\cdot
\rACp_{\Ba}=\rACp_{w\cdot \Ba}$ for a unique $w\cdot \Ba$, so this
defines a $\widehat{W}$ action on $\ACs$.

First, we note how the polynomial representation changes when we complete it to match
$ {\widehat{\mathscr{B}}_{\upsilon'}}$.  
Given $\upsilon$, we
let $\widehat{S}^{(\upsilon)}_1=\varprojlim
 {S_1}/\mathfrak{m}_\upsilon^N$.

\begin{proposition}\label{prop:hat-rep}
  The category $ {\mathscr{\widehat{B}}_{\upsilon'}}$ has a 
  representation $\mathscr{P}$ sending $(\acham, \upsilon) \mapsto
  \widehat{S}^{(\upsilon)}_1$ defined by the formulas (\ref{eq:ract}--\ref{eq:muact}), and one $\mathscr{F}$ sending 
  \[(\acham, \upsilon)\mapsto \bigoplus \operatorname{Frac}( \widehat{S}^{(\upsilon)}_1)\cdot [t^\la] \]
  where $\operatorname{Frac}(-)$ denotes the fraction field of a commutative ring, with morphisms acting as in (\ref{eq:ract-frac}--\ref{eq:muact-frac}).
\end{proposition}
Since $\widehat{S}^{(\upsilon)}_1$ is a profinite dimensional algebra,
we can decompose any element of it into its semi-simple and nilpotent
parts, by doing so in each quotient $S_1/\mathfrak{m}_\upsilon^N$.
The grading we seek on a dense subcategory of  $\widehat{\mathscr{B}}_{\upsilon'}$ is uniquely
fixed by a small list of  requirements:
\begin{enumerate}[label=(\roman*)]
\item For $\mu \in \ft^*$, the nilpotent part $\mu-\langle\mu,\upsilon\rangle$ of $\mu$
  acting on $\widehat{S}^{(\upsilon)}_1$ is homogeneous of degree 2.
\item The action of $\widehat{W}$ is homogeneous of degree 0.
\item If $\acham\in \rACp_{\Ba}$ and $\acham'\in \rACp_{\Bb}$ then the
  obvious isomorphism   $\mathscr{P}(\acham_p,\upsilon)\cong \widehat{S}^{(\upsilon)}_1\cong \mathscr{P}(\acham'_p,\upsilon)$
is homogeneous of degree $\sum_{i=1}^d a_i-b_i$.
\end{enumerate}
The principles (i-iii) fix a grading on $\mathscr{P}(\acham,\upsilon)$
for all $\acham\in \ft_{1,\To,\R}$ and $\upsilon\in \hat{W}\cdot
\upsilon'$ up to a global shift. 

We would like to use this to fix a notion of what it means
for a morphism in $\widehat{\mathscr{B}}_{\upsilon'}$ to be
homogeneous, however this is slightly complicated by the fact that
$\mathscr{P}$ is not faithful.  However, it can still be a useful guide
to the choice of an appropriate grading.  

Let $\hat \Phi_0(\acham,\acham',\upsilon')$ be the product of the terms
$\varphi^+_i-n$  over pairs $(i,n) \in [1,d]\times\Z$ such that we have the inequalities
\[ {\varphi_i^{\operatorname{mid}}}(\acham)>n\qquad \varphi_i^{\operatorname{mid}}(\acham')<n \qquad
\langle\varphi^+_i,\upsilon'\rangle\not\equiv n\pmod p.\]  Note that these are precisely the factors in the product $\Phi(\acham,\acham')$ which remain invertible after reduction modulo $p$.  
Consider the morphisms:
  \begin{equation*}
  \wall(\acham,\acham')=\frac{1}{\hat{\Phi}_0(\acham,\acham',\upsilon')}r(\acham,\acham)
 \end{equation*}
We'll check below that these morphisms satisfy the requirements of  homogeneous morphisms from (i-iii) above, and together
with a few other obvious homogeneous morphisms, they generate a dense
subspace inside morphisms.

\notation{${\sfB}$}{The extended BFN category with $p$th root conventions (Definition \ref{def:pth-root}).}

Consider the extended BFN category $\sfB$ with $p$th root conventions.  Since $h=0$, this category is graded with   
\begin{samepage}
    \newseq\begin{equation*}\subeqn
    \label{eq:Stein-grading1}
\deg  {r(\acham,\acham')}=\deg \Phi(\eta,\eta')+\deg \Phi(\eta',\eta)
\end{equation*}\begin{equation*}\subeqn
\deg w=0\qquad \deg u_\al(\Ba)=-2 \qquad  \deg
  \mu=2\label{eq:Stein-grading2}.
\end{equation*}
\end{samepage}
Note that here $\deg \Phi(\eta,\eta')$ should be interpreted in the grading on $ {S_0}$ where $\ft^*$ is concentrated in degree 1.  
We define a functor $\gamma_{\mathsf{B}}\colon \sfB\to
 {\mathscr{\widehat{B}}_{\upsilon'}}$ by sending
$\eta\mapsto (\eta_p+\upsilon',
\upsilon')$, and acting on morphisms by:
\notation{$\gamma_{\sfB}$}{The equivalence $\sfB\cong   {\mathscr{\widehat{B}}_{\upsilon'}}$. }
\newseq 
  \begin{align*}
    \gamma_{\mathsf{B}}( {r(\acham,\acham')})&=\wall(\eta_p+\upsilon',\eta_p'+\upsilon')= \frac{1}{\hat{\Phi}_0(\eta_p+\upsilon',\eta_p'+\upsilon',\upsilon')}r(\eta_p+\upsilon',\eta_p'+\upsilon')\subeqn \label{gamma1}\\
    \gamma_{\mathsf{B}} (w)&= w_p \subeqn \label{gamma2}\\
    \gamma_{\mathsf{B}}(u_\al) &=u_{\al^{(p)}}\subeqn \label{gamma3}\\
    \gamma_{\mathsf{B}}(\mu)&= \mu-\langle\mu,\upsilon'\rangle \subeqn \label{gamma4}
  \end{align*}
Note that since $w_p\acham_p=(w\acham)_p$ and $w_p\cdot \upsilon'=\upsilon'$, these morphisms go
between the correct objects.

\begin{remark}
  Just as discussed in \cite{websterKoszulDuality2019} below the proof of Lemma \ref{SD-lem:HC-functor}, this isomorphism has a natural
  geometric interpretation as localization to the fixed points of a
  group action.  In the characteristic zero case, we analyze the space
  corresponding to a weight in terms of the fixed points of the
  corresponding character; a version of this is explained in
  \cite[Th. 4.4]{WebGT}, generalizing work of Varagnolo-Vasserot \cite[\S
  2]{varagnoloDoubleAffine2010}. In characteristic $p$, we only obtain an isomorphism
  after completion with the fixed points of the $p$-torsion subgroup
  of this cocharacter.  Of course, these fixed points again give
  versions of the spaces appearing in the extended BFN category.
  While this is a beautiful perspective, we think it will be clearer,
  especially for the reader unused to the geometry of the affine
  Grassmannian, to give an algebraic proof.
\end{remark}

  \begin{proposition}\label{prop:B-equiv}
    The functor $\gamma_{\mathsf{B}}\colon \sfB\to  {\mathscr{\widehat{B}}_{\upsilon'}}$ is faithful, topologically full and essentially surjective; that is, it induces an equivalence $\sfBhat\cong  {\widehat{\mathscr{B}}_{\upsilon'}}$
  \end{proposition}

  \begin{proof}
  Consider the representation $\mathscr{F}$ of  $\widehat{\mathscr{B}}_{\upsilon'}$.  In order to confirm the result, we must show that the images above under $\gamma_{\mathsf{B}}$ satisfy the relations of $\sfB$ and define a faithful representation on $\mathscr{F}$. 
  
  The formula \ref{gamma4} identifies each summand of $\mathscr{F}$ with the fraction field of the completion of $\widehat{S}_0^{(0)}$ at the origin of $S_0$, and thus $\mathscr{F}(\eta_p,\upsilon')\cong \bigoplus \widehat{S}_0^{(0)}\cdot [t^\la]$.  Consider the images of the right-hand sides of (\ref{gamma1}--\ref{gamma4}) under transport of structure.  We wish to show that this agrees with the representation $\mathsf{K}$ of $\mathsf{B}$ defined in Lemma \ref{lem:frac-rep}.  
  
For the morphism $\gamma_{\mathsf{B}}(\mu)$, this is automatic.  For $\gamma_{\mathsf{B}}(w)$, this is an immediate consequence of the definition.  
  
Now, consider $r(\eta,\eta')$.  In its polynomial
representation, this element acts by
${\mathsf{\Phi}(\acham,\acham')}$, which is the product of $\varphi_i^+$ raised to the number of
integers $n$ satisfying ${\varphi}^{\operatorname{mid}}_i(\acham)< n<
{\varphi}^{\operatorname{mid}}_i(\acham')$.  This is sent under
$\gamma_{\mathsf{B}}$ to the shift $\varphi_i^+-\langle
\varphi^+,\upsilon'\rangle$, leaving the overall structure of the
product the same.  

Now note that by \eqref{eq:p-mid}, we have that
\[{\varphi}^{\operatorname{mid}}_i(\acham_p+\upsilon') =p \varphi^{\operatorname{mid}}_i(\acham)+\langle
\varphi^+,\upsilon'\rangle\qquad \langle \varphi^+,\acham_p+\upsilon'\rangle =p \langle \varphi^+,\acham\rangle+\langle
\varphi^+,\upsilon'\rangle.\]
Thus, we have that
\begin{equation}
 {\varphi_i^{\operatorname{mid}}}(\acham_p+\upsilon')>n\qquad \varphi_i^{\operatorname{mid}}(\acham_p'+\upsilon')<n \qquad
  \langle\varphi^+_i,\upsilon'\rangle\equiv n\pmod p\label{eq:hyperplanes1}
\end{equation}

if and only if we have that
$n_{1/p}=\frac{1}{p}(n-\langle\varphi^+_i,\upsilon'\rangle)\in \Z$ and 
\begin{equation}
 {\varphi_i^{\operatorname{mid}}}(\acham)>n_{1/p}\qquad
\varphi_i^{\operatorname{mid}}(\acham')<n_{1/p}\label{eq:hyperplanes2}
\end{equation}

Thus, we have an
equality  \[{\mathsf{\Phi}(\acham,\acham')}=\frac{{\Phi}(\acham_p-\upsilon',\acham'_p-\upsilon')}{\hat{\Phi}_0(\acham_p-\upsilon',\acham_p-\upsilon',\upsilon'')}\]
since the factors of the RHS are given by \eqref{eq:hyperplanes1} and
of the LHS by \eqref{eq:hyperplanes2}.
Note the difference between $\mathsf{\Phi}$ and $\Phi$ in the equation above, denoting the use of the $p$th root conventions on the left-hand side.  Thus, in $\bigoplus \widehat{S}_0^{(0)}$, we see that $\gamma_{\mathsf{B}}(r(\eta,\eta'))=\wall(\eta_p,\eta_p')$ acts by multiplication by $\mathsf{\Phi}(\acham,\acham')$, as expected.  Finally, $\gamma_{\mathsf{B}}(u_\al)$ must have the desired image because it can be written as $\frac{1}{\al}(s_{\al}-r(s_\al\acham,\acham))$ when it is well-defined and $\al$ is invertible in $\operatorname{Frac}(\widehat{S}_0^{(0)}).$  This shows that we have recovered $\mathsf{F}$.

Since $\mathsf{F}$ is a faithful representation, this shows that $\gamma_{\mathsf{B}}$ is well-defined and faithful.

Any object $(\acham, w\cdot \upsilon')$ is in the essential image of
the functor, since it was isomorphic to $(w^{-1}\cdot \acham,
\upsilon')$, and the map $(\cdot)_p$ is a bijection. 

Finally, we need to show that the image of the functor is dense.
Note that:   \begin{align*}
   {r(\acham,\acham')}  &=\hat{\Phi}_0(\acham,\acham',\upsilon')\wall(\acham,\acham')=
                            \hat{\Phi}_0(\acham,\acham',\upsilon')\gamma_{\mathsf{B}}(r((\acham-\upsilon')_{1/p},(\acham'-\upsilon')_{1/p}))\\
   w&=  \gamma_{\mathsf{B}} (w_{1/p})\\
   u_{\al^{(p)}}&=\gamma_{\mathsf{B}}(u_\al) \\
   \mu &=  \gamma_{\mathsf{B}}(\mu+\langle\mu,\upsilon'\rangle). 
  \end{align*}
We are only left the task of showing that $u_{\al}$ is
in the closure of the image of $\gamma_{\mathsf{B}}$ for $s_{\al}\notin \widehat{W}_{\upsilon'}$.  In this case, $\al$
thought of as an element of $S_1$ will act invertibly, so $1/\al$ lies in the closure of the image.   Thus, we can use the
formula $\frac{1}{\al}(s_{\al}-r(s_\al\acham,\acham))$ as before. This shows the density.
  \end{proof}

\begin{remark}\label{rem:coefficients}
  We should emphasize that the functor $\gamma_{\mathsf{B}}$ is only well-defined over a field of characteristic $p$, but the category $\sfB$  makes sense with coefficients in  an arbitrary commutative ring (in
  particular, $\Z$).  We'll write $\sfB(\K)$ when we wish to emphasize the choice of base field.  
\end{remark}
    
 \notation{${\mathsf{A}}$}{The subcategory  whose objects are $\upsilon'+\ft_{1,\Z}$ whose morphisms match those of $\sfB$ after reindexing (Definition \ref{def:sfA}).}
 \begin{definition}\label{def:sfA}
  We let $\sfA(\K)$ be the category whose objects are 
 the elements of $\upsilon'+\ft_{1,\Z}$, with morphisms 
 $\Hom_{ \mathsf{A}(\K)}(\xi,\xi')\cong
 \Hom_{\mathsf{B}(\K)}((\zero-\xi)_{1/p}, (\zero-\xi')_{1/p})$, and
 $\sfAhat (\K)$ its completion with respect to the grading. 
\end{definition}
 Note that we have broken a bit from our convention of using sans-serifizing (which would suggest that $\mathsf{A}$ should be the ring $\EuScript{A}_0$); the reason for this will be clearer below.

Since objects in $\mathsf{A}(\K)$ that differ by the level $p$ action of $\What$ are
isomorphic, we could also consider only the elements  of
$-\upsilon'+\ft_{1,\Z}$ in a fundamental region for the action of this group.  However, it's more convenient to describe the elements $\wall$ in the unrolled picture.

Note that in $ {\widehat{\mathscr{B}}_{\upsilon'}}$,  the translation by
$\upsilon'-\upsilon $ induces an isomorphism
$(\zero ,\upsilon)\cong (\zero-\upsilon+\upsilon', \upsilon')$.  
Conjugating the functor $\gamma_{\mathsf{B}}$ by this isomorphism induces a functor $\gamma\colon
\widehat{\mathsf{A}} \to\widehat{\mathscr{A}}_{\upsilon'} $ sending $\gamma(\upsilon)=(\zero,\upsilon)$.  Proposition \ref{prop:B-equiv} immediately implies that:
\begin{theorem}\label{thm:pStein-equiv}
  The categories $\mathscr{\widehat{A}}_{\upsilon'}$ and
  $\widehat{\mathsf{A}}$ are equivalent via the functors
 \[\tikz[->,thick]{
\matrix[row sep=12mm,column sep=35mm,ampersand replacement=\&]{
\node (d) {$\widehat{\mathsf{A}}$}; \& \node (e)
{$\sfBhat$}; \\
\node (a) {$\widehat{\mathscr{A}}_{\upsilon'}$}; \& \node (b)
{$\widehat{\mathscr{B}}_{\upsilon'}$}; \\
};
\draw (a) -- (b) ; 
\draw (d) -- (a) node[left,midway]{$\gamma$} ; 
\draw (e) -- (b) node[right,midway]{$\gamma_{\mathsf{B}}$}; 
\draw (d) -- (e) node[above,midway]{$\upsilon\mapsto (\zero-\upsilon)_{1/p}$}; 
}\]
\end{theorem}
Note that the square commutes when applied to objects since
 \[\tikz[->,thick]{
\matrix[row sep=12mm,column sep=35mm,ampersand replacement=\&]{
\node (d) {$\upsilon$}; \& \node (e)
{$(\zero-\upsilon)_{1/p}$}; \\
\node (a) {$(\zero,\upsilon)$}; \& \node (b)
{$(\zero,\upsilon)\cong (\zero-\upsilon+\upsilon',\upsilon')$}; \\
};
\draw (a) -- (b) ; 
\draw (d) -- (a) ; 
\draw (e) -- (b); 
\draw (d) -- (e); 
}\]

\subsection{Consequences for representation theory}
\label{sec:cons-repr-theory}
Theorem \ref{thm:pStein-equiv} on its own has a quite interesting
consequence for the behavior of the finite-dimensional
representations of $ {\EuScript{A}_\phi}$ for
different primes $p$ with  $\flav/p$  ``held constant.''  For simplicity in this section, we only consider the case where $\K=\mathbb{F}_p$, though the results could be
generalized without much difficulty.  As we discussed in Remark \ref{rem:coefficients}, the
category $\sfB$ has relations which are independent of $p$.
This allows us to compare the representations of $\EuScript{A}_\phi$,
by matching them with the representations of $\mathsf{B}.$

     \notation{${\Lambda}$}
    {The set of vectors such that $\rACp_{\Ba}$
  contains $\xi_{1/p}$ for $\xi\in \ft_{1,\Z}$.}
 \newcommand{\barLambda}{\bar{\Lambda}}
\begin{definition}\label{def:Lambda}
  Let $\Lambda\subset \Z^{d}$ be the vectors such that $\rACp_{\Ba}$
  contains $(\zero-\xi)_{1/p}$ for $\xi\in \phi+\ft_{1,\Z}$.   Let $\barLambda$ be the quotient of this set by the action of $\What$ on $\Lambda$.
\end{definition}
Note that if we wish to compare different primes, it is perhaps more natural to consider the dependence of
this set on
$\psi=\phi_{1/p}$.  The sets $\Lambda,\bar{\Lambda}$ 
  make sense for an arbitrary $\psi \in \R^{d}$.

The set $\bar \Lambda$ is finite; its size is bounded above by the number
of collections of weights $\vp_i$ which form a basis of $\ft^*$.   We
can then divide up choices of $\phi$ according to what the
corresponding set $\bar \Lambda$ is.  Given  $\Ba\in \Z^d$, we let
$\bar{\Ba}$ be its $\widehat{W}$-orbit.
\notation{ $ {\mathsf{B}^{\bar \Lambda}(\K)}$
}{The category with objects given by chambers in $\barLambda$, and morphisms by morphisms between points in the relevant chambers in $\sfB$.}
\begin{definition}\label{def:BLam}
  For a fixed $\barLambda$, we let $ {\mathsf{B}^{\bar \Lambda}(\K)}$
  for any commutative ring $\K$ be the category with object set
  $\bar \Lambda$ such that
  $\Hom_{\mathsf{B}^{\bar
      \Lambda}(\K)}(\bar{\Ba},\bar{\Bb})=\Hom_{\mathsf{B}(\K)}(\eta_\Ba,\eta_\Bb)$,
  for $\eta_\Ba$ an arbitrary element of the chamber $\rACp_{\Ba}$.

  We can also encapsulate this in the graded ring
  \begin{equation}
  A(\K)=\bigoplus_{\bar{\Ba},\bar{\Bb}\in
    \bar\Lambda^{\R}}\Hom_{ {\mathsf{B}^{\bar{\Lambda}^{\R}}(\K)}}(\bar{\Ba},\bar{\Bb}),\label{eq:A-def}
\end{equation}
\notation{$A$}{The noncommutative resolution of $\fM$ constructed by quantization.  Also, the sum of morphisms in $ {\mathsf{B}^{\bar \Lambda}(\Q)}$.}
\end{definition}
Let $e(\bar{\Ba})\in A(\K)$ be the identity morphism on $\bar{\Ba}$.  Since the center $Z(A(\K))$ is of finite
codimension, the algebra $A(\K)$ has finitely many graded simple
modules, all of which are finite-dimensional.  Each such simple for
$\K=\Q$ has a $\Z$-form, which remains irreducible mod $p$ for all but finitely
many $p$.  That is:
\begin{theorem}\label{thm:which-polytope}
For a fixed $\bar\Lambda$ and  $p\gg 0$, there is a bijection
between homogeneous simple $A(\K)$-modules $L$ and simple finite-dimensional $\EuScript{A}_{\phi}$-modules $L(p)$ for all $\flav\in \Z^d$ with $\bar
\Lambda(\flav)=\bar \Lambda$.   Under this bijection,  each weight space of
$L(p)$ for a weight $\upsilon$ with $\upsilon_{1/p}\in \rACp_{\Ba}$ is the same
dimension as the $\Q$-vector space $L(\Ba)=e(\bar \Ba)L$.  
\end{theorem}
\begin{proof}
  As discussed above, for $p\gg 0$, the simple graded representations of $A(\K)$ are given by reductions mod $p$ of an arbitrary invariant lattice of the simples $L$ of $\mathsf{B}^{\bar \Lambda}(\Q)$.  This clearly preserves the dimension of the vector space assigned to an object $\Ba$.  Under the equivalence of Theorem \ref{thm:pStein-equiv}, the $\upsilon$ weight space of a $\EuScript{A}_{\phi}$-module matches the vector space assigned $\Ba$ defined as before in the $\mathsf{B}^{\bar \Lambda}({\mathbb{F}_p})$-module.  
\end{proof}
Thus, the dimension of $L(p)$ only depends on the number of weights of
$\EuScript{A}_{\phi}$  with $\upsilon_{1/p}\in \rACp_{\Ba}$: it is the sum of the
dimensions of the 
spaces $L(\Ba)$ weighted by this count of integral points in a
polytope.  By the usual quasi-polynomiality of Erhart polynomials, we have that:
\begin{corollary} Fix $\barLambda$ and  let $p$ and $\phi$ vary over values where $\bar{\Lambda}(\phi)=\bar{\Lambda}$.  
  For $p\gg 0$, the dimension of $L(p)$ is a quasi-polynomial function of $\psi=\flav_{1/p}$ and
  $p$.  
\end{corollary}

 \section{Relation to geometry}
\label{sec:geometry}
Now, we turn to relating this approach to the study of coherent
sheaves on resolved Coulomb branches.   Recall that $\Coulomb$ is an affine algebraic variety, and in this section, we will study a BFN resolution of this variety $\tM$, given by the GIT quotient of $\MQ$ with respect to a character $\chi\colon K\to \C^*$.

Throughout this section, we will
take the convention that $\widehat{\mathscr{A}}_*$
or $\EuScript{A}_*$ with $*\in \{h,0,1\}$ denotes the category $ {\mathscr{A}_\phi}$ or algebra
$ {\EuScript{A}_\phi}$ with $\phi$ left implicit, and $h$ left as a formal variable, or specialized to
be $0$ or $1$ (depending on the subscript).  

\subsection{Frobenius constant quantization}

Recall that a quantization $R_h$ of a $\K$-algebra $R_0$ is called {\bf Frobenius constant} if there is a multiplicative map $\sigma\colon
R_0\to R_h$ congruent to the Frobenius map
modulo $h^{p-1}$.  

In the case of the quantum Coulomb branch, the Frobenius constancy of the quantization was recently proven by Lonergan.
\begin{theorem}[\mbox{\cite[Thm. 1.1]{lonerganSteenrodOperators2021}}]\label{thm:lonergan}
  There is a ring homomorphism $\sigma\colon  {\EuScript{A}_0^{\operatorname{sph}}}\to   {\EuScript{A}_h^{\operatorname{sph}}}$
  making $\EuScript{A}_h^{\operatorname{sph}}$ into a FCQ for $\EuScript{A}_0^{\operatorname{sph}}$.
\end{theorem}
Since Lonergan's construction is quite technical, it's worth reviewing the actual map that results.    If $G$ is abelian, then we can write this morphism very explicitly:  in this case, we consider $\EuScript{A}_h^{\operatorname{sph}}$ as $\End_{\mathscr{B}}(\second)$, and this space is spanned over $S_h$ by the elements $r_\nu=y_{\nu}r(-\nu,0)$ and \cite[\S 3.15(3)]{lonerganSteenrodOperators2021} shows that
  it is induced by 
\begin{align}
    \vp &\mapsto \AS(\vp)=\vp^p-h^{p-1}\vp\\
     r_\nu&\mapsto r_{p\nu}
\end{align} 
We can rewrite the action of the polynomial
$\Phi(\acham+p\gamma,\acham)$ for $\gamma\in \ft_{\Z}$ using this map:
this is a product of consecutive factors $\vp_i^+-kh$ for $k\in \Fp$,
and must range over a number of these factors divisible by $p$.
Furthermore, the number of such factors is $\vp_i(\gamma)p$ if
$\vp_i(\gamma)\geq 0$ and $0$ otherwise.  That
is,
\[\Phi(\acham+p\gamma,\acham)=\prod_{i=1}^d
\AS(\vp^+_i)^{\operatorname{max}(\vp_i(\gamma),0)}\]
Having noted this, it is a
straightforward calculation that this is a ring homomorphism.

If $G$ is non-abelian, then this homomorphism is induced by the
inclusion of $\EuScript{A}_0^{\operatorname{sph}}$ and $\EuScript{A}_h^{\operatorname{sph}}$ into the
localization of the Coulomb branch algebras for the maximal torus $T$ by inverting $\al$ for all affine
roots $\al$, since Steenrod operations commute with pushforward from the $T$-fixed locus, as discussed in \cite[\S 3.15(4)]{lonerganSteenrodOperators2021}.

A natural property to consider for varieties in characteristic $p$ is whether they are Frobenius split.  For the abelian case, it's easy to construct a splitting.  Let $\kappa_0\colon S_0\to S_0$ be any homogeneous Frobenius splitting.
\begin{proposition}\label{prop:abelian-splitting}
  The map
  \[\kappa(f\cdot r_\la)=
    \begin{cases}
      \kappa_0(f) r_{\la/p} & \la/p\in \ft_\Z\\
      0 & \la/p\notin \ft_\Z
    \end{cases}\]
  is a Frobenius splitting for the ring $\K[\fM]$ when $G$ is abelian.
\end{proposition}
\begin{proof}
  This map is obviously a homomorphism of abelian groups sending 1 to 1, so we need only show that $\kappa(a^pb)=a\kappa(b)$ in the case where $a$ and $b$ are both of the form $a=f\cdot r_\la$ and $b=g\cdot r_\mu$.  This is easy to see, since $r_\la^p=r_{p\la}$ and 
  \[\kappa(f^p r_{p\la}\cdot g r_\mu )=\kappa(f^pg \Phi(-p\la-\mu,-\mu,0)\cdot r_{p\la+\mu})\]
  If $\mu$ is not $p$-divisible, then this expression is 0, as is $fr_\la\kappa(gr_\mu)=0$, so the result holds.  On the other hand, if $\mu/p\in \ft_\Z$, then
  \[\kappa(f^pg \Phi(-p\la-\mu,-\mu,0)\cdot r_{p\la+\mu})=f\kappa_0(g) \Phi(-\la-\frac{\mu}{p},\frac{\mu}{p},0)r_{\la+\frac{\mu}{p}}=fr_{\la}\kappa(g r_\mu )\] as desired.  
\end{proof}

Now, assume that $G$ is non-abelian and that the map $\kappa_0$ is equivariant for the group $ \weylW$; as usual this is possible because the average of the $W$-conjugates of a Frobenius splitting is again a splitting.  Recall from \cite[Def. \ref{SD-def:r-tilde}]{websterKoszulDuality2019} that we have an element $\tilde{r}_\pi$ for any path $\pi$; let us write $\tilde{r}(\acham,\acham')$ for the straight line path from $\acham'$ to $\acham$.

By \cite[Prop. \ref{SD-prop:dressed-basis}]{websterKoszulDuality2019}, the algebra  $ {\EuScript{A}_0^{\operatorname{sph}}}$ has a basis given by the dressed monopole operators: the elements \[\mathbbm{m}_{\la}(f)=y_{\la}\tilde{{r}}(-\la, -\la\epsilon) f \tilde{r}(-\la \epsilon,0)\] for $\epsilon>0$ a very small real number, $\la$ running over dominant coweights of $G$ and $f$ over a basis of $S_0^{W_\la}$;
we only need dominant coweights because \[\mathbbm{m}_{\la}(f) =y_{w\la}\tilde{r}(-w\la,-w\la \epsilon) f^w \tilde{r}(-w\la \epsilon,0)\] for any $w\in  \weylW$.

\begin{proposition}\label{prop:nonabelian-splitting}
  There is a Frobenius splitting $\kappa\colon \K[\fM]\to \K[\fM]$ such that
  \begin{equation}
    \kappa(\mathbbm{m}_{\la}(f))=
    \begin{cases}
      \mathbbm{m}_{\la/p}(\kappa_0(f)) & \la/p\in \ft_\Z\\
      0 & \la/p\notin \ft_\Z
    \end{cases}\label{eq:nonabelian-splitting}
  \end{equation}
\end{proposition}
\begin{proof}
  Consider the usual inclusion $\K[\fM]\to \K[\fM_{\operatorname{ab}}^0]^W$ where $\fM_{\operatorname{ab}}^0$ is the open subset of $\fM_{\operatorname{ab}}$ where the root functions are non-vanishing.    The former is Frobenius split by Proposition \ref{prop:abelian-splitting}, and the restriction of the splitting map to $\K[\fM]$ acts by \eqref{eq:nonabelian-splitting}.  In particular, it preserves the subring $\K[\fM]$ and thus gives a Frobenius splitting.

  In order to do this calculation, it is useful to note that $(\mathbbm{m}_{\la}(1))^p=\mathbbm{m}_{p\la}(1)$, so this shows the result when $f=1$.  There are elements $h_w\in  \K[\fM_{\operatorname{ab}}^0]^W$ such that
  \[\mathbbm{m}_{\la}(f)=\sum_w h_w (w\cdot f)\qquad \mathbbm{m}_{p\la}(f)=\sum_w h_w^p(w\cdot f) .\]  Thus, we have that \[\kappa(\mathbbm{m}_{p\la}(f))=\sum_w h_w \kappa_0(w\cdot f)=\mathbbm{m}_{\la/p}(\kappa_0(f)).\qedhere \]
\end{proof}
If $G$ has non-trivial $\pi_1$, then this splitting is obviously equivariant for the induced action of the Pontryagin dual of $\pi_1$, and thus descends to the GIT quotient.  Since any (partial) BFN resolution is a GIT quotient of this form, we thus also have that:
\begin{corollary}\label{cor:BFN-split}
  Any partial BFN resolution is Frobenius split.  
\end{corollary}

There are two natural ways to view $\EuScript{A}_h^{\operatorname{sph}}$ as a sheaf of algebras on $\fM=\Spec \EuScript{A}_0^{\operatorname{sph}}$:
\begin{enumerate}
\item The first is the usual microlocalization $\salg$ of
  $\EuScript{A}_h^{\operatorname{sph}}$. The sections $\salg(U_f)$ on the
  open set $U_f$ where $f$ is non-vanishing are given by
  $\EuScript{A}_h^{\operatorname{sph}}$ with every element congruent
  to $f$ mod $h$ inverted.  This construction is discussed, for example, in \cite[\S 4.1]{BLPWquant}.  This is a quantization in the usual sense
  of \cite{BKpos}, and thus {\it not} a coherent sheaf.

  \notation{${\psalg_\phi}$}
      {The coherent sheaf of generically Azumaya algebras on $\tM$ or
   its pushforward to $\Coulomb$,  such that
   $\Gamma(\tM; \psalg_\phi)= {\EuScript{A}_1 ^{\operatorname{sph}}
   }$ with the quantization parameter $\phi$; see Definition \ref{def:psalg}.}
\item On the other hand, we can use $\sigma$ to view  $\EuScript{A}_h^{\operatorname{sph}}$ as a finite $\EuScript{A}_0^{\operatorname{sph}}[h]$-algebra, by the finiteness of the Frobenius map.  We'll typically consider the specialization at
$h=1$, which realizes $\EuScript{A}_1^{\operatorname{sph}}$ as a finitely generated $\EuScript{A}_0^{\operatorname{sph}}$-module.  Let $\psalg$ be the corresponding coherent sheaf on $\Coulomb=\Spec \EuScript{A}_0^{\operatorname{sph}}$.  This is essentially the 
pushforward of the usual microlocalization by the
Frobenius map, specialized at $h=1$.
\end{enumerate}

  \notation{${\Isalg_\phi}$}
      {The coherent sheaf of generically Azumaya algebras on $\tM$ or
   its pushforward to $\Coulomb$,  such that
   $\Gamma(\tM; \Isalg_\phi)= {\EuScript{A}_1 
   }$ with the quantization parameter $\phi$; see Definition \ref{def:psalg}.}
The sheaf of algebras $\psalg$ is an Azumaya
algebra on the smooth locus of $\fM$  of degree $p^{\operatorname{rank}(G)}$ by \cite[Lemma 3.2]{BKpos}.  We can also localize the algebra $\EuScript{A}_1$ using the map $\sigma$, and obtain an algebra $\Isalg$ which on the smooth locus is Azumaya of degree $p^{\operatorname{rank}(G)}\cdot \#W$; the spherical idempotent in $\EuScript{A}_1$ induces a Morita equivalence between the Azumaya algebras $\psalg$ and $\Isalg$.

Note that up to this point we have only obtained coherent sheaves on the affine variety $\fM$, but we will be more interested in considering the resolution $\tM$.  By assumption, this resolution is the Hamiltonian reduction of the  Coulomb branch $\MQ$ of $\To$ by $\groupK=T_F^\vee$.  This Hamiltonian action of $\gK$ is quantized by a noncommutative moment map $U(\gk)\to \EuScript{A}_{1,\To}^{\operatorname{sph}}$.  Let \[\EuScript{Q}_h=\EuScript{A}_{h,\To}^{\operatorname{sph}}/\gk \cdot (\EuScript{A}_{h,\To}^{\operatorname{sph}});\] by \cite[3(vii)(d)]{BFN} and \cite[Lem. \ref{SD-lem:q-mm}]{websterKoszulDuality2019}, we then have that
\begin{equation}
\EuScript{A}_{h}^{\operatorname{sph}}=\End_{\EuScript{A}_{h,\To}^{\operatorname{sph}} }(\EuScript{Q}_h)^K\cong \EuScript{Q}_h^K.\label{eq:qham}
\end{equation}
Thus, we can follow the usual yoga for constructing quantizations of Hamiltonian reductions (see \cite[4.3]{Stadnik} for a discussion of doing this reduction for a torus in characteristic $p$, and \cite[\S 2.5]{KR07} for a more general discussion in characteristic $0$)  to obtain a Frobenius constant quantization of the
resolved Coulomb branch $\tM$.  We'll give an alternate construction of this quantization below using $\Z$-algebras.

\begin{definition}\label{def:psalg}
Pushing forward by the
Frobenius map and specializing $h=1$ as above, we obtain a coherent
sheaf of algebras, also denoted by $\psalg$ which is Azumaya on the smooth locus of
$\tilde{\fM}$.  We can perform the analogous operation with $\EuScript{A}_h ^{\operatorname{sph}}$ replaced by  $\EuScript{A}_h$.  As before, we denote this by $\Isalg$.  
\end{definition}

In particular: \begin{lemma}
  If $\tM$ is smooth, then $\psalg$ is an Azumaya algebra of degree $p^{\operatorname{rank}(G)}$ and $\Isalg$ is Azumaya of degree $p^{\operatorname{rank}(G)}\cdot \#W$.  
\end{lemma}

\subsection{Homogeneous coordinate rings}
\label{sec:homogeneous}
While this discussion is quite abstract, we can make it much more concrete by thinking about $\tM$ in terms of its homogeneous coordinate ring.

The variety $\tilde{\fM}$ is a GIT quotient of the moment map level $\mu^{-1}(0)\subset \MQ$ with respect to some character $\chi\colon \groupK\to \mathbb{G}_m$. We do not assume for now that this is smooth.  Note that in our notation, we have that 
  \begin{align*}
    \Fp[\MQ]&=\EuScript{A}_{0,\To}^{\operatorname{sph}}\\
    \Fp[\mu^{-1}(0)]&= {\EuScript{Q}_0}=\EuScript{A}_{0,\To}^{\operatorname{sph}}/\Big(\mu^*(\gk)\cdot(\EuScript{A}_{0,\To}^{\operatorname{sph}})\Big)\\
    \Fp[\fM]&=\EuScript{Q}_0^K= (\EuScript{A}_{0,\To}^{\operatorname{sph}})^K/\Big(\mu^*(\gk)\cdot(\EuScript{A}_{0,\To}^{\operatorname{sph}})^K\Big)
  \end{align*}
  where $\gk$ is thought of as the space of linear functions on $\gk^*$, and $\mu^*$ is pullback by the moment map.    
  By definition, we have that the section space of powers of the canonical ample bundle on the GIT quotient is given by the semi-invariants for $\chi^n$:
  \begin{equation*}
    \Gamma(\tilde{\fM};\cO(n))\cong \EuScript{Q}_0^{\chi^n} =\{q\in  {\EuScript{Q}_0}\mid a^*( q)=\chi^n(k)q \} 
  \end{equation*}
  for $a\colon K\times \mu^{-1}(0)\to \mu^{-1}(0) $ the action map. Since we are working in characteristic $p$, we need to phrase semi-invariance in terms of pullback of functions; it is necessary but not sufficient to check that $k\cdot q=\chi^n(k)q$ for points of the group $K$.  Of course, we have, by definition, that
  \begin{equation}
T\cong \bigoplus_{m\geq 0}\Gamma(\tilde{\fM};\cO(m))\cong \bigoplus_{m\geq 0}\EuScript{Q}_0^{\chi^m}\qquad \tilde{\fM}=\operatorname{Proj}(T).\label{eq:proj-coord}
\end{equation}

  Let us describe the quantum version of this structure.  It is tempting to simply change $h=0$ in \eqref{eq:proj-coord} to $h=1$; unfortunately, this doesn't result in an algebra
  or a module over the projective coordinate ring.  Instead, $\EuScript{Q}_1^{\chi^m}= {{}_{\phi+m\nu}T^{\:\operatorname{sph}}_\phi}$ is the twisting bimodule associated to the derivative $\nu=d\chi\in \mathfrak{k}_\Z^*\cong \ft_\Z$.   With a bit more care, we could modify this structure to a $\Z$-algebra as discussed in \cite[\S 5.5]{BLPWquant}.

However, being in characteristic $p$ and having a Frobenius map gives us a second option.  The quantum Frobenius map $\sigma$ sends $\chi$-semi-invariants to $\chi^p$-semi-invariants, and thus induces a graded $T$-module structure on the graded algebra \[\EuScript{T}^{\operatorname{sph}}:=\bigoplus_{m\geq 0}\EuScript{Q}_1^{\chi^{pm}}=\bigoplus_{m\geq 0} {{}_{\phi+pm\nu}T^{\:\operatorname{sph}}_\phi}.\]
It's easy to see that the associated graded of this noncommutative algebra is \[\bigoplus_{m\geq 0}\Gamma(\tilde{\fM};\cO(pm)),\] with $T$ acting by the twist of the obvious action by the Frobenius. In particular, $\EuScript{T}^{\operatorname{sph}}$ is finitely generated over $T$ by the finiteness of the Frobenius map.
This allows us to give our more ``hands-on'' definition of $\psalg$.
\begin{definition}
 Let $\psalg$ be the coherent sheaf of algebras on $\tM$ induced by $\EuScript{T}^{\operatorname{sph}}$.  That is, $\psalg= {\EuScript{Q}_1^{\chi^{pN}}}\otimes_{\EuScript{A}_0^{\operatorname{sph}}}\mathcal{O}(-N)$ for $N\gg 0$.  
\end{definition}
This sheaf stabilizes for $N$ sufficiently large because of the finite generation of $\EuScript{T}^{\operatorname{sph}}$; thus multiplication is induced by the graded multiplication on $\EuScript{T}^{\operatorname{sph}}$ and on $T$.   It follows immediately from standard results on projective coordinate rings that:
\begin{corollary}
  The functor $\mathcal{F}\mapsto \bigoplus_{m\geq 0}\Gamma(\tilde{\fM},\mathcal{F}(m))$ induces an equivalence between the category of coherent $\psalg$-modules and the category of graded finitely generated $\EuScript{T}^{\operatorname{sph}}$-modules modulo those of bounded degree.
\end{corollary}
As with the other structures we have considered, we can remove the superscripts of $\operatorname{sph}$.  This can be done from first principles, reconstructing all the objects defined above, but we ultimately know that the result will be Morita equivalent to the spherical version, so we can define it more quickly.  Consider the tensor product $\EuScript{T}^{\operatorname{sph}}\otimes_{\EuScript{A}_1^{\operatorname{sph}}}e_{\operatorname{sph}}\EuScript{A}_1$, which is just a free module of rank $\# \weylW$, and let $\EuScript{T}$ be the endomorphism algebra of this module.  We let $\Isalg$ be the corresponding algebra of coherent sheaves.

\subsection{Infinitesimal splittings}

Assume now that $\tM$ is smooth and a resolution of $\fM$, that is, a BFN resolution.  Recall that we have a map $\tilde{\fM}\to \ft/ \weylW$ induced by the inclusion of $S_0^W$ into
$\EuScript{A}_0^{\operatorname{sph}}$.
\begin{definition}
  We let $\widehat{\tM}$ be the formal neighborhood of the fiber over the origin in $\ft/W$, that is the formal scheme obtained by completing at the schematic fiber.\notation{$\widehat{\tM}$}{The formal neighborhood of the fiber over the origin in $\ft/W$.}

  Let $ {\hat{\psalg}_\phi}$ be the corresponding pullback of $  {\psalg_\phi}$, let $ {\hat{\Isalg}_\phi}$
be the corresponding pullback of $  {\Isalg_\phi}$ and similarly, $\hat{\EuScript{A}}_\phi$ the corresponding completion of $ \EuScript{A}_{\phi}$.
\end{definition}

The algebra $\hat{\Isalg}_\phi$ can be written as the inverse limit
\[ {\hat{\Isalg}_\phi}=\varprojlim {\Isalg}_\phi/{\Isalg}_\phi\mathfrak{m}^N\] for
$\mathfrak{m}\subset S_0^W$ the maximal ideal corresponding to the
origin.  Of course, $\hat{\Isalg}_\phi$ contains the larger commutative
subalgebra $\hat{S}_1=\varprojlim {S}_1/{S}_1\mathfrak{m}^N$ so we can consider how this profinite-dimensional algebra acts on
${\Isalg}_\phi/{\Isalg}_\phi\mathfrak{m}^N$.

As is well-known, an element $a\in \K$ satisfies $a^p-a=0$ if and only if $a\in \mathbb{F}_p$.
This extends to show that in
$S_1$, the ideal $\mathfrak{m}S_1$ has radical given by the intersection of the maximal ideals $\mathfrak{m}_{\mu}$ defined by the points in $\mu \in \ft_{1,\Fp}$.  Thus, $\hat{S}_1$ breaks up as the sum of the completions at these individual maximal ideals.  For a given $\mu\in \ft_{1,\Fp}$ let $e_{\mu}$ be the idempotent that acts by 1 in the formal neighborhood of $\mu$ and vanishes everywhere else.  Thus, $e_{\mu}\hat{\Isalg}_\phi =\varprojlim \Isalg_\phi/  \Isalg_\phi \mathfrak{m}_{\mu}^N$.  Standard calculations show:
\begin{equation}\label{eq:summand-hom}
  \Hom_{\hat{\Isalg}_\phi}(e_{\mu}\hat{\Isalg}_\phi, e_{\mu'}\hat{\Isalg}_\phi)\cong \Gamma(\tilde{\fM}, e_{\mu'}\hat{\Isalg}_\phi e_{\mu}).  
\end{equation}
\notation{$e_{\mu}$}{The idempotent in $\hat{S}_1$ that acts by 1 in the formal neighborhood of $\mu$ and vanishes everywhere else, thought of as a section of $\hat{\Isalg}_\phi$.}

Of course, the reader should recognize this analysis as almost precisely the analysis of the functors of taking weight spaces discussed in Section \ref{sec:reps} and in particular that of the category $\widehat{\mathscr{A}}$ defined in that section.  We wish to consider the subcategory $\widehat{\mathscr{A}}_{\mathbb{F}_p}$ of objects of the form $(\zero, \mu)$ with $\mu \in \ft_{1,\Fp}$; for simplicity, we will just denote this object by $\mu$.  In the notation introduced in that section, this subcategory would be $\widehat{\mathscr{A}}_0$, but we think that this is too likely to generate confusion with our convention of using this to denote objects with $h=0$.
\begin{lemma}\label{lem:A-H}
There is a fully faithful functor from $\widehat{\mathscr{A}}_{\mathbb{F}_p}$ to the category of right $  {\hat{\Isalg}_\phi}$ modules sending  $\mu\mapsto e_{\mu}\hat{\Isalg}_\phi $.  
\end{lemma}
\begin{proof}
  Note that the isomorphism $\EuScript{A}_1\cong \Gamma(\tilde{\fM},{\Isalg}_\phi)$ induces a map
  \[ \EuScript{A}_1/(\mathfrak{m}_{\mu}^N
 \EuScript{A}_1+\EuScript{A}_1\mathfrak{m}_{\mu'}^N) \to \Gamma(\tilde{\fM}, {\Isalg}_\phi/(\mathfrak{m}_{\mu}^N
 {\Isalg}_\phi+{\Isalg}_\phi\mathfrak{m}_{\mu'}^N)  )\]
It's not clear if this map is an isomorphism since sections are not right exact as a functor, but the theorem on formal functions \cite[\href{https://stacks.math.columbia.edu/tag/02OC}{Theorem 02OC}]{stacks-project} shows that after completion, we obtain an isomorphism
  \[ \varprojlim
 \EuScript{A}_1/(\mathfrak{m}_{\mu}^N
 \EuScript{A}_1+\EuScript{A}_1\mathfrak{m}_{\mu'}^N)\to \Gamma(\tilde{\fM}, e_{\mu'}\hat{\Isalg}_\phi e_{\mu})\]
  By \eqref{eq:summand-hom}, this shows that we have the desired fully-faithful functor.
\end{proof}

In particular, this means that in the case of $\mu=\second$, this weight space has an additional action of the nilHecke algebra of $W$, so $e_{\second}$ is the sum of $\#W$ isomorphic idempotents which are primitive in this subalgebra. We let $e_{0,\second}$ be such an idempotent; since we assume  $p$ does not divide the order of $\# W$, we can assume that this is the symmetrizing idempotent for the $W$-action on the weight space.   
\notation{$e_{0,\second}$}{A primitive ideal in the nilHecke algebra, considered as an element of $\Hom_{\widehat{\mathscr{A}}_{\mathbb{F}_p}}(\second,\second)$.}

\begin{lemma}\label{lem:0-split}
  For each $\mu$, the algebra $e_\mu {\hat{\Isalg}_\phi} e_\mu$ is Azumaya of degree $\# \weylW$ over $\fM$ and split by the natural action on the vector bundle $ {\mathcal{\hat Q}_\mu}:=e_\mu\hat{\Isalg}_\phi e_{0,\second}$.
\end{lemma}
\notation{${\mathcal{\hat Q}_\mu}$}{The vector bundle $e_\mu\hat{\Isalg}_\phi e_{0,\second}$ on $\widehat{\tM}$.}
Note that \cite[Prop. 1.24]{BKpos} implies that these algebras must be split, but it is more satisfying to have a concrete splitting bundle. 
\begin{proof}
 Note first that for any idempotent $e$ in an Azumaya algebra $A$, the centralizer $eAe$ is again Azumaya.  Thus, these algebras must all be Azumaya.

 If $\tilde{\fM}$ is smooth, then $\mathcal{\hat Q}_\mu$ is a vector bundle since it is a summand of an Azumaya algebra.  By Lemma \ref{lem:Q-rank}, it is thus of rank $\#W$.
  
  Since these algebras are Azumaya, this shows that their degree is no more than $\# W$, and if this bound is achieved, then they split. Since the idempotents $p^{\operatorname{rank}(G)}$ of the form $e_\mu$ sum to the identity, and the total degree is $\# W\cdot p^{\operatorname{rank}(G)}$, this is only possible if the degree of each algebra is $\#W$. This shows the desired splitting.
\end{proof}

\begin{corollary}\label{cor:Q-splitting}
  The vector bundle $\hat{\mathcal{Q}}\cong \bigoplus { \hat{\mathcal{Q}}_\mu}$ is a splitting bundle for the Azumaya algebra $ {\hat{\Isalg}_\phi}$.

There is a fully faithful functor from  $\widehat{\mathscr{A}}_{\mathbb{F}_p}$ to the category of $\Coh^{\ell \!f}(\hat{\fM})$ of locally free coherent sheaves on $\hat{\fM}$ sending  $\mu\mapsto \hat{\mathcal{Q}}_\mu$.  
\end{corollary}

Note that the bundle $e_{\operatorname{sph}}\hat{\mathcal{Q}}$ consequently is a splitting bundle for $\hat{\psalg} _\phi$; this summand can also be realized 
as the invariants of a $W$-action on $\hat{\mathcal{Q}}$.   If $W$ acts freely on the orbit of $\mu$, then $\hat{\mathcal{Q}}_\mu$ is a summand of this bundle, but otherwise, we only obtain the invariants of the stabilizer of $\mu$ in $W$ acting on this bundle.  However, since $e_{\operatorname{sph}}$ induces a Morita equivalence, these bundles satisfy $\hat{\mathcal{Q}}\cong (e_{\operatorname{sph}}\hat{\mathcal{Q}})^{\oplus \# W}$.

\begin{remark}\label{rem:tilting-from-twisting}
Note that this bundle only depends on $\phi$'s reduction mod $p$;  of course, tensoring each summand of $\hat{\mathcal{Q}}$ with an arbitrary line bundle gives a splitting bundle for the same Azumaya algebra.  We will ultimately prove that $\hat{\mathcal{Q}}$ is tilting, and this property is only preserved by tensor product of the whole bundle with a line bundle.  We can give natural realizations these tilting bundles by replacing $\hat{\Isalg}_\phi$ with the localization of the twisting bimodule ${}_{\phi+\nu}\mathscr{T}_{\phi}$, and consider the images of $e_\mu$ acting on the left on the completion of this coherent sheaf.  This still carries an action of $\hat{\Isalg}_\phi$ on the right, and we'll show below that this is again a splitting bundle for this Azumaya algebra.  
\end{remark}

\subsection{Lifting to characteristic 0}

Recall from Theorem \ref{thm:pStein-equiv} that we have an equivalence $\widehat{\mathscr{A}}_{\Fp}\cong \sfAhat(\Fp)$. Given $\mu\in \ft_{1,\Fp}$, let $\tilde{\mu}\in \ft_{1,\Z}$ be a lift. Combining this with Corollary \ref{cor:Q-splitting}, we that that:
\begin{lemma}\label{lem:Gamma-iso}
There is a fully-faithful functor $\mathsf{Q}\colon \sfAhat\to  \Coh^{\ell \!f}(\hat{\fM})$ sending $ \tilde{\mu}\mapsto  {\mathcal{\hat Q}_\mu}$.  \end{lemma}
Note that since $\zero$ is isomorphic to the direct sum of $\#  \weylW$ copies of the object $\second$ in $\mathsf{B}$, we thus have that this functor sends $\second=\second_{1/p}\mapsto \mathcal{O}_{\hat{\fM}}=e_{0,\second} {\hat{\Isalg}_\phi} e_{0,\second}$. This means that:
\begin{lemma}\label{lem:Frob-or-B}
  The functor $\mathsf{Q}$ when combined with quantum Frobenius $\sigma$ or the functor $\gamma\colon \sfAhat\to \widehat{\mathscr{A}}_{\Fp}$ induce two different isomorphisms \[\End_{\widehat{\mathscr{B}}}((\second, \second),(\second, \second))\cong \EuScript{A}^{\operatorname{sph}}_0.\]

The resulting module structures on $\Hom_{\widehat{\mathscr{B}}}( (\second,\second),(\acham,\mu))$ are isomorphic.  
\end{lemma}
\begin{proof}
  Using the action of $\widehat{W}$, we can assume that $\mu=\second$.
  The module $\Hom_{\widehat{\mathscr{B}}}( (\second,\second),(\acham,\second))$ is spanned as a module over the dots by a basis consisting of the elements $y_w\tilde{r}_{\pi}$for $w\in \widehat{W}$ such that $w\cdot \second=\second$ and  a minimal length path $\second$ to $w\cdot \acham$.  The same is true of $\Hom_{\widehat{\mathsf{B}}}(\second,\acham_{1/p})$.

  We define an isomorphism \[\ell\colon \Hom_{\widehat{\mathsf{B}}}(\second,\acham_{1/p})\to \Hom_{\widehat{\mathscr{B}}}( (\second,\second),(\acham,\second)) \]
by the formulas
\[ \ell(\la)=\la^p-\la \qquad \ell(w)=w_p\qquad\ell(u_{\al})=\frac{u_{\al^{(p)}}}{(\al^{(p)})^{p-1}-1}\]
\[\ell(r(\eta,\eta')) =r(\eta_{p},\eta'_{p}).\]
This defines an isomorphism since the polynomials $\hat \Phi_0(\acham,\acham',\second)$ and $\al^{p-1}-1$ are invertible.  It's important to note that this does not define an equivalence of categories, but only of $\EuScript{A}^{\operatorname{sph}}_0$-modules.
\end{proof}

We wish to extend this result to the coherent sheaves $ {\mathcal{\hat Q}_\mu}$.  In order to do this, it is useful to consider the completed category $ {\widehat{\mathscr{B}}^{\To}}$ attached to the gauge group $\To$.  We have a functor from this category to $\Coh^K(\hat{\fM}_{\To})$, the category of $\groupK$-equivariant coherent sheaves on the corresponding completion of the Coulomb branch $\fM _{\To}$.  This functor is given by considering $\Hom_{\widehat{\mathscr{B}}^{\To}}( (\second,\second),(\acham,\mu))$ as a module over $\EuScript{A}^{\To}_0=\End_{\widehat{\mathscr{B}}^{\To}}((\second, \second),(\second, \second))$, where the isomorphism is via the quantum Frobenius.

This inherits a $K$-action from the category $ {\widehat{\mathscr{B}}^{\To}}$ itself.  If we change $\acham\mapsto \acham+p\nu$ for $\nu \in \ft_{\To,\Z}$, this has the effect of twisting the equivariant structure by the corresponding character of $K$ induced by exponentiating $\gamma$. In particular, as an equivariant sheaf, this depends only on the image of $\nu$ in $\ft_{F,\Z}$, so if $\gamma\in \ft_{\Z}$, the resulting sheaf is $K$-equivariantly isomorphic.

By definition, the module $\mathcal{\hat Q}_\mu$ is the reduction of the coherent sheaf \[\mathcal{\hat R}_\mu=\Hom_{\widehat{\mathscr{B}}^{\To}}( (\second,\second),(\zero,\mu)),\] thought of as a $\EuScript{A}^{\To;\operatorname{sph}}_0$-module via the quantum Frobenius $\sigma$.  

Of course, we can apply the functor of Proposition \ref{prop:B-equiv} with the gauge group $\To$;  this gives us an identification of $\mathcal{\hat R}_\mu$ with $\mathsf{\hat R}_\mu=\Hom_{\widehat{\mathsf{B}}^{\To}}( \second,-\tilde{\mu}_{1/p}+\zero)$.  This is a module over $\EuScript{A}^{\To;\operatorname{sph}}_0\cong \Hom_{\widehat{\mathsf{B}}^{\To}}( \second,\second)$, and the two possible module structures are isomorphic by Lemma \ref{lem:Gamma-iso}.

Note that using this presentation has enormous advantages---we can consider the induced module $\mathsf{ R}_\mu=\Hom_{{\mathsf{B}}^{\To}}( \second,\mu_{1/p})$ in the uncompleted category ${\mathsf{B}}^{\To}$; localizing, this gives a $K\times \mathbb{G}_m$-equivariant module on $\MQ$. Furthermore, whereas all of the geometry discussed earlier in this category required us to consider $\fM$ over a base field of  characteristic $p$, the category $ {\mathsf{B}}^{\To}(\K)$ is well-defined over $\Z$ and thus over any commutative base ring $\K$ .  
\begin{definition}\label{def:Q-def}
Let $ {\mathcal{Q}_\mu^{\K}}$ be the $\mathbb{G}_m$-equivariant coherent sheaf on $\tilde{\fM}$ given by Hamiltonian reduction of $ \mathsf{ R}_\mu(\K)=\Hom_{{\mathscr{B}}^{\To}(\K)}( \second,-\tilde{\mu}_{1/p}+\zero)$.
\end{definition}
\notation{$\mathcal{Q}_\mu^{\K}$}{The $\mathbb{G}_m$-equivariant coherent sheaf on $\tilde{\fM}$ given by Hamiltonian reduction of $ \mathsf{ R}_\mu(\K)=\Hom_{{\mathscr{B}}^{\To}(\K)}( \second,-\tilde{\mu}_{1/p}+\zero)$ (Definition \ref{def:Q-def}). This is a lift to arbitrary base ring of $\widehat{\mathcal{Q}}_\mu$.}

\begin{remark}\label{rem:tilting-from-twisting2}
  We can construct the more general tilting bundles from Remark \ref{rem:tilting-from-twisting} by choosing a lift of the cocharacter $\nu\colon \mathbb{G}_m\to \Aut_{G}(V)$, and consider the Hamiltonian reduction of $\Hom_{{\mathscr{B}}^{\To}(\K)}( \second,-\tilde{\mu}_{1/p}+\zero+\nu)$.  This is simply $ {\mathcal{Q}_\mu^{\K}}\otimes \mathcal{L}_{\nu/p}$ where $\mathcal{L}_{\nu/p}$ is the line bundle that arises from $\Hom_{{\mathscr{B}}^{\To}(\K)}( \second,\nu)$.  If $\nu$ is $p$-divisible, then this is the associated bundle for $\nu/p$, identified with a character of $K$.  If  not, then describing this line bundle is more complicated (though we should emphasize that it is an honest line bundle, not a fractional power).  
\end{remark}

\subsection{Derived localization}

For now, let us specialize back to the case where $\K=\mathbb{F}_p$.  By  Grauert and Riemenschneider for Frobenius split varieties (\cite{mehtaGrauertRiemenschneiderVanishing1992}) and the splitting of Proposition \ref{prop:nonabelian-splitting}, we have that:
\begin{corollary}\label{cor:cohomology-vanishing}
For any prime $p$, we have the higher cohomology vanishing $H^i(\tM;\cO)=0$ for all $i>0$.
\end{corollary}
As discussed in \cite{KalDEQ}, this means that the
derived functor of localization $\LLoc$ is right inverse to the
functor  $\Rsecs$ of derived sections for modules over $ {\psalg_\phi}$.  
Recall that we have chosen $\chi$ such that  $\tilde{\fM}$ is smooth. We can conclude from \cite[Thm. 4.2]{KalDEQ} that:

\begin{lemma}\label{lem:upper-bound}
  There is an integer $N$, such that for any $p$, and any line parallel to $\chi$ in $\ft_{1,\Fp}$, there are at most $N$ values of $\phi$ for which $\LLoc$ and $\Rsecs$ are {\em not} inverse equivalences.  
\end{lemma}
\begin{remark}
It seems likely that this result also holds when $\tilde{\fM}$ is not smooth, at least for the quantizations we have constructed, but let us leave this point unresolved for the time being.
\end{remark}

\begin{lemma}\label{lem:tiling-localization}
The vector bundle $\mathcal{Q}^{\Fp}=\bigoplus_{\mu} {\mathcal{Q}_\mu^{\Fp}}$ is a tilting generator for $\Coh(\fM)$ if and only if derived localization holds for $ {\hat{\psalg}_\phi}$.  
\end{lemma}
\begin{proof}
First note that  by semi-continuity, it's enough to show this for $\mathcal{\hat Q}^{\Fp}$ on $\hat{\fM}$.  We know that on $\hat{\fM}$, we have an isomorphism $\hat{\psalg}\cong \sHom_{\cO_{\hat{\fM}}}(\mathcal{\hat Q}^{\Fp},\mathcal{\hat Q}^{\Fp})$.  Since the higher cohomology of $\hat{\psalg}$ vanishes, this shows that $\mathcal{\hat Q}^{\Fp}$ is a tilting bundle.  

The $\hat{\psalg}$-modules are precisely the sheaves of the form $\sHom_{\cO_{\hat{\fM}}}(\mathcal{\hat Q}^{\Fp},\mathcal{F})$ for a coherent sheaf $\mathcal{F}$. Since $\mathcal{\hat Q}^{\Fp}$ is a vector bundle, we have that \[H^i(\fM;\sHom_{\cO_{\hat{\fM}}}(\mathcal{\hat Q}^{\Fp},\mathcal{F}))\cong \Ext^i_{\cO_{\hat{\fM}}}(\mathcal{\hat Q}^{\Fp},\mathcal{F}).\] Thus, $\mathcal{Q}^{\Fp}$ is a generator if and only if no module over $\hat{\psalg}$ has all cohomology groups trivial.   
\end{proof}

\begin{corollary} \label{cor:char-p-equiv}
If derived localization holds at $\phi$, then the fully faithful functor  $\mathsf{Q}\colon \sfA(\Fp)\to \Coh(\tilde{\fM})$  induces an equivalence of derived categories $D^b(\mathsf{A}(\Fp)\mmod)\cong D^b(\Coh(\tilde{\fM}))$.
\end{corollary}
\begin{proof}
  If derived localization holds at $\phi$, then the induced derived functor is essentially surjective, since $\mathcal{Q}$ is a generator of the derived category.  Thus, this derived functor is an equivalence. 
\end{proof}

Let $\Lambda,\barLambda$ be as defined in Definition \ref{def:Lambda}.  As noted before, the set $\bar{\Lambda}$ is finite.
\begin{definition}
 We call a choice of $\psi=\phi_{1/p}$ {\bf generic} if the number of elements of $ \barLambda$ is maximal amongst all choices of $\psi\in \ft_{1,F,\R}$.  
\end{definition}
Note that for a given $p$, there may be no generic choices of $\psi$ in $\ft_{1,F,\frac{1}{p}\Z}$, but since real numbers can be arbitrarily well approximated by fractions with prime denominators, there are generic $\psi$ with $\phi\in \ft_{1,F,\Z}$ for all sufficiently large $p$.   In fact, we can divide $\ft_{1,F,\R}$ up into regions $R_{\bar{\Lambda}'}$ according to  what the set $\bar \Lambda'$ attached to $\psi$ is.  Having a maximal number of such non-empty chambers is a open dense property (it is the complement of the integral translates of finitely many hyperplanes).  Simple geometry shows that:
\begin{lemma}\label{lem:segment}
For a fixed $\barLambda$ with $R_{\bar{\Lambda}}$ open and non-empty and a fixed integer $N$, there is a constant $M$ such that if $p>M$ then there is a choice $\phi\in \ft_{1,\Z}$ such that $\phi,\phi+\chi,\phi+2\chi,\dots, \phi+N\chi$ are generic and \[R_{\bar{\Lambda}}\supset \left\{(\phi+k\chi)_{1/p}\,\big| \,k\in \R, 0\leq k\leq N\right\}.\]
\end{lemma}
Recall that as we mentioned earlier that there is a constant $N$ such that for a fixed $\phi$, localization can only fail at $N$ values of the form $\phi+k\chi$ for $k\in \Z/p\Z$.  Fix $\bar{\Lambda}$ with $R_{\bar{\Lambda}}$ open and non-empty and let $M$ be the associated constant in Lemma \ref{lem:segment}.
\begin{theorem}\label{thm:asymptotic-derived}
  If $\phi$ is a generic parameter with $\barLambda$ as fixed above, and $p>M$, then derived localization holds for $\phi$, and so the associated $ {\mathcal{Q}^{\Fp}}$ is a tilting generator.  
\end{theorem}
\begin{proof}
  First note that it is enough to replace $\phi$ by any other generic parameter with the same set $\bar{\Lambda}$. In this case,  tensor product with the bimodule ${}_{\phi}T_{\phi'}$  sends any object in $\AC_{\Ba}$ in the preimage of $\phi$ to one in $\AC_{\Ba}$ in the preimage of $\phi'$ (see \eqref{eq:aff-cham} for the definition of $\AC_{\Ba}$).  Thus the categories $\sfA(\Fp)$ are naturally equivalent via tensor product with bimodule ${}_{\phi}T_{\phi'}$ connecting them. 
  
  Thus, we can assume that $\phi$ is as in Lemma \ref{lem:segment}.  If derived localization fails at $\phi$, then it also fails at $\phi+\chi,\phi+2\chi,\dots, \phi+N\chi$.  This is impossible by our upper bound on the number of points where it fails from Lemma \ref{lem:upper-bound}.
\end{proof}

This is certainly too crude to give a sharp characterization of when derived localization holds.  We expect that we will instead find that:
\begin{conjecture}
If $\phi$ is a generic parameter, then derived localization holds for $\phi$.  Equivalently, if $\phi$ and $\phi'$ are generic, then derived tensor product with ${}_{\phi}T_{\phi'}$ is an equivalence between $D^b(\EuScript{A}_\phi\mmod)$ and $D^b(\EuScript{A}_{\phi'}\mmod)$.
\end{conjecture}

These results have consequences for the case where $\K$ is an arbitrary commutative ring.
Note that by construction $\mathsf{A}(\K)$, and thus $\mathcal{Q}^{\K}$, depends on a choice of
$\phi$ and ultimately a prime $p$, but for fixed $\K$, this dependence is very
weak.
\begin{lemma}\label{lem:doesnt-depend2}
  The vector bundle
   $ {\mathcal{Q}_\mu^{\K}}$ only
  depends on which element of $\barLambda$ corresponds to the
  chamber $\rACp_{\Ba}$ containing $\mu$. Consequently, the vector bundles that appear this way for a fixed $\flav$ only depends on the set $\barLambda$.
\end{lemma}
\begin{proof} If
  $\mu_1$ and $\mu_2$ both lie in $\rACp_{\Ba}$ then we obtain an
  isomorphism $\mathcal{Q}_{\mu_1}^{\K}\cong \mathcal{Q}_{\mu_2}^{\K}$.
\end{proof}
As we change $\flav$ and $p$ while keeping $\bar \Lambda$ fixed, the
number of integral points in each chamber $\rACp_{\Ba}$  will increase and decrease, so the vector
bundle $\mathcal{Q}^{\K}$ will change, but only by changing the
number of times different summands appear; that is, the vector bundles $\mathcal{Q}^{\K}$ for different $\phi$ are {\bf equiconstituted}.  Which summands appear at
least once will only change when we change $\bar \Lambda$.

     \notation{${\Lambda^{\R}}$}
    {The set of vectors such that $\rACp_{\Ba}$
 is non-empty.}
 \newcommand{\LambdaR}{\Lambda^{\R}}
We obtain the cleanest statement if we pass to $\Q$, which as we mentioned before is essentially the case of $p$ is infinitely large.  In this case, it is convenient to fix a parameter $\psi\in \ft_{1,F,\R}$, defining a real flavor, and consider the set $\LambdaR$ of vectors with $\rACp_{\Ba}$ non-empty and $\bar{\Lambda}^{\R}$ its quotient by $\What$; as before, we call $\psi$ generic if the set $\bar{\Lambda}^{\R}$ has maximal size.  We let $ {\cQ^{\Q}_\phi}$ be the sum of the vector bundles under $ {\cQ^\Q_\mu}$ for representatives $\mu$ of each different chamber in $\bar{\Lambda}^{\R}$.  This is analogous to the construction of the category $ {\mathsf{B}^{\bar{\Lambda}^{\R}}(\Q)}$ discussed in Definition \ref{def:BLam}.   
\notation{$\cQ^{\Q}_\phi$}{The sum of the vector bundles under $ {\cQ^\Q_\mu}$ for representatives $\mu$ of each different chamber in $\bar{\Lambda}^{\R}$. For generic $\phi$, this is a tilting generator. }
\begin{theorem}\label{th:Q-equiv}
  If $\psi$ is a generic parameter, the vector bundle $ {\mathcal{Q}^{\mathbb{Q}}_\phi}$ on $ {\tilde{\fM}_{\mathbb{Q}}}$ is 
  a tilting generator with a summand isomorphic to the structure sheaf $\mathcal{O}_{\tilde{\fM}}$ and 
  induces an equivalence $D^b( {\mathsf{B}^{\bar{\Lambda}^{\R}}(\Q)})\cong D^b(\Coh(\tilde{\fM}_{\mathbb{Q}})).$
\end{theorem}
\begin{proof}
  By construction, $\mathsf{R}_\tau(\K)\cong \mathcal{O}_{\tilde{\fM}}^{\oplus \#W}$ since $\zero$ is isomorphic to $\#W$ copies of $\tau$, so this shows the structure sheaf is a summand. 
   Being a vector bundle and a tilting generator after base change to a point of $\Spec\,
   \Z$ is an open property, so if the set of primes where this holds
   is non-empty, it must be so over $\Q$ as well.  Thus, we need only
   show that $\mathcal{Q}^{\Fp}$ is a tilting generator for some
   prime $p$.  By Lemma \ref{lem:doesnt-depend2}, this fact only
   depends on the corresponding $\bar \Lambda$.  By Theorem
   \ref{thm:asymptotic-derived}, for $p\gg 0$, there is a $\phi$ which
   gives $\bar \Lambda$ as the set of chambers with integral points
   such that derived localization holds at $\phi$.  Thus, by Lemma
   \ref{lem:tiling-localization}, the associated sheaf
   $\mathcal{Q}^{\mathbb{F}_p}_\phi$ is a tilting generator, which
   establishes the result.
 \end{proof}
 Of course, the tilting bundles constructed from $\mathcal{Q}^{\mathbb{Q}}_\phi$ by tensoring with line bundles have the corresponding line bundle as a summand.  
 
\subsection{Noncommutative crepant resolutions}
\label{sec:noncomm-crep}

Recall the notion of a {\bf noncommutative crepant resolution} of the affine variety $\Coulomb$, originally defined in \cite{vandenberghNoncommutativeCrepant2004}: This is an algebra $A=\End(M),$ for some reflexive coherent sheaf $M$ on $\fM$, such that $A$ is a Cohen-Macaulay as a coherent sheaf and the global dimension of $A$ is equal to $\dim \fM.$  A {\bf D-equivalence} between a commutative resolution $\tM$ and a noncommutative resolution $A$ is an equivalence of dg-categories $D^b(\Coh(\tilde{\fM}))\cong D^b(A\mmod)$.  

The following is a corollary of \cite[Lem. 3.2.9 \& Prop. 3.2.10]{vandenberghThreedimensionalFlops2004}:
\begin{lemma}\label{lem:tilt-nccr}
Suppose $\mathcal{T}$ is a tilting generator on a resolution $\tM$ such that the structure sheaf $\mathcal{O}_{\tfM}$ is a summand of $\mathcal{T}$, and let $M=\Gamma(\tfM;\mathcal{T}).$ Then $A=\End_{\Coh(\tfM)}(\mathcal{T})\cong \End_R(M)$ is a noncommutative crepant resolution of singularities, canonically D-equivalent to $Y$.  
\end{lemma}
We have already defined the ring $A=A(\Q)$ that appears in our case in \eqref{eq:A-def}.
 By the equivalence of Theorem \ref{th:Q-equiv}, the definition of $A$ implies we have an isomorphism \[A=\End_{\Coh(\tfM)}( \mathcal{Q}^{\mathbb{Q}}_\phi)\] and an idempotent $e_0$ projecting to the structure sheaf as a summand.     Then, applying Lemma \ref{lem:tilt-nccr}, we can see that:
\begin{corollary}\label{cor:A-nccr}
  The ring $A$ is a noncommutative crepant resolution of the Coulomb branch $\fM$, induced by the pushforward $M$ of $\mathcal{Q}^{\mathbb{Q}}_\phi$ to $\fM$.  
\end{corollary}
As mentioned earlier, we can give very explicit computations of the algebras in question when $\fM$ is a quiver gauge theory, which we will discuss in much greater detail in \cite{WebcohII}.  This is also true in the hypertoric case, as discussed in \cite[Prop. 3.35]{mcbreenHomologicalMirror2024} and \cite[\S 4.1]{gammageHomologicalMirror2023}.

\subsection{Presentations}
\label{sec:presentations}

For the sanity of the reader, let us try to give a more explicit description of the resulting algebra $A$ which gives our noncommutative resolution of singularities.  For our gauge group $G$, consider the fundamental alcove $\nabla$ in the Cartan of $\fg$ modulo the action of the group $\widehat{W}_0$ of length 0 elements in the extended affine Weyl group (which are, by definition, the elements sending the fundamental alcove to itself). That is, $\nabla$ is the subset of the positive Weyl chamber in $\ft$ that is not separated from the origin by the zeros of any affine root.   For fixed flavor $\flav$, every point in this space gives an object in the extended category $\sfB$, but there are only finitely many isomorphism types, given by the set $\barLambda$.  

First of all, we divide this fundamental alcove by considering the hyperplanes defined by $\varphi_i(\acham)\equiv -\phi_i/p\pmod \Z$ for $\phi_i$ the weight of the flavor $\phi$ on the weight space $V_i$; these are the unrolled matter hyperplanes.  Unrolled root hyperplanes only appear on the boundaries of the alcove and only ones corresponding to simple roots of the affinization of $G$ are relevant.  Also, note that the objects corresponding to the walls of the fundamental alcove are summands of the nearby generic objects, so up to isomorphism or inclusion of summands, we can take the algebra $A$ to be the endomorphisms of a sum of representatives of the chambers cut out by the unrolled hyperplanes.  By \cite[Cor. \ref{SD-lem:Coulomb-basis}]{websterKoszulDuality2019}, we have a basis of these endomorphisms which we can visualize as straight (or small perturbation of straight) paths in $\ft$, folded using reflections to fit in the fundamental alcove.  Of course, having chosen representatives of each chamber, we can factor this path to pass through the representative of each chamber it passes through, and thus factor it into shorter segments that either:
\begin{enumerate}
	\item join chambers which are adjacent across an unrolled matter hyperplane
	\item ``bounce'' off a root hyperplane within a chamber bounding it.  
\end{enumerate}  
We can thus, we have that
\begin{proposition}\label{prop:presentation}
The algebra $A$ is a quotient of the path algebra of the quiver where:
\begin{enumerate}
	\item  nodes are given by $\barLambda$, the chambers in this arrangement,
	\item  we add as endomorphisms to each node the semi-direct product of $S_h$ with the stabilizer of the corresponding chamber in $\widehat{W}_0$
	\item we add an opposing pair of edges for every pair of chambers adjacent across a matter hyperplane 
	\item we add a self-loop for each adjacency of a chamber to a root hyperplane.
\end{enumerate}
\end{proposition}
The relations that we need arise from (\ref{eq:dot-commute}--\ref{eq:triple}).   One simply takes the pictures \cite[(\ref{SD-eq:graphical1}--\ref{SD-eq:graphical3})]{websterKoszulDuality2019}, replaces chambers with nodes in the quiver and hyperplanes with arrows, and interprets in the path as one in the quiver. These are a bit tedious to write out in full generality, so we leave this as an exercise to the reader.   
\begin{example}
	One valuable example to consider is when $\C^*$ acts on $\C^n$ with weight $1$.  In this case, the fundamental alcove is all of $\ft_\R$ and the extended affine Weyl group the coweight lattice, so the quotient is the maximal compact of the torus $T\subset G$.  The flavor $\phi$ has $n$ components $(\phi_1,\dots, \phi_n)$, and the unrolled matter hyperplane arrangement is given by removing the points $x=-\phi_i/p$ from the circle.  Thus, we have $n$ chambers arranged in a circle.  For simplicity, we draw each pair of arrows from a matter hyperplane as a double-headed arrow, so the structure we see is:
	
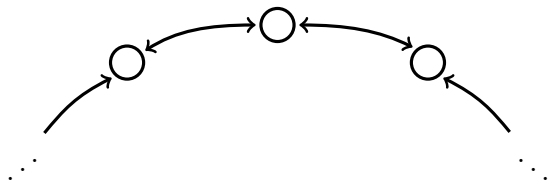
\[\tikz[very thick,xscale=2]{\node[draw,circle ,outer sep=2pt] (A) at (0,-.5) {$\,$}; \node [draw,circle ,outer sep=2pt] (B) at (1,0) {$\,$}; \node [draw,circle,outer sep=2pt] (C) at (2, -.5) {$\,$}; \node[outer sep=2pt,inner sep=0pt] (D) at (-.7,-1.8) {$\iddots$}; \node[outer sep=2pt,inner sep=0pt] (E) at (2.7, -1.8) {$\ddots $};  \draw[<->] (A) to[out=45,in=180](B); \draw[<->] (B) to[out=0,in=135](C);  \draw [->]  (D) to[in=-135,out=60] (A); \draw [->](E) to[out=120,in=-45] (C);}\] 
	In fact, $A$ is the preprojective algebra of this quiver, which is well-known to give the desired noncommutative resolution. 
\end{example}
\begin{example}
In our usual running example, with $G=GL(2)$, the fundamental alcove is the region $\{(x,y)\in \mathbb{R}^2 \mid 0\leq x-y \leq 1$, and the length 0 elements of the affine Weyl group act by the integer powers of the glide reflection $(x,y)\mapsto (y+1, x)$.  The quotient is thus a M\"obius band, which we can identify with the configuration space of pairs of points on a circle.  

We take matter representation $V=\C^2\oplus \C^2$ and thus obtain a chamber structure in Figure \ref{fig:pthroot}.  That is, we have a geometry like
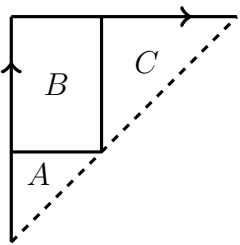
\[\tikz[very thick,scale=3]{\draw[dir] (0,0) -- (0,1); \draw[dir] (0,1) -- (1,1);  \draw (0,.4) -- (.4,.4); \draw (.4,.4) -- (.4,1); \draw[dashed] (0,0) -- (1,1);\node at (.12,.3){$A$}; \node at (.2,.7){$B$}; \node at (.6,.8){$C$};}\]	where the solid lines are matter hyperplanes, dashed lines are root hyperplanes, and the lines with arrows indicate gluing to obtain a M\"obius strip with dashed boundary.  Thus, we have between $A$ and $B$ two adjacencies and thus two {\it pairs} of arrows, and similarly with $B$ and $C$, with $A$ and $C$ both having self-loops corresponding to the adjacent root hyperplane. 
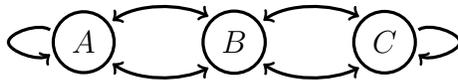
\[\tikz[very thick,xscale=2]{\node[draw,circle ,outer sep=2pt] (A) at (0,0) {$A$}; \node [draw,circle ,outer sep=2pt] (B) at (1,0) {$B$}; \node [draw,circle,outer sep=2pt] (C) at (2,0) {$C$}; \draw[->] (A.160) to[out=125,in=90] (-.5,0) to[out=-90,in=-125] (A.-160); \draw[->] (C.20) to[out=55,in=90] (2.5,0) to[out=-90,in=-55] (C.-20); \draw[<->] (A) to[out=45,in=135](B); \draw[<->] (A) to[out=-45,in=-135](B); \draw[<->] (B) to[out=45,in=135](C); \draw[<->] (B) to[out=-45,in=-135](C);}\] 
\end{example}

\section{Schobers and wall-crossing}
\label{sec:schob-wall-cross}

Our final section will concern the theory of {\bf twisting functors} (also called {\bf wall-crossing functors}), and in particular, their connection to the theory of Schobers.  These functors are discussed for general symplectic singularities in \cite[\S 2.5.1]{losevModularCategories2021}.  Schobers constructed from categories of coherent sheaves and variation of GIT have already appeared in work of Donovan \cite{donovanPerverseSchobers2019} and Halpern-Leistner and Shipman \cite{HL-S}.  These works have mostly focused on a single wall-crossing, rather than a more complicated hyperplane arrangement, but the simplicity of Coulomb branches compared to other symplectic singularities gives us a tighter control over the structures appearing.  

We will first give some preliminary results on Morita contexts.  These are, of course, standard objects of study in noncommutative geometry and algebra, but their connections to spherical functors and thus to Schobers seem to have mostly escaped notice.  Then, we turn to the construction of a Schober and thus a $\pi_1$-action from the algebraic and geometric objects considered earlier in the paper.  We'll note here that essentially identical arguments will construct Schobers in many similar contexts where actions of fundamental groups have been constructed, in particular for the twisting and shuffling functors in characteristic 0 considered in \cite{BLPWquant,BLPWgco}.

We'll also note that it seems quite likely that this argument proceeds essentially identically for all symplectic resolutions of singularities.  However, both for reasons of notational convenience, and avoiding certain technical difficulties (in particular, proving the analogue of Lemma \ref{lem:just-hyperplanes}), we will restrict ourselves to the case of Coulomb branches. 

 \subsection{Morita contexts and spherical functors}
\label{sec:morita-cont-spher}

Recall that a {\bf Morita} context (called ``pre-equivalence data'' in \cite{BassK}) is a category with 2 objects $\{+,-\}$.  The endomorphism algebras of the two objects give two rings $R_+$ and $R_-$, and the Hom spaces give  $R_\pm$-$R_\mp$ bimodules ${}_{\pm}R_\mp$.  Let $I_\pm={}_\pm R_\mp\cdot {}_\mp R_\pm$ be the two-sided ideal of morphisms factoring through $\mp$, and $Q_\pm=R_{\pm}/I_{\pm}$.  For simplicity, we assume that $R_+$ and $R_-$ have finite global dimension. Modules over this category are the equivalent to modules over the ``matrix'' ring \[R=
\begin{bmatrix}
  R_+ & {}_+R_-\\
   {}_-R_+ & R_-
 \end{bmatrix}\]
Let $e_+,e_-$ be the identities on the 2-objects. For any context, we have quotient functors $q_{\pm}\colon R\mmod \to R_{\pm}\mmod$ with $q_{\pm}(M)=e_{\pm}M=e_{\pm}R\otimes_RM=\Hom_{R}(Re_{\pm},M)$.  This functor has left and right adjoints \[{}^*q_{\pm}(N)=Re_{\pm}\otimes_{R_{\pm}}N\qquad q_{\pm}^*(N)=\Hom_{R_{\pm}}(e_{\pm}R,N).\]  Of course, both of these functors are fully faithful.  The images of their derived functors thus give two copies of $\mathcal{E}_{\pm}:=D^b(R_{\pm}\mmod)$ in $\mathcal{E}_0 =D^b(R\mmod)$ which are the left and right perpendiculars of $\mathcal{F}_{\pm}$, the subcategory of the derived category of $D^b(R\mmod)$ which become acyclic after applying $e_{\pm}$.  This can be identified with modules over the dg-algebra $F_{\pm}=\Ext_R^\bullet(Q_{\mp},Q_{\mp})$.  The inclusion $\xi_{\pm}$ of this subcategory can then be identified with $Q_{\mp}\Lotimes_{F_{\pm}}-$.  Thus, left and right adjoints of this functor are given by \[{}^*\xi_{\pm}(M)=Q_{\mp}\Lotimes_{R}-\qquad \xi_{\pm}^*(M)=\RHom_{R}(Q_{\mp}, M).\]

The inclusions $j_{\pm}=q_{\pm}^*$ and $\xi_{\pm}$ thus fit in the setup of \cite[\S 3.C]{kapranovPerverseSchobers2015}.  Consider the composition $S=q_{\pm}\circ \xi_{\mp}$.  This has left and right adjoints \[L={}^*\xi_{\mp}\circ {}^*q_{\pm} =Q_{\pm}\Lotimes_{R_{\pm}}- \qquad R=\xi_{\mp} ^*\circ q_{\pm}^*= \RHom_{R_{\pm}}(Q_{\pm}, -).  \]

Consider the functors \[{}^*j_{\pm}\circ j_{\mp}=q_{\pm}\circ q_{\mp}^*={}_{\pm}R_{\mp}\Lotimes_{R_{\mp}}-\colon  \mathcal{E}_{\mp}\to \mathcal{E}_{\pm} \]
\begin{lemma}\label{lem:equiv-sphere}
  If ${}^*j_{\pm}\circ j_{\mp}$ and ${}^*j_{\mp}\circ j_{\pm}$ are equivalences of derived categories, then the data above define a spherical pair in the sense of Kapranov and Schechtman \cite[\S 3.C]{kapranovPerverseSchobers2015}, and the functor $S$ is spherical.
\end{lemma}
\begin{proof}
  In addition to our hypotheses, we need to prove that $\xi_{\mp}^*\circ \xi_{\pm}$ are equivalences of derived categories.  If ${}_{\mp}R_{\pm}\Lotimes_{R_{\pm}}-$ is an equivalence, then its inverse is its adjoint $\Hom_{R_{\mp}}({}_{\mp}R_{\pm},-)$.   Thus, $N'={}^{**}j_{\mp}( {}_{\mp}R_{\pm}\Lotimes_{R_{\pm}}N)$ is an $R$-module such that ${}^*j_{\pm}(N')\cong N$.  This shows we have a natural map $j_{\pm}(N)\to N'$, which is a quasi-isomorphism after applying $e_{\pm}$ (by the observation we just made) and a quasi-isomorphism after applying $e_{\mp}$, by the isomorphism of ${}^*j_{\mp}{}^{**}j_{\mp}$ to the identity.

  Thus, $j_{\pm}$ and ${}^{**}j_{\mp}$ have the same image.  Obviously, $\mathcal{F}_{\pm}$ is the left orthogonal to this image, and $\mathcal{F}_{\mp}$ its right orthogonal.  Thus, $\xi_{\mp}^*\circ \xi_{\pm}$ is the mutation with respect to these dual semi-orthogonal decompositions.  Note that this is a special case of \cite[Thm. 3.11]{HL-S}, with the ambient dg-category being the derived category of $R$-modules, the category $\mathcal{A}$ being the image of $j_{\pm}$ and  ${}^{**}j_{\mp}$, and $\mathcal{A}'$ the image of $j_{\mp}$ and  ${}^{**}j_{\pm}$.
\end{proof}

\subsection{Wall-crossing functors}
\label{sec:wall-cross-funct}

For different choices of flavor $\flav$, we obtain different quantizations of
the structure sheaf of $\fM$.  Quantized line bundles give canonical
equivalences of categories between the categories of modules over
these sheaves, as in \cite{BLPWquant}.  Note that the isomorphism type
of the underlying sheaf only depends on $\phi$ considered modulo $p$,
but for different elements of the same coset, there is still a
non-trivial autoequivalence, induced by tensoring with the
quantizations of $p$th power line bundles.  Similarly, for each
element of the Weyl group $W_F$, there's an isomorphism between the
section algebras of $\EuScript{A}_{\phi}$ and $ \EuScript{A}_{w\cdot
  \phi}$; together, these give us such a morphism for every $w\in
\widehat{W}_F$, affine Weyl group of $F$.  We thus can consider the twisting bimodule
$ {{}_{w\phi'}T_{\phi}}$ discussed earlier, turned into a
$\EuScript{A}_{\phi'}\operatorname{-}\EuScript{A}_{\phi}$-bimodule
using the isomorphism above to twist the left action.

\begin{definition}\label{def:twisting}
Given flavors $\flav$ and $\phi'$, and $w\in \widehat{W}_F$, we define the {\bf twisting} or {\bf wall-crossing functor} $ {\Phi_w^{\phi',\phi}}\colon D^b(\EuScript{A}_{\phi}\mmod) \to D^b(\EuScript{A}_{\phi'}\mmod)$ to be the derived tensor product with $ {{}_{w\phi'}T_{\phi}}$.
\end{definition}
\notation{${\Phi_w^{\phi',\phi}}$}{The twisting or wall-crossing functor (Definition \ref{def:twisting}).}
One can think of this functor as measuring the different sets $\Lambda,\Lambda'$ attached to the parameters $\phi',\phi$.  In particular:
\begin{lemma}\label{lem:Lam-same}
  If $\phi, \phi'$ are generic and $\Lambda=\Lambda'$, then ${}_{\phi'}T_{\phi}$ induces a Morita equivalence and $\Phi_1^{\phi',\phi}$ is an exact functor.
\end{lemma}
\begin{proof}
  Of course, we have natural maps ${}_{\phi'}T_{\phi}\otimes {}_{\phi}T_{\phi'}\to \EuScript{A}_{\phi'}$ and similarly with $\phi,\phi'$ reversed.  This gives a Morita context, as discussed above, and by \cite[II.3.4]{BassK}, we will obtain the desired Morita equivalence if we prove both of these maps are surjective.

  If this map is not surjective, then its image is a proper 2-sided ideal (sometimes called the trace of the Morita context).  Since $  \EuScript{A}_{\phi'}$ is finitely generated over its center, the quotient by this ideal has the same property, so it has at least one finite dimensional simple module $L$, which thus satisfies $ {}_{\phi}T_{\phi'}\otimes_{\EuScript{A}_{\phi'}}L=0$.  Thus, any chamber that appears in the support of $L$ must lie in $\Lambda'$ but not $\Lambda$, which is impossible since these sets coincide.  In fact, it's clear from Theorem \ref{thm:pStein-equiv} that $ {}_{\phi}T_{\phi'}\otimes_{\EuScript{A}_{\phi'}}-$ induces an equivalence on the category of finite dimensional representations.  Thus, we must have that ${}_{\phi'}T_{\phi}\otimes {}_{\phi}T_{\phi'}\to \EuScript{A}_{\phi'}$ is surjective, and similarly with $\phi, \phi'$ reversed.
\end{proof}

Recall that $\tM$ depends on a choice of $\chi\in \ft_{F,\Z}$.  This dependence is rather crude, though.  By the usual theory of variation of variation of GIT \cite{DHGIT}, the space $\ft_{F,\R}$ is cut into a finite number of convex cones, such that $\tilde{\fM}$ is smooth when $\chi$ lies in the interior of one of these cones, called ``chambers'' in \cite{DHGIT}. An element $\chi'$ will give an ample line bundle on $\tilde{\fM}$ if it is in the same chamber as $\chi$ (since their stable loci coincide), or a semi-ample bundle if it lies in the boundary of the cone (since the semi-stable locus becomes strictly larger by the Hilbert-Mumford criterion).   Since by Corollary \ref{cor:BFN-split}, the variety $\fM$ is Frobenius split, \cite[Thm. 1.4.8]{BrionKumar} shows that the corresponding line bundle induced by $\chi'$ has vanishing cohomology for all $\chi'$ in the closure of the chamber containing $\chi$.
\begin{lemma}\label{lem:localize-twist}
If $\chi'=w\cdot \phi'-\phi$ lies in the closure of the chamber containing $\chi$, then we have a natural isomorphism \[\Phi_w^{\phi',\phi}(M)\cong\Rsecs({}_{w\phi'}\mathcal{L}_{\phi}\otimes \LLoc(M))\] where the action on the RHS is twisted by the isomorphism $\EuScript{A}_{\phi'}\cong \EuScript{A}_{w\cdot \phi'}$.  
\end{lemma}
\begin{proof}
It is enough to check this on the algebra $\EuScript{A}_{\phi}$ itself. Thus, we need to show that $H^i(\fM;{}_{w\phi'}\mathcal{L}_{\phi})=0$ for $i>0$.  This is clear since this is a quantization of the line bundle induced by $\chi'$, which has trivial cohomology as discussed above.
\end{proof}

\begin{corollary}
If derived localization holds at $\phi'$ and $\phi$, then the functor $ {\Phi_w^{\phi',\phi}}$ is an equivalence of categories. 
\end{corollary}

Corresponding to a flavor $\phi$, we have a set $\LambdaR$ as defined to be the vectors in $\Z^d$ such that $\rACp_{\Ba}\neq 0$; this agrees with $\Lambda$ for $p$ sufficiently large.  The set $\Lambda^{\mathbb{R}}$ is locally constant, and only changes when $\psi=\phi_{1/p}$ lies on a hyperplane in $\ft_{1,F,\R}$ defined by a circuit in the unrolled matter hyperplanes.  We can thus cut the set $\ft_{1,F,\Z}$ into chambers according to what the set $\Lambda^\R$ is; these are chambers induced by the hyperplane arrangement defined by the circuits of the unrolled matter hyperplanes.  We will use repeatedly that by choosing $p$ sufficiently large, we make sure that any non-empty chamber in $\ft_{1,F,\R}$ contains a point of the form $\phi_{1/p}$ and in fact, any point in $\ft_{1,F,\R}$ can be approximated arbitrarily well by points satisfying this property.  
Combining Lemmata \ref{lem:Lam-same} and \ref{lem:localize-twist}, we see an important compatibility for the twisting functors:
\begin{lemma}\label{lem:just-hyperplanes}
  For $p$ sufficiently large, if no hyperplane $H_\al$ separates both $\phi$ and $\phi''$ from $\phi'$, then ${}_{\phi''}T_{\phi'}\Lotimes_{\EuScript{A}_{\phi'}} {}_{\phi'}T_{\phi}\cong {}_{\phi''}T_{\phi}$.
\end{lemma}
\begin{proof}
We induct on the number $m$ of hyperplanes separating $\phi$ and $\phi''$.  If $m=1$, then this is trivial by Lemma  \ref{lem:Lam-same}, since $\phi'$ must be in the same chamber as one of the endpoints.  Let $\phi_1$ be a point in the first chamber that the line segment joining $\phi $ to $\phi'$ passes through.  Given that $p$ is sufficiently large, we can assume that there is a point in this chamber such that $\phi'-\phi_1$ and $\phi_1-\phi$ lie in the same GIT chamber, so we have ${}_{\phi'}T_{\phi_1}\Lotimes_{\EuScript{A}_{\phi_1}} {}_{\phi_1}T_{\phi}\cong {}_{\phi'}T_{\phi}$.  By induction, ${}_{\phi}T_{\phi'}\Lotimes_{\EuScript{A}_{\phi'}} {}_{\phi'}T_{\phi_1}\cong {}_{\phi}T_{\phi_1}$.  Thus, it suffices to prove that ${}_{\phi''}T_{\phi_1}\Lotimes_{\EuScript{A}_{\phi_1}} {}_{\phi_1}T_{\phi}\cong {}_{\phi''}T_{\phi}$.  By replacing $\phi$ by another point in its chamber (again, we use that $p$ is sufficiently large), we can assume that the straight line from $\phi$ to $\phi''$ passes through the chamber of $\phi_{1}$. This completes the proof.
\end{proof}

As usual, we'll want to think of this action in a way such that $p$ becomes large and then can be forgotten.  Thus, we will want to take as our basic parameter $\psi=\phi_{1/p}\in \ft_{1,F,\R}$ which we can continuously vary.  Note that the bad locus in $\ft_{1,F,\R}$ where the set $\Lambda^{\mathbb{R}}$ changes is closed under the action of the affine Weyl group $\widehat{W}_F$.  We let $\mathring{\ft}_{1,F}$ denote the complement of the complexifications of these hyperplanes in $\ft_{1,F}=\ft_{1,F,\C}$, and $\mathring{T}_{1,F}$ the image of this locus under the isomorphism $T_{1,F}\cong \ft_{1,F}/\ft_{F,\Z}$.
\notation{$\mathring{T}_{1,F}$}{The complement in $T_{1,F}$ of the subtori that correspond to changes of the set $\barLambda$.}

Consider the fundamental group $\pi=\pi_1(\mathring{T}_{1,F}/W_F,\psi)=\pi_1(\mathring{\ft}_{1,F}/\widehat{W}_F,\psi)$. For each fixed $p$, we can consider the subgroupoid $\pi^{(p)}$ of the fundamental group with objects $\psi=\phi_{1/p}$ given by generic $\phi\in \ft_{1,F,\Z}$ (that is, the values of $\phi$ where derived equivalence holds).  

It is a fact that seems to well-known to experts, though the author has not found any particularly satisfactory reference (this is stated as a conjecture in \cite[\S 3.2.8]{OkGRT}), that:
\begin{proposition}\label{prop:pi-action}
For $p$ sufficiently large,  the functors $\Phi_w^{\phi',\phi''}$  define an action of the groupoid $\pi^{(p)}$ that induces an action of $\pi$ on $D^b(\EuScript{A}_{\phi}\mmod)$.
\end{proposition}
This should not be a special fact about Coulomb branches, but is expected to be a general fact about symplectic resolutions.  A version of it is proven in \cite{BezRiche} for the case of the Springer resolution and in the case of a Higgs branch by Halpern-Leistener and Sam in \cite{HLSdeq}.

\subsection{Schobers}
\label{sec:schobers}

We'll give a proof of Proposition \ref{prop:pi-action} below, and in fact, show that this action is part of a more complicated structure: a {\it perverse Schober}, a notion proposed by Kapranov and Schechtman \cite{kapranovPerverseSchobers2015}.  Perverse schobers are not, in fact, a structure which has been defined in full generality, but for the complement of a subtorus arrangement, they can be defined using the presentation of the perverse sheaves on a complex vector space stratified by a complexified hyperplane arrangement given by the same authors in \cite{KShyperplane}.
\notation{$\gamma_{CC'},\delta_{C'C},\phi_{CC''}$}{The generalization/specialization/transition functors of a Schober structure.}
\begin{definition}
  Let $Z$ be a finite-dimensional $\R$-affine space, and let $\{H_\gamma\}$ for $\gamma$ running over a (possible infinite) index set be a locally finite hyperplane arrangement.  Let $\nabla$ be the poset of faces of this arrangement.  A {\bf perverse Schober} on the space $Z\otimes_{\R}\C$ stratified by the intersections of the hyperplanes $\{H_\gamma\} $ is an assignment of a dg-category $\mathcal{E}_C$ for each $C\in \nabla$, and   to every pair of faces where $C'\leq C$, an assignment of 
  {\bf generalization functors} $\gamma_{CC'}:\mathcal{E}_{C'}\to \mathcal{E}_{C}$ and their left adjoints, the  {\bf specialization functors} $\delta_{C'C}\colon \mathcal{E}_{C}\to \mathcal{E}_{C'}$.  These combine to give {\bf transition functors} $\phi_{CC''}=\gamma_{CC'}\delta_{C'C''}$ whenever $\bar{C}\cap \bar{C''}\neq \emptyset$, and $C'$ is the unique open face in this intersection.
  \begin{enumerate}
  \item We have isomorphisms of functors $\gamma_{CC'}\gamma_{C'C''}\cong \gamma_{CC''}$  for a triple $C''\leq C'\leq C$ with the usual associativity for a quadruple.  
  \item If $C'\leq C$, the unit of the adjoint pair $(\delta_{CC'},\gamma_{C'C})$ is an isomorphism of  $\gamma_{C'C}\delta_{CC'}$ to the identity.  This gives a canonical isomorphism between $\phi_{CC''}$ and $\gamma_{CC'}\delta_{C'C''}$ for $C'$ any face in the intersection $\bar{C}\cap \bar{C}''$.  
  \item If $(C, C', C'')$ is colinear, then we have isomorphisms  $\phi_{CC'}\phi_{C'C''}\cong\phi_{CC''}$ again with associativity for a colinear quadruple $(C_1,C_2,C_3,C_4)$.  This means we can define the functor $\phi_{CC''}$ for any pair of faces $(C,C')$ by taking a generic line segment between these faces, and composing the functors $\phi_{CC_1}\phi_{C_1C_2}\cdots \phi_{C_nC'}$ for $C_1,\dots, C_n$ the full list of faces this line passes through.  
  \item If $C$ and $C'$ have the same dimension, span the same subspace, and are adjacent across a face with codimension 1 in $C$ and $C'$, then $\phi_{CC'}$ is an equivalence. 
  \end{enumerate}
\end{definition}
\begin{remark}
  For reasons of convenience here, we have departed a little from the framework of Kapranov and Schechtman.  It would be more consistent with their definition of a Schober on a disk \cite{kapranovPerverseSchobers2015}, to assume that the equivalence $\phi_{CC'}$ will be the twist equivalence of a spherical functor, while it is more convenient for us to present it as the cotwist, as Lemma \ref{lem:equiv-sphere} shows, and the definition of a spherical functor is not totally symmetric. This seems to be a general feature of equivalences arising from Morita contexts.  
\end{remark}

A Schober on a complex torus $T$ that is smooth on the faces of a subtorus arrangement is just a Schober on the preimage in the universal cover $\ft$, together with an action of $\pi_1(T)$
compatible with all the data above.  

We'll be interested in the case where $Z=\ft_{1,F,\R}$ and $H_\al$ the hyperplanes defined by the circuits in unrolled matter hyperplanes. Thus, the faces are the sets on which $\Lambda$ is constant.  
This collection of hyperplanes is invariant under the action of $\widehat{W}_F$.   We can therefore define a Schober on the quotient $\mathring{T}_{1,F}/W_F\cong \mathring{\ft}_{1,F}/\widehat{W}_F$ by defining a $\widehat{W}_F$-equivariant Schober on $\mathring{\ft}_{1,F}$, which we will do below.

This might concern some readers, since there are infinitely many hyperplanes in this arrangement, and thus infinitely many Schober relations to check.  However, under the action of the affine Weyl group $\widehat{W}_F$, there are only finitely many orbits of faces, hyperplanes, etc., and thus finitely many Schober relations to check, once we have proven the obvious commutations with elements of the affine Weyl group.  In particular, in the following section, we will give a proof where checking each Schober relation might require enlarging the prime $p$.  Since we will only need to this once for each orbit of $\widehat{W}_F$, we can safely enlarge $p$ as much as necessary at each step of the proof, and still have a finite $p$ at the end.  

\subsection{The Schober of quantized modules}
\label{sec:schob-quant-modul}

There are two natural ways to define a Schober based on a Coulomb branch.  Let us first describe the quantum route, based on the representation theory of the algebras $\efA$ and the wall-crossing functors of Section \ref{sec:wall-cross-funct}. Accordingly, this Schober is only defined over a positive base field.  Now, choose a disjoint collection of open subsets $U_C\subset Z$ for each face $C$, contained in the star of this face, and having non-trivial intersection with each face in this star.  Let $\mathsf{u}_C$ be the set of points $\phi\in \ft_{1,F;\Z}$ such that derived localization holds at $\phi$ and we have that  $\phi_{1/p}\in U_C$.   If $\mathsf{u}_{C}=\{\phi_1,\dots, \phi_k\}$, then we let $\EuScript{A}_{C}$ be the matrix algebra where the entry $(i,j)$ is an element of ${}_{\phi_i}T_{\phi_j}$, that is,  
\newseq
\begin{equation*}\label{def:AC}
\subeqn
    \EuScript{A}_{C}=\begin{bmatrix}     \EuScript{A}_{\phi_1} & {}_{\phi_1}T_{\phi_2}& \cdots & {}_{\phi_1}T_{\phi_k}\\          {}_{\phi_2}T_{\phi_1}&\EuScript{A}_{\phi_2} & \cdots & {}_{\phi_2}T_{\phi_k}\\      \vdots & \vdots &\ddots & \vdots\\       {}_{\phi_k}T_{\phi_1} & {}_{\phi_k}T_{\phi_2} & \cdots &\EuScript{A}_{\phi_k}    \end{bmatrix}  
\end{equation*}
\notation{$\EuScript{A}_{C},T_{C,C'}$}{The matrix algebras with entries in quantizations in characteristic $p$ and bimodules defined by \Cref{def:AC,def:TCC}.}
with obvious multiplication.  Any pair $C$ and $C'$ has a similarly defined bimodule where $\mathsf{u}_{C'}=\{\psi_1,\dots, \psi_h\}$ given by
\begin{equation*}\label{def:TCC}
\subeqn
  T_{C,C'}=\begin{bmatrix}
    {}_{\phi_1}T_{\psi_k} & {}_{\phi_1}T_{\psi_2}& \cdots & {}_{\phi_1}T_{\psi_k}\\    
     {}_{\phi_2}T_{\psi_1}& {}_{\phi_2}T_{\psi_2} & \cdots & {}_{\phi_2}T_{\psi_k}\\
     \vdots & \vdots &\ddots & \vdots\\
      {}_{\phi_h}T_{\psi_1} & {}_{\phi_h}T_{\psi_2} & \cdots & {}_{\phi_h}T_{\psi_k}
   \end{bmatrix}
 \end{equation*}
 Of course, we can define this bimodule $T_{\mathsf{u},\mathsf{v}}$ for any pair $\mathsf{u},\mathsf{v}\subset \ft_{1,F,\Z}$; if $\mathsf{u}$ or $\mathsf{v}$ is a singleton, then we omit brackets and just write the single element.   It's easy to check using Lemma \ref{lem:Lam-same} that:
 \begin{lemma}
   If we replace $U_C, U_{C'}$ by open sets $U'_C, U_{C'}'$ satisfying the same conditions, then the resulting algebras $\EuScript{A}_{C}$ and $\EuScript{A}_{C}'$ are Morita equivalent via the bimodules $ T_{\mathsf{u}_C,\mathsf{u}'_C}$ and $ T_{\mathsf{u}'_C,\mathsf{u}_C}$, with this Morita equivalence preserving the bimodules $ T_{C,{C'}}'\cong  T_{\mathsf{u}'_{C'},\mathsf{u}_{C'}}\Lotimes_{\EuScript{A}_{{C'}}}  T_{C,{C'}}\Lotimes_{\EuScript{A}_{C}}T_{\mathsf{u}_A,\mathsf{u}'_C}$
 \end{lemma}
 Thus the category $\mathcal{E}_C^{(p)}\cong D^-(\EuScript{A}_{C}\mmod)$ is independent of the choice of $U_C$, and only depends on $C$.

 \begin{theorem}\label{thm:p-Schober}
   The assignment $\mathcal{E}_C^{(p)}\cong D^-(\EuScript{A}_{C}\mmod)$ for all $C\in \nabla$ and $\phi_{CC'}=T_{C,{C'}}\Lotimes_{\EuScript{A}_{C'}}-$ defines a Schober on $\ft_{1,F,\R}$ which is equivariant for the action $\widehat{W}_F$.  
 \end{theorem}
 \begin{proof}   
   First, we note that if $C'\leq C$, then the star of $C$ lies in the star of $C'$, so for any element of $\mathsf{u}_C$, there is an element of $\mathsf{u}_{C'}$ Morita equivalent by the twisting bimodule.  Thus, $\EuScript{A}_{C'}$ is Morita equivalent to the algebra obtained by taking the union of the sets $\mathsf{u}_C\cup \mathsf{u}_{C'}$.  Now, let us check the conditions of a Schober in turn:
   \begin{enumerate}[wide]
   \item As discussed, if $C'' \leq C'\leq C$, then $\EuScript{A}_{C''}$ is Morita equivalent to the set obtained from the union $
     \mathsf{u}_C\cup \mathsf{u}_{C'}\cup \mathsf{u}_{C''}$.  Thus, we need only prove the corresponding transitivity for any decomposition of $1$ in a ring as the sum of 3 orthogonal idempotents $e+e'+e''$, in which case it is clear.
 \item Using the union $\mathsf{u}_C\cup \mathsf{u}_{C'}$ again, this is just the fact that for any idempotent, we have $e(Ae\otimes_{eA}M)=M$, giving the required isomorphism of  $\gamma_{C'C}\delta_{CC'}$ to the identity.
\item By assumption, if $(C,C',C'')$ are colinear, then we can assume that the line joining them is generic in  the span of these faces.   Let $\EuScript{H}_0$ be the set (possibly empty) of hyperplanes that contain all three faces, $\EuScript{H}_1$ the set of hyperplanes separating $C$ and ${C'}$, and $\EuScript{H}_2$ the set separating ${C'}$ and $C$.

  Choose a point in $\phi\in\mathsf{u}_C$.  We have a functor ${}_{C}T_{\phi}\Lotimes-\colon D^b(\EuScript{A}_\phi\mmod) \to \mathcal{E}_C^{(p)}$ given by the tensor products with ${}_{\phi'}T_\phi$ for all $\phi'\in\mathsf{u}_C$.  Now consider the derived tensor product with ${}_{{C'}}T_C$.  Since the image of $\EuScript{A}_\phi$ is projective, the composition is the functor of tensor product ${}_{\psi'}T_\phi$ for all $\psi'\in \mathsf{u}_{C'}$, that is, tensor product with ${}_{C'}T_{\phi}$.
  For any point $\psi'\in\mathsf{u}_{C'}$, we can find a point in the same chamber such that the straight line to  $\phi$ passes through any hyperplanes in $\EuScript{H}_0$ that separating $\psi$ and $\phi$ before crossing any hyperplanes in $\EuScript{H}_1$.  We can choose $\psi\in \mathsf{u}_{C'}$ on the same side as $\phi$  of all hyperplanes in $\EuScript{H}_0$, so ${}_{\psi'}T_{\psi}\Lotimes_{\EuScript{A}_\psi}{}_{\psi}T_\phi\cong {}_{\psi'}T_{\phi}$ by Lemma \ref{lem:just-hyperplanes}.  That is, we have
  \[{}_{{C'}}T_C\Lotimes_{\EuScript{A}_C}{}_{C}T_{\phi}\cong {}_{{C'}}T_\psi\Lotimes_{\EuScript{A}_\psi}{}_{\psi}T_\phi.\]  Applying this result a second time with $\chi$ an element of $\mathsf{u}_C$ on the same side of all hyperplanes in  $\EuScript{H}_0$ as $\phi$ and $\psi$, we have
  \[{}_{C}T_{C'}\Lotimes_{\EuScript{A}_{C'}} {}_{{C'}}T_{C''}\Lotimes_{\EuScript{A}_{C''}}{}_{C''}T_{\phi}\cong {}_{C}T_\phi\cong {}_{C}T_{C''}\Lotimes_{\EuScript{A}_{C''}}{}_{C''}T_{\phi}.\]  Since the projective modules ${}_{C''}T_{\phi}$ for all $\phi$ are generators for $\EuScript{A}_{C''}\mmod$, this establishes that $\phi_{CC'}\phi_{C'C''}=\phi_{CC''}$.  Furthermore, since these isomorphisms are induced by the natural tensor product maps, they are appropriately associative.  
\item Now, assume that $C$ and ${C'}$ are both $d$-dimensional, and differ across a face of codimension 1.  As before, let $\EuScript{H}_0$ be the hyperplanes that contain both of these faces.  Note that for each $\phi\in \mathsf{u}_C$, there is a unique chamber that intersects $\mathsf{u}_C$ separated from $\phi$ by all hyperplanes in $\EuScript{H}_0$ and no others.  Let $\phi'$ lie in this face.  Then, we have that ${}_{\phi''}T_\phi$ can also be written as $\RHom_{\EuScript{A}_{\phi'}}({}_{\phi'}T_{\phi''}, {}_{\phi'}T_\phi)$ for all $\phi''\in \mathsf{u}_C$, using Lemma \ref{lem:just-hyperplanes} to show that $ {}_{\phi'}T_\phi\cong  {}_{\phi'}T_{\phi''}\Lotimes_{\EuScript{A}_{\phi''}} {}_{\phi''}T_{\phi'}$ and the fact that the inverse of a derived equivalence is its adjoint.

  Now let $\psi,\psi'\in \mathsf{u}_{C'}$ be elements not separated from $\phi,\phi'$ respectively by any hyperplane in $\EuScript{H}_0$.  Applying the argument above and Lemma \ref{lem:just-hyperplanes}  again, we see that
  \[{}_{{C'}}T_C\Lotimes_{\EuScript{A}_C}{}_{C}T_{\phi}= \RHom_{\EuScript{A}_{\psi'}}({}_{\psi'}T_{{C'}},  {}_{\psi'}T_\phi).\]  The adjoint version of Lemma  \ref{lem:just-hyperplanes} then implies that
  \[  \RHom_{\EuScript{A}_{C'}}({}_{{C'}}T_C , {}_{{C'}}T_C\Lotimes_{\EuScript{A}_C}{}_{C}T_{\phi})=\RHom_{\EuScript{A}_{\psi'}}({}_{\psi'}T_{C},  {}_{\psi'}T_\phi)={}_{C}T_{\phi}.\]  Again, since the projectives ${}_{C}T_{\phi}$ generate, the functors $\phi_{{C'}C}$ are thus an equivalence of derived categories.  \qedhere
 \end{enumerate} 
 \end{proof}

Note that Losev shows that when $C$ and $C''$ are top-dimensional faces and $(C,{C'},C'')$ are colinear with ${C'}\subset \bar{C}\cap \bar{C}''$ , then $\phi_{CC''}$ is not just any equivalence of categories, but a partial Ringel duality functor in an appropriate sense (or rather, the degrading of one) and a perverse equivalence \cite[Thm. 9.10]{losevModularCategories2021}.  It would be interesting to consider whether this is true in the case where $C$ and $C''$ are lower-dimensional faces with the same span.    

\subsection{The coherent Schober}
\label{sec:coherent-schober}

Of course, it is a bit inelegant to consider this Schober in the case where $\K=\Fp$ for some large $p$; it would be preferable to send $p\to \infty $ and replace the algebra $\efA$  quantizing $\Fp[\fM]$ with the noncommutative resolution $A$.    

In order to do this, we must feed every object that appeared in the quantum Schober through the woodchipper of Theorem \ref{thm:pStein-equiv}, which allowed us to construct $A$ in the first place.  
Applying this result to the bimodule ${}_{\phi'}T_\phi$, we send the wall-crossing functor to tensor product with a bimodule $ {}_{\phi'_{1/p}}\mathsf{\hat{T}}_{\phi_{1/p}}$ over the categories $ {\hat{\mathsf B}_{\phi'_{1/p}}}$ and $\hat{\mathsf B}_{\phi_{1/p}}$.  Applying Theorem \ref{thm:pStein-equiv} again, but now to the gauge group $\To$, we can describe the resulting bimodule as the completion of $ {}_{\phi'_{1/p}}\mathsf{T}_{\phi_{1/p}}$, the bimodule given by the Hom spaces in the quotient $\bar{\mathsf B} = {{\mathsf B}^Q} /( {\ft_F})$ of the category for $\To$ with  the Lie algebra $ {\ft_F}$ set to 0 in the morphism spaces (which is well-defined since $h=0$ in the $p$th root conventions).

We can easily extend the presentation of Theorem \ref{thm:BFN-pres} to $\bar{\mathsf B}$; essentially the only change needed is that we expand the set of objects to include all of $\ft_{\To}$.  In particular, as in Section \ref{sec:cons-repr-theory}, we can replace the object set with just the elements of $\bar{\Lambda}_Q$ and form a category $\mathsf B^{\Lambda_Q} $ by choosing a representative $\eta_{\Ba}$ for each chamber.  Note that the set $\bar{\Lambda}_Q$ contains the sets $\bar{\Lambda}$ and $\bar{\Lambda}'$ corresponding to the flavors $\phi$ and $\phi'$, and the corresponding full subcategories are exactly $\mathsf B^{\bar \Lambda}$ and $\mathsf B^{\bar \Lambda'}$ as defined in Section \ref{sec:cons-repr-theory}.  Composing the equivalence of Theorem \ref{thm:pStein-equiv} with the equivalences
\begin{equation}
{\mathsf B}_{\phi_{1/p}}\cong \mathsf B^{\Lambda}\qquad  {\mathsf B}_{\phi'_{1/p}}\cong \mathsf B^{\Lambda'}\qquad \bar{\mathsf B} \cong \mathsf B^{\Lambda_Q},\label{eq:B-Lam}
\end{equation}
we have that:
\begin{lemma}
  The bimodule ${}_{\phi'}T_\phi$ matches with the completion of the bimodule ${}_{\Lambda'}\mathsf{T}_\Lambda$ that sends $\Ba\in \Lambda,\Bb\in \Lambda'$ to
  \[(\Bb,\Ba) \mapsto \Hom_{\bar{\mathsf B}}(\eta_{\Ba},\eta_{\Bb}).\]
\end{lemma}
This latter bimodule is independent of $p$, and thus can be defined over any field, in particular over $\Q$.  
 
This has a very simple consequence for the structure of the category of modules over the algebra $\EuScript{A}_{C}$.  Consider the set $\bar{\Lambda}_C=\cup_{\phi\in \mathsf{u}_C}\bar{\Lambda}^\R_\phi$; this set is independent of $p$, and describes all the chambers that appear in the preimage of the star of $C$ under the map $\ft_{\To}\to \ft_{F}$.  As with all objects here, we therefore obtain the result that:
\begin{proposition}
  The category $\EuScript{A}_{C}\mmod_{\upsilon'}$ for $\K$ a field of large positive characteristic $p$ is equivalent to the category of modules over the completion of $\mathsf {B}^{
   \bar{\Lambda}_C}(\K)$, the subcategory of the completion $\bar{\mathsf B}(\K) $ with object set $\eta_{\Ba}$ for $\Ba\in \Lambda_C$.  
\end{proposition}
As we exploited earlier, the latter category is well-defined over any base ring, in particular over $\Q$.  By analogy with the noncommutative resolution $A$, we let:
\begin{equation}\label{def:AC2}
A_C=\bigoplus_{\bar{\Ba},\bar{\Bb}\in \bar\Lambda_C}\Hom_{ {\mathsf{B}^{\bar{\Lambda}_C}(\Q)}}(\bar{\Ba},\bar{\Bb}).    
\end{equation}

This ring has a presentation directly analogous to that of $A$ given in \cref{sec:presentations}.  We need only adjust Proposition \ref{prop:presentation} by changing the vertex set to be $\Lambda_{C}$.  
\notation{$A_C,\mathcal{E}^{\Q}_C$}{The sum of morphism spaces in $\mathsf{B}^{\bar{\Lambda}_C}(\Q)$ defined in \eqref{def:AC2} and its derived category $\mathcal{E}^{\Q}_C=D^-(A_C\mmod)$.}
\notation{$\phi_{CC'}^{\Q}$}{The transition derived equivalences of the Schober of \Cref{th:coherent-Schober}.}
\begin{definition}
  Let $\mathcal{E}^{\Q}_C=D^-(A_C\mmod)$, and $\phi_{CC'}^{\Q}$ be derived tensor product with the bimodule ${}_{\Lambda_{C}}\mathsf{T}_{\Lambda_{C'}}$.  
\end{definition}
Note that if $C$ is maximal dimensional, then $A_C$ is a noncommutative crepant resolution of $\fM$ by Corollary \ref{cor:A-nccr},  so $D^-(A_C\mmod)\cong D^-(\Coh(\fM))$.  Unfortunately we know no such convenient geometric interpretation of the other categories that appear for smaller strata.  
\begin{theorem}\label{th:coherent-Schober}
  The assignment $\mathcal{E}^{\Q}_C$ and $\phi_{CC'}^{\Q}$ above defines a $\widehat{W}_F$-equivariant Schober.
\end{theorem}
\begin{proof}
  The required isomorphisms are all induced by composition of maps, so in order to show that the Schober relations hold, it is enough to check that we have the Schober relations mod infinitely many primes $p$.  This is clear from comparison with the Schober $\mathcal{E}^{(p)}$ of Theorem \ref{thm:p-Schober} via the functor of Lemma \ref{lem:Gamma-iso}.  
\end{proof}
Note the similarity of this action with that defined using the ``magic windows'' approach of \cite{HLSdeq}.  It would be quite interesting to understand how these approaches compare when the same symplectic singularity can be written as both a Higgs and Coulomb branch.

For a fixed basepoint, we can choose a D-equivalence between the nccr $A_C$ for a maximal dimensional face and the commutative resolution $\tfM$.  This shows that:
\begin{corollary}\label{cor:pi1-action}
  The functors $\phi_{CC'}$ define an action of $\pi$ on $D^b(\Coh(\tM_\Q))$.   
\end{corollary}

A long-standing conjecture of Bezrukavnikov and Okounkov connects these actions to enumerative geometry, as discussed in \cite[\S 3.2]{OkGRT}:
\begin{conjecture}\label{conj:BO}
The action of $\pi$ on $\Coh(\tM _\Q)$ categorifies the monodromy of the quantum connection.
\end{conjecture}
A positive resolution to this conjecture has been announced for hyperk\"ahler reductions by Bezrukavnikov and Okounkov, but as of the current moment, the proof has not appeared.  
Of course, it would be quite interesting to understand whether the Schober discussed above contains deeper information about the quantum D-module.

 \bigskip
\IndexOfNotation

{\renewcommand{\markboth}[2]{}\printbibliography}
\end{document}